\documentclass[11pt]{amsart}
\usepackage{pb-diagram,amssymb,epic,eepic,verbatim,graphicx}
\usepackage{graphics,epsfig,psfrag}

\newtheorem{theorem}{Theorem}[section]
\newtheorem{corollary}[theorem]{Corollary}

\newtheorem{proposition}[theorem]{Proposition}
\newtheorem{lemma}[theorem]{Lemma}
\newtheorem{lem}[theorem]{}
\theoremstyle{definition}
\newtheorem{definition}[theorem]{Definition}
\theoremstyle{remark}
\newtheorem{remark}[theorem]{Remark}
\newtheorem{example}[theorem]{Example}
\newtheorem{convention}[theorem]{Convention}
\newcommand{\blem}{\begin{lem} \rm}
\newcommand{\elem}{\end{lem}}
\newcommand\E{\mathcal{E}}

\newcommand\cA{\mathcal{A}}

\renewcommand{\L}{\mathcal{L}}

\newcommand{\F}{\mathcal{F}}
\newcommand{\N}{\mathbb{N}}
\newcommand{\R}{\mathbb{R}}

\newcommand{\cC}{\mathcal{C}}

\newcommand{\Z}{\mathbb{Z}}

\newcommand{\on}{\operatorname}
\newcommand{\ainfty}{{$A_\infty$\ }}
\newcommand{\Don}{\on{Don}}

\newcommand{\Horm}{\text{H\"orm}}

\newcommand{\Floer}{\on{Floer}}
\newcommand{\Brane}{\on{Brane}}
\newcommand{\FIO}{\on{FIO}}
\newcommand{\Cat}{\on{Cat}}

\newcommand{\dual}{\vee}
\newcommand{\Ab}{\on{Ab}}
\newcommand{\Ham}{\on{Ham}}

\newcommand{\Fun}{\on{Fun}}
\newcommand{\Obj}{\on{Obj}}

\newcommand{\graph}{\on{graph}}
\newcommand{\Symp}{\on{Symp}}

\newcommand{\Lag}{\on{Lag}}

\newcommand{\Hom}{ \on{Hom}}

\newcommand\dirac{/\kern-1.2ex\partial} 
\newcommand\qu{/\kern-.7ex/} 
\newcommand\lqu{\backslash \kern-.7ex \backslash} 

\newcommand\dr{r_+ \kern-.7ex - \kern-.7ex r_-}
 
\newcommand{\f}{\frac}
\newcommand{\lan}{\langle}
\newcommand{\ran}{\rangle}
\newcommand{\hh}{{\f{1}{2}}}

\newcommand\pt{\on{pt}}

\newcommand\cL{\mathcal{L}}

\newcommand\cF{\mathcal{F}}

\newcommand\ul{\underline}

\newcommand\bra[1]{ \lan {#1} \ran} 
\newcommand\bdefn{\begin{definition}}
\newcommand\edefn{\end{definition}}
\newcommand\bea{\begin{eqnarray*}}
\newcommand\eea{\end{eqnarray*}}
\newcommand\bcv{\left[ \begin{array}{r} }
\newcommand\ecv{\end{array} \right] }

\newcommand\bma{\left[ \begin{array} }
\newcommand\ema{\end{array} \right]}
\newcommand\ben{\begin{enumerate}}
\newcommand\een{\end{enumerate}}
\newcommand\beq{\begin{equation}}
\newcommand\eeq{\end{equation}}
\newcommand\bex{\begin{example}}
\newcommand\bsj{\left\{ \begin{array}{rrr} }
\newcommand\esj{\end{array} \right\}}

\newcommand\Id{\on{Id}}
\newcommand\cI{\mathcal{I}}

\newcommand\eex{\end{example}}

\newcommand\sx{*\kern-.5ex_X}

\def\mathunderaccent#1{\let\theaccent#1\mathpalette\putaccentunder}
\def\putaccentunder#1#2{\oalign{$#1#2$\crcr\hidewidth \vbox
to.2ex{\hbox{$#1\theaccent{}$}\vss}\hidewidth}}

\newcommand\Fuk{{\on{Fuk}}}

\begin{document}

\title{Functoriality for Lagrangian correspondences in Floer theory}

\author{Katrin Wehrheim and Chris T. Woodward}

\address{Department of Mathematics,
Massachusetts Institute of Technology,
Cambridge, MA 02139.
{\em E-mail address: katrin@math.mit.edu}}

\address{Department of Mathematics, 
Rutgers University,
Piscataway, NJ 08854.
{\em E-mail address: ctw@math.rutgers.edu}}

\begin{abstract}  
We associate to every monotone Lagrangian correspondence a functor between Donaldson-Fukaya categories.
The composition of such functors agrees with the functor associated to the geometric composition of the correspondences, if the latter is embedded.  That is ``categorification
commutes with composition'' for Lagrangian correspondences.
This construction fits into a symplectic $2$-category with a categorification $2$-functor, in which all correspondences are composable, and embedded geometric composition is isomorphic to the actual composition. As a consequence,
any functor from a bordism category to the symplectic category gives rise to a category valued topological field theory.

\vspace{-8mm}
\end{abstract} 

\maketitle

\tableofcontents\vspace{-10mm}

\section{Introduction}  

Correspondences arise naturally as generalizations of maps in a number
of different settings: A correspondence between two sets is a subset
of the Cartesian product of the sets -- just like the graph of a map.
In symplectic geometry, the natural class is that of Lagrangian
correspondences, that is, Lagrangian submanifolds in the product of
two symplectic manifolds (with the symplectic form on the first factor
reversed).  Lagrangian correspondences appear in H\"ormander's
generalizations of pseudodifferential operators \cite{ho:fio1}, and
were investigated from the categorical point of view by Weinstein
\cite{we:sc}.  In gauge theory Lagrangian correspondences arise as
moduli spaces of bundles associated to cobordisms \cite{wi:jo}.

One would hope that various constructions associated to symplectic
manifolds, which are compatible with symplectomorphisms, can also be made functorial for Lagrangian correspondences.  The
constructions considered by H\"ormander and Weinstein correspond to
various notions of quantization, by which a symplectic manifold
is replaced by a linear space; one then tries to attach to a
Lagrangian correspondence a linear map.  More recently, categorical invariants associated to a symplectic manifold have been introduced by Donaldson and Fukaya, see for example \cite{fooo} and \cite{se:bo}.  
To the symplectic manifold is associated a category
whose objects are certain Lagrangian submanifolds, and whose morphisms
are certain chain complexes or Floer cohomology groups.  The
composition in this category gives a way to understand various
product structures in Floer theory, and plays a role in the
homological mirror symmetry conjecture of Kontsevich \cite{kon:hom}.

In this paper we associate to every (compact monotone or geometrically bounded exact)
symplectic manifold $(M,\omega)$ a category $\Don^\#(M)$, which is a slight enlargement of the usual Donaldson-Fukaya category.
Its objects are certain sequences of (compact, oriented, relatively spin, monotone or exact) Lagrangian correspondences and its morphisms are quilted Floer cohomology classes, as introduced in \cite{quiltfloer}.
Given two symplectic manifolds $M_0$ and $M_1$ of the same monotonicity type and an admissible Lagrangian correspondence $L_{01} \subset M_0^- \times M_1$ we construct a functor 
$$ 
\Phi(L_{01}) : \Don^\#(M_0) \to \Don^\#(M_1) .
$$
On objects it is given by concatenation, e.g.\ $\Phi(L_{01})(L_0) = (L_0,L_{01})$ for a Lagrangian submanifold $L_0\subset M_0$. 
On morphisms the functor is given by a relative Floer theoretic invariant constructed from moduli spaces of pseudoholomorphic quilts introduced in \cite{quilts}.

Given a triple $M_0$, $M_1$, $M_2$
of symplectic manifolds and admissible Lagrangian correspondences 
$L_{01}\subset M_0^-\times M_1$ and $L_{12}\subset M_1^-\times M_2$, 
the algebraic composition $\Phi(L_{01})\circ\Phi(L_{12}): \Don^\#(M_0) \to \Don^\#(M_2)$ 
is always defined.  On the other hand, one may consider
the {\em geometric composition} introduced by Weinstein \cite{we:sc}
$$
L_{01} \circ L_{12} := \pi_{02}( L_{01} \times_{M_1} L_{12}) \subset M_0^- \times M_2 ,
$$ 
given by the image under the projection $ \pi_{02}: M_0^- \times M_1 \times M_1^- \times M_2 \to M_0^- \times M_2 $ of
\begin{equation} \label{int}
L_{12} \times_{M_1} L_{01} := (L_{01} \times L_{12})
 \cap (M_0^- \times \Delta_1 \times M_2) .
\end{equation}
If we assume transversality of the intersection then the restriction of $\pi_{02}$ to $L_{01} \times_{M_1} L_{12}$ is automatically an immersion, see \cite{gu:rev,quiltfloer}. 
Using the strip-shrinking analysis from \cite{isom} we prove that if $L_{01} \times_{M_1} L_{12}$ is a transverse intersection and embeds by $\pi_{02}$ into $M_0^- \times M_2$ then
\begin{equation} \label{maineq}
 \Phi(L_{01}) \circ \Phi(L_{12}) \cong \Phi(L_{01} \circ L_{12}) .
\end{equation}
This is the "categorification commutes with composition" result alluded to in the abstract.  
If $M_1$ is not spin, there is also a shift of relative spin structures on the right-hand side.  

There is a stronger version of this result, expressed in the language
of $2$-categories as follows. (See e.g.\ Section~\ref{2cat} for an introduction to this language.) We construct a {\em Weinstein-Floer}
$2$-category $\Floer^\#$ whose objects are symplectic manifolds,
$1$-morphisms are sequences of Lagrangian correspondences, and
$2$-morphisms are Floer cohomology classes. (Again, we impose monotonicity and certain further admissibility assumptions on all objects and $1$-morphisms.) 
The composition of $1$-morphisms in this category is concatenation, which we denote by $\#$.  The construction of the functor $\Phi(L_{01})$ above extends to
a categorification $2$-functor to the $2$-category of categories
\begin{equation}\label{2func}
\Floer^\# \longrightarrow \Cat .
\end{equation}
On objects and elementary $1$-morphisms (i.e.\ sequences consisting of a single correspondence) it is given by associating to every
symplectic manifold $M$ its Donaldson-Fukaya category $\Don^\#(M)$, and to
every Lagrangian correspondence $L_{01}$ the associated functor
$\Phi(L_{01})$. The further $1$-morphisms are concatenations of elementary
Lagrangian correspondences, mapped to the composition of functors.
The 2-morphisms are quilted Floer homology classes, to which we
associate natural transformations.  A refinement of \eqref{maineq}
says that the concatenation $L_{01} \# L_{12}$ is $2$-isomorphic to
the geometric composition $L_{01} \circ L_{12}$ as $1$-morphisms in
$\Floer^\#$.  The formula \eqref{maineq} then follows by combining
this result with the $2$-functor axiom for $1$-morphisms in
\eqref{2func}.

Alternatively, one could identify the $1$-morphisms $L_{01} \# L_{12}$ and $L_{01} \circ L_{12}$ if the latter is a transverse, embedded composition. This provides
an elementary construction of a symplectic category $\Symp^\#$  explained in Section \ref{symp cat}. It consists of symplectic manifolds and equivalence classes of sequences of Lagrangian correspondences, whose composition is always defined and coincides with geometric composition in transverse, embedded cases.

The categorical point of view fits in well with one of the
applications of our results, which is the construction of topological
field theories associated to various gauge theories.  A corollary of
our categorification functor \eqref{2func} is that any functor from a
bordism category to the (monotone subcategory of the)
symplectic category $\Symp^\#$ gives rise to a
category valued TFT.  For example, in \cite{field} we investigate the
topological quantum field theory with corners (roughly speaking; not
all the axioms are satisfied) in $2 +1 + 1$ dimensions arising from
moduli spaces of flat bundles with compact structure group on
punctured surfaces and three-dimensional cobordisms containing
tangles.  In particular, this gives rise to $SU(N)$ Floer theoretic
invariants for $3$-manifolds that should be thought of as Lagrangian
Floer versions of gauge-theoretic invariants investigated by Donaldson
and Floer, in the case without knots, and Kronheimer-Mrowka
\cite{km:kh} and Collin-Steer \cite{co:in}, in the case with knots.
The construction of such theories was suggested by Fukaya in
\cite{fu:fl1} and was one of the motivations for the development of
Fukaya categories.

Many of our results have chain-level versions, that is, extensions to
Fukaya categories.  These will be published in \cite{Ainfty}, which is
joint work with S. Mau.  To each monotone Lagrangian correspondence
with minimal Maslov number at least three we define an \ainfty functor
$\Psi(L_{01}) : \Fuk^\#(M_0) \to \Fuk^\#(M_1) $ between extended
versions of the Fukaya categories.
Moreover, we are working on extending this construction to an $A_\infty$ functor 
$$
\Fuk^\#(M_0,M_1) \longrightarrow \Fun(\Fuk^\#(M_0),\Fuk^\#(M_1)) ,
$$ where the Fukaya category on the left hand side should be a
chain-level version of the morphism space of $\Floer^\#$ between $M_0$
and $M_1$, i.e.\ its objects are Lagrangian correspondences and
sequences thereof, starting at $M_0$ and ending at $M_1$.  On homology
level, for the Donaldson-Fukaya categories, this functor is given as
part of the $2$-categorification functor \eqref{2func}.  On chain
level, it would finalize the proof of homological mirror symmetry for
the four-torus by Abouzaid and Smith \cite{ab:hms4torus}. It should be
seen as the symplectic analogue of the quasi-equivalence of
dg-categories \cite{toen:dg} in algebraic geometry $D^b_\infty(X\times
X) \simeq \Fun(D^b_\infty(X),D^b_\infty(X))$ for (somewhat enhanced)
derived categories of coherent sheaves on a projective variety $X$.
Abouzaid and Smith utilize the conjectural symplectic functor to prove
that a given subcategory $\cA$ (for which a fully faithful functor to
a derived category of coherent sheaves is known) generates the Fukaya
category $\Fuk^\#(T^4)$, by resolving the diagonal $\Delta\subset
(T^4)^-\times T^4$ in terms of products of Lagrangians in $\cA$.

\medskip

{\em We thank Paul Seidel and Ivan Smith for
encouragement and helpful discussions.}

\section{Symplectic category with Lagrangian correspondences}
\label{symp cat}

We begin by summarizing some results and elementary notions from
\cite{quiltfloer}.  
Restricted to linear Lagrangian correspondences between symplectic
vector spaces, the {\em geometric composition} of Lagrangian correspondences 
defined in \eqref{int} is a well defined composition and defines a linear
symplectic category \cite{gu:rev}.  In general, however, even when the
intersection \eqref{int} is transverse, the geometric composition only
yields an immersed Lagrangian. While it may be natural to allow
immersed Lagrangian correspondences (and attempt a definition of Floer
cohomology for these), a construction of a symplectic category based
on geometric composition would require the inclusion of perturbation
data.  A simple resolution of the composition problem is given by
passing to sequences of Lagrangian correspondences and defining a
purely algebraic composition.
Here and throughout we will write $M$ for a symplectic manifold $(M,\omega)$ consisting of a manifold with symplectic $2$-form; and we denote by $M^-:=(M,-\omega)$ the same manifold equipped with the symplectic form $-\omega$.

\begin{definition} \label{Lag cor}
Let $M,M'$ be symplectic manifolds.  A {\em generalized Lagrangian
correspondence} $\ul{L}$ from $M$ to $M'$ consists of
\begin{enumerate}
\item a sequence $N_0,\ldots,N_r$ of any length $r+1\geq 2$ of
symplectic manifolds with $N_0 = M$ and $N_r = M'$ ,
\item a sequence $L_{01},\ldots, L_{(r-1)r}$ of Lagrangian
correspondences with $L_{(j-1)j} \subset N_{j-1}^-\times N_{j}$ for
$j=1,\ldots,r$.
\end{enumerate}
\end{definition}
\begin{definition}
Let $\ul{L}$ from $M$ to $M'$ and $\ul{L}'$ from $M'$ to $M''$ 
be two generalized Lagrangian correspondences. Then we define composition
$$
(\ul{L},\ul{L}') :=
\bigl(L_{01},\ldots,L_{(r-1)r},L'_{01},\ldots,L'_{(r'-1)r'}\bigr) 
$$
as a generalized Lagrangian correspondence from $M$ to $M''$.
\end{definition}

We will however want to include geometric composition into our category -- if it is well defined.
For the purpose of obtaining well defined Floer cohomology we will restrict ourselves to the following class of compositions, for which the resulting Lagrangian correspondence is in fact a smooth submanifold.

\begin{definition} \label{embedded}
We say that the composition $ L_{01} \circ L_{12}$ is {\em embedded}
if the intersection $(L_{01} \times L_{12}) \pitchfork (M_0^- \times \Delta_1 \times M_2)$ is transverse and the projection $\pi_{02}:L_{12} \times_{M_1} L_{01}\to L_{01}\circ L_{12}\subset M_0^-\times M_2$ is injective (and hence automatically an embedding).
\end{definition}

Using these notions we can now define a {\em symplectic category} $\Symp^\#$ which includes all Lagrangian correspondences.  
An extension of this approach, using Floer cohomology spaces to
define a $2$-category, is given in Section \ref{2cat}.

\begin{definition} 
The symplectic category $\Symp^\#$ is defined as follows:
\begin{enumerate}
\item 
The objects of $\Symp^\#$ are smooth symplectic manifolds
$M=(M,\omega)$.
\item 
The morphisms $\Hom(M,M')$ of $\Symp^\#$ are generalized Lagrangian
correspondences from $M$ to $M'$ modulo the equivalence relation
$\sim$ generated by
$$
\bigl(\ldots,L_{(j-1)j},L_{j(j+1)},\ldots \bigr) 
\sim \bigl(\ldots,L_{(j-1)j}\circ L_{j(j+1)},\ldots\bigr)
$$
for all sequences and $j$ such that $L_{(j-1)j} \circ L_{j(j+1)}$ is embedded.
\item 
The composition of morphisms $[\ul{L}]\in\Hom(M,M')$ and
$[\ul{L}']\in\Hom(M',M'')$ is defined by
$$
[\ul{L}]\circ[\ul{L}'] := [(\ul{L},\ul{L}')] \;\in\Hom(M,M'') .
$$
\end{enumerate}
\end{definition}

Note that a sequence of Lagrangian correspondences in $\Hom(M,M')$ can run through any sequence $(N_i)_{i=1,\ldots,r-1}$ of intermediate
symplectic manifolds of any length $r-1\in\N_0$.  Nevertheless, the
composition of two such sequences is always well defined.  In (c) the
new sequence of intermediate symplectic manifolds for
$\ul{L}\circ\ul{L}'$ is
$(N_1,\ldots,N_{r-1},N_r=M'=N_0',N_1',\ldots,N_{r'-1}')$.  This
definition descends to the quotient by the equivalence relation $\sim$
since any equivalences within $\ul{L}$ and $\ul{L}'$ combine to an
equivalence within $\ul{L}\circ\ul{L}'$.  

\begin{remark} \label{small}
\hspace{5mm}\\
\vspace{-5mm}
\begin{enumerate}
\item
The composition in $\Symp^\#$ is evidently associative:
$[\ul{L}]\circ[\ul{L}']\circ[\ul{L}''] = [(\ul{L},\ul{L}',\ul{L}'')]$.
\item
The identity in $\Hom(M,M)$ is the equivalence class $[\Delta_M]$ of
the diagonal $\Delta_M\subset M^-\times M$.
It composes as identity since e.g.\ $L_{(r-1)r}\circ\Delta_{M_r}=L_{(r-1)r}$ is always smooth and embedded.  
\item
In order to make $\Symp^\#$ a small category, one should fix a set of smooth manifolds, 
for example those embedded in Euclidean space.  Any smooth
manifold can be so embedded by Whitney's theorem.
For a fixed manifold, the possible symplectic forms again form a set, hence we have a set of objects.
Given two symplectic manifolds, the finite sequences of objects $(N_i)_{i=1,\ldots,r-1}$ again form a set, and for each fixed sequence the generalized Lagrangian correspondences between them can be exhibited as subsets satisfying submanifold, isotropy, and coisotropy conditions. Finally, we take the quotient by a relation to obtain a set of morphisms.
\end{enumerate}
\end{remark}


\begin{lemma}
\begin{enumerate}
\item If $L_a,L_b \subset M^- \times M'$ are distinct Lagrangian
submanifolds, then the corresponding morphisms $[L_a],[L_b]
\in\Hom(M,M')$ are distinct.
\item The composition of smooth Lagrangian correspondences $L\subset
M^-\times M'$ and $L'\subset M'^-\times M''$ coincides with the
geometric composition, $[L]\circ[L']=[L\circ L']$ if $L\circ L'$ is
embedded.
\end{enumerate}
\end{lemma}

\begin{proof}
To see that $L_a\neq L_b\subset M^-\times M'$ define distinct morphisms
note that the projection to the (possibly singular) Lagrangian
$\pi([\ul{L}]):=L_{01}\circ\ldots\circ L_{(r-1)r} \subset M^-\times
M'$ is well defined for all $[\ul{L}]\in\Hom(M,M')$.  The rest follows
directly from the definitions.
\end{proof}

\begin{remark} \label{rmk fio1}
Lagrangian correspondences appeared in the study of Fourier integral
operators by H\"ormander and others.  Any immersed homogeneous\footnote{ An immersed
  Lagrangian correspondence $L_{01}$ is called homogeneous if its image lies
  in the complement of the zero sections,
  $L_{01}\to(T^*Q_0^-\setminus 0_{Q_0})\times (T^*Q_1\setminus
  0_{Q_1})$, and if it is conic, i.e.\ invariant under positive scalar
  multiplication in the fibres.}
Lagrangian correspondence $L_{01}\to T^*Q_0^-\times T^*Q_1$ gives
rise to a class of operators $\FIO_\rho(L_{01})$ depending on a real
parameter $\rho > 1/2$, mapping smooth functions on $Q_0$ to
distributions on $Q_1$.  These operators satisfy the composability
property similar to \eqref{maineq}.  Namely, \cite[Theorem
  4.2.2]{ho:fio1} shows that if a pair $L_{01}\to T^*Q_0^-\times
T^*Q_1$, $L_{12}\to T^*Q_1^-\times T^*Q_2$ satisfies
\begin{equation} \label{PPcond}
\begin{array}{l}
\text{$L_{01}\times L_{12}$ intersects $T^*Q_0^- \times\Delta_{T^*Q_1}\times T^*Q_2$ transversally and} \\
\text{the projection from the intersection to $T^*Q_0^- \times T^*Q_2$ is proper,
}
\end{array}
\end{equation}
then the corresponding operators are composable and
\begin{equation} \label{fio1}
FIO_\rho(L_{01}) \circ FIO_\rho(L_{12}) \subset FIO_\rho(L_{01} \circ L_{12}).
\end{equation}
Hence, similar to our construction of $\Symp^\#$, one could define a category $\Horm^\#$, whose
\begin{itemize}
\item objects are compact smooth manifolds,
\item morphisms are sequences of pairs $(L_{01},P_{01})$ of
immersed  homogeneous Lagrangian correspondences (between cotangent bundles) together with operators $P_{01} \in FIO_\rho(L_{01})$, modulo the equivalence relation that is
  generated by $(\ldots, (L_{01},P_{01}),(L_{12},P_{12}), \ldots )
  \sim (\ldots, (L_{01} \circ L_{12},P_{01} \circ P_{12}),\ldots)$ for
  $L_{01},L_{12}$ satisfying \eqref{PPcond}.
\end{itemize}
%
A morphism in this category might be called a {\em generalized Fourier
  integral operator}. \end{remark}

\section{Donaldson-Fukaya category of Lagrangians}

Throughout this paper we will use the notation and constructions for
(quilted) Floer homology and relative invariants introduced in
\cite{quiltfloer, quilts}.  In particular, we will be using the
following standing assumptions on symplectic manifolds, Lagrangian
submanifolds, and gradings; see \cite{quiltfloer} for details.

\label{M ass}
\begin{itemize}
\item[\bf(M1):] $(M,\omega)$ is {\em monotone}, that is 
$[\omega] = \tau c_1(TM) $
for some
$\tau\geq 0$.
\item[\bf(M2):] If $\tau>0$ then $M$ is compact. If $\tau=0$ then 
$M$ is (necessarily) noncompact but satisfies ``bounded geometry'' 
assumptions as in \cite{se:bo}.
\end{itemize}

\begin{itemize}
\setlength{\itemsep}{1mm}
\item[\bf(L1):]
\label{L ass}
$L\subset M$ is {\em monotone}, that is the symplectic area and Maslov index are related by $ 2 A(u) = \tau I(u)$ for all $u \in \pi_2(M,L)$,
where the $\tau\geq 0$ is (necessarily) that from (M1).
\item[\bf(L2):]
$L$ is compact and oriented.
\item[\bf(L3):] 
 $L$ has minimal Maslov number $N_L\geq 3$.
\end{itemize}

\label{G ass}
\begin{itemize}
\item[\bf(G1):]
$M$ is equipped with a Maslov covering $\Lag^N(M)$ for $N$ even, 
and the induced $2$-fold Maslov covering $\Lag^2(M)$ is the one 
given by oriented Lagrangian subspaces.
\item[\bf(G2):]
$L\subset M$ is equipped with a grading $\sigma_L^N:L\to\Lag^N(M)$, and
the induced $2$-grading $L\to\Lag^2(M)$ is the one given by the orientation of $L$.
\end{itemize}

In the following we review the construction of the Donaldson-Fukaya category $\Don(M)$ for a symplectic manifold  $(M,\omega)$ satisfying (M1-2).
The ``closed'' analog of this category, whose morphisms are
symplectomorphisms, was introduced by Donaldson in a
seminar talk \cite[12.6]{ms:jh}.  Subsequently Fukaya introduced
an \ainfty category involving Lagrangian submanifolds.  Here we describe the category arising from the Fukaya category by taking homology.

We fix a Maslov cover $\Lag^N(M) \to M$ as in (G1), which will be used to grade Floer cohomology groups,
and a {\em background class $b \in H^2(M,\Z_2)$},
which will be used to fix orientations of moduli spaces and thus
define Floer cohomology groups with $\Z$ coefficients.  In our
examples, $b$ will be either $0$ or the second Stiefel-Whitney class
$w_2(M)$ of $M$.

\begin{definition} \label{def:adm}
We say that a Lagrangian submanifold $L \subset M$ 
is {\em admissible} if it satisfies (L1-3), (G2), and the image of $\pi_1(L)$ in $\pi_1(M)$ 
is torsion.
\end{definition}

The assumption on $\pi_1(L)$ guarantees that any collection of admissible
Lagrangian submanifolds is monotone with respect to any surface in the sense of \cite{quiltfloer}.  Alternatively, one could work
with Bohr-Sommerfeld monotone Lagrangians as described
in \cite{quiltfloer}.  The assumption (L3) implies that the
Floer cohomology of any sequence is well-defined, and can be relaxed
to $N_L\geq 2$ by working with matrix factorizations as explained in
\cite{fieldb}.

\begin{definition}
A {\em brane structure} on an admissible $L$ consists of an orientation,
a grading, and a relative spin structure with background class $b$, see \cite{quiltfloer,orient} for details.  An admissible Lagrangian
equipped with a brane structure will be called a {\em Lagrangian
brane}.
\end{definition}

\begin{remark}
\begin{enumerate} 
\item We have not included in the definition of Lagrangian branes the
data of a flat vector bundle, in order to save space.  The extension
of the constructions below to this case should be straight forward and is left to the reader.
\item If one wants only $\Z_{2}$-gradings on the morphism spaces of
the Donaldson-Fukaya category, then the assumptions (G1-2) may be ignored.
\item 
If one wants only $\Z_2$ coefficients, then the background class and
relative spin structures may be ignored.
\end{enumerate}
\end{remark}

\begin{definition} \label{dfn DonFuk}
The {\em Donaldson-Fukaya category}
$\Don(M) := \Don(M,\Lag^N(M),\omega,b)$
is defined as follows:

\begin{enumerate}
\item The objects of $\Don(M)$ are Lagrangian branes in $M$.
\item  The morphism spaces of $\Don(M)$ are the $\Z_N$-graded Floer 
cohomology groups with $\Z$ coefficients
$ \Hom(L,L') := HF(L,L') $
constructed using a choice of perturbation datum consisting of a pair
$(J,H)$ of a time-dependent almost complex structure $J$ and a
Hamiltonian $H$, as in e.g.\ \cite{quiltfloer}.
\item
The composition law in the category $\Don(M)$ is defined by 
\begin{align*} 
\Hom(L,L') \times  \Hom(L',L'') &\longrightarrow \;\;\Hom(L,L'') \\ 
 (f,g)\qquad\qquad\quad &\longmapsto  f\circ g:=\Phi_{P}(f\otimes g)  ,
\end{align*}
where $\Phi_{P}$ is the relative Floer theoretic invariant associated to the
``half-pair of pants'' surface ${P}$, that is, the disk with three
markings on the boundary (two incoming ends, one outgoing end) as in
Figure~\ref{pants}.
\end{enumerate}
\end{definition}

\begin{figure}[ht]

\begin{picture}(0,0)%
\includegraphics{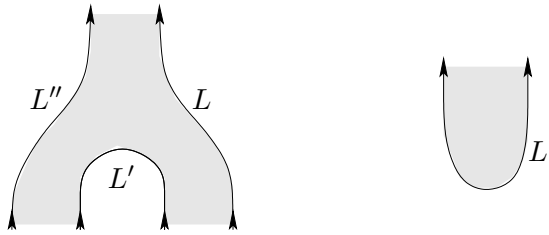}%
\end{picture}%
\setlength{\unitlength}{2901sp}%
\begingroup\makeatletter\ifx\SetFigFont\undefined%
\gdef\SetFigFont#1#2#3#4#5{%
  \reset@font\fontsize{#1}{#2pt}%
  \fontfamily{#3}\fontseries{#4}\fontshape{#5}%
  \selectfont}%
\fi\endgroup%
\begin{picture}(4928,1959)(1309,-2998)
\put(1500,-1951){\makebox(0,0)[lb]{\smash{$L''$}}}
\put(2886,-1951){\makebox(0,0)[lb]{\smash{$L$}}}
\put(2171,-2626){\makebox(0,0)[lb]{\smash{$L'$}}}
\put(5761,-2401){\makebox(0,0)[lb]{\smash{$L$}}}
\end{picture}%
%
%
\caption{Composition and identity in the Donaldson-Fukaya category}
\label{pants}
\end{figure}

\begin{remark}
\begin{enumerate}
\item
Associativity of the composition follows from the standard gluing theorem (see e.g.\cite[Theorem 2.7]{quilts}) applied to the surfaces in Figure \ref{assoc}: The
two ways of composing correspond to two ways of gluing the pair of
pants. The resulting surfaces are the same (up to a deformation of the
complex structure), hence the resulting compositions are the same.
\begin{figure}[ht]
\begin{picture}(0,0)%
\includegraphics{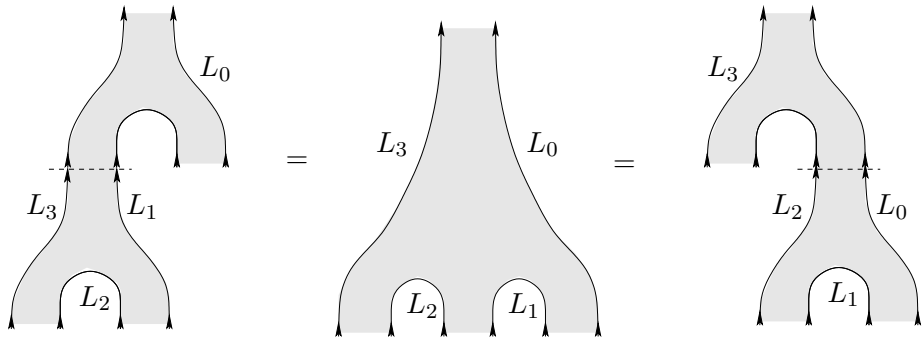}%
\end{picture}%
\setlength{\unitlength}{2072sp}%
\begingroup\makeatletter\ifx\SetFigFont\undefined%
\gdef\SetFigFont#1#2#3#4#5{%
  \reset@font\fontsize{#1}{#2pt}%
  \fontfamily{#3}\fontseries{#4}\fontshape{#5}%
  \selectfont}%
\fi\endgroup%
\begin{picture}(10929,3984)(184,-5023)
\put(9991,-4651){\makebox(0,0)[lb]{$L_1$}}
\put(9321,-3600){\makebox(0,0)[lb]{$L_2$}}
\put(10571,-3600){\makebox(0,0)[lb]{$L_0$}}
\put(8551,-1951){\makebox(0,0)[lb]{$L_3$}}
\put(6391,-2851){\makebox(0,0)[lb]{$L_0$}}
\put(6166,-4786){\makebox(0,0)[lb]{$L_1$}}
\put(4951,-4786){\makebox(0,0)[lb]{$L_2$}}
\put(4591,-2851){\makebox(0,0)[lb]{$L_3$}}
\put(2476,-1951){\makebox(0,0)[lb]{$L_0$}}
\put(1600,-3600){\makebox(0,0)[lb]{$L_1$}}
\put(401,-3600){\makebox(0,0)[lb]{$L_3$}}
\put(1036,-4696){\makebox(0,0)[lb]{$L_2$}}
\put(3511,-2986){\makebox(0,0)[lb]{$=$}}
\put(7426,-3031){\makebox(0,0)[lb]{$=$}}
\end{picture}%
\caption{Associativity of composition}
\label{assoc}
\end{figure}
\item
The identity $ 1_L\in \Hom(L,L)$ is the relative Floer theoretic invariant 
$1_L:=\Phi_S\in HF(L,L)$ 
associated to a disk $S$ with a single marking (an outgoing end),
see Figure~\ref{pants}.
The identity axiom $1_{L_0} \circ f = f = f \circ 1_{L_1}$
follows from the same gluing argument applied to the surfaces
on the left and right in Figure~\ref{ident}.
Here -- in contrast to the strips counted towards the Floer differential, where the equations are $\R$-invariant -- the equation on the strip need not be $\R$-invariant and solutions are counted without quotienting by $\R$. However, as in the strip example \cite[Example 2.5]{quilts} one can choose $\R$-invariant perturbation data to make the equation $\R$-invariant. Then the only isolated solutions contributing to the count are constant, and hence the relative invariant is the identity.
\end{enumerate}
\end{remark}

\begin{figure}[ht]
\begin{picture}(0,0)%
\includegraphics{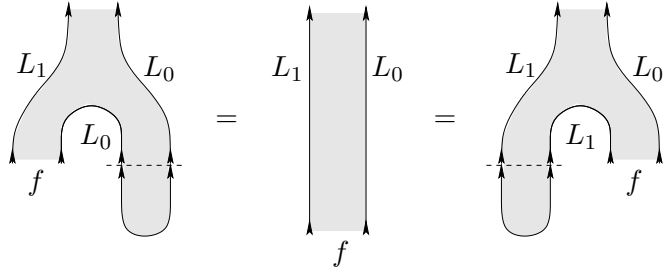}%
\end{picture}%
\setlength{\unitlength}{2072sp}%
\begingroup\makeatletter\ifx\SetFigFont\undefined%
\gdef\SetFigFont#1#2#3#4#5{%
  \reset@font\fontsize{#1}{#2pt}%
  \fontfamily{#3}\fontseries{#4}\fontshape{#5}%
  \selectfont}%
\fi\endgroup%
\begin{picture}(8033,3031)(1309,-4070)
\put(4450,-1996){\makebox(0,0)[lb]{$L_1$}}
\put(5650,-1996){\makebox(0,0)[lb]{$L_0$}}
\put(1396,-1951){\makebox(0,0)[lb]{$L_1$}}
\put(2916,-1996){\makebox(0,0)[lb]{$L_0$}}
\put(2161,-2806){\makebox(0,0)[lb]{$L_0$}}
\put(3781,-2536){\makebox(0,0)[lb]{$=$}}
\put(6391,-2536){\makebox(0,0)[lb]{$=$}}
\put(7246,-1951){\makebox(0,0)[lb]{$L_1$}}
\put(8766,-1996){\makebox(0,0)[lb]{$L_0$}}
\put(7966,-2806){\makebox(0,0)[lb]{$L_1$}}
\put(5176,-4221){\makebox(0,0)[lb]{$f$}}
\put(8686,-3366){\makebox(0,0)[lb]{$f$}}
\put(1531,-3366){\makebox(0,0)[lb]{$f$}}
\end{picture}%
\caption{Identity axiom}
\label{ident}
\end{figure}

\begin{remark} \label{Don indep}
The category $\Don(M)$ is independent of the choices of perturbation
data involved in the definition of Floer homology and the relative invariants, up to isomorphism of categories: The relative invariants for the
infinite strip with perturbation data interpolating between two
different choices gives an isomorphism of the morphism spaces, see e.g.\ \cite{quiltfloer}. The gluing theorem implies compatibility of these
morphisms with compositions and identities.
\end{remark}

\subsection{Functor associated to symplectomorphisms}
\label{sym}

Next, we recall that any graded symplectomorphism 
(see \cite{se:bo} or \cite{quiltfloer} for the grading)
$\psi: M_0 \to M_1$ induces a functor between Donaldson-Fukaya categories.

\begin{definition}
Let  $\Phi(\psi) : \Don(M_0) \to \Don(M_1) $ be the functor defined
\begin{enumerate}
\item on the level of objects by $L \mapsto \psi(L)$,
\item on the level of morphisms by the map 
$ HF(L_0,L_1) \to HF(\psi(L_0),\psi(L_1)) $
induced by the obvious map of chain complexes
$$ CF(L_0,L_1) \to CF(\psi(L_0),\psi(L_1)), 
\ \ \ \bra{x} \mapsto \bra{\psi(x)} $$
for all $x \in \cI(L_0,L_1)$. (Here we use the Hamiltonians
$H\in\Ham(L_0,L_1)$ and $H\circ\psi^{-1}\in\Ham(\psi(L_0),\psi(L_1))$.)
\end{enumerate}
\end{definition}

Note that $\Phi(\psi)$ satisfies the functor axioms
$$ \Phi(\psi)(f\circ g) = \Phi(\psi)(f) \circ \Phi(\psi)(g), 
\ \ \ \ \ \Phi(\psi)(1_L) = 1_{\psi(L)} .$$
Furthermore if $\psi_{01}: M_0 \to M_1$ and $\psi_{12}: M_1 \to M_2$
are symplectomorphisms then
$$ \Phi(\psi_{12} \circ \psi_{01}) = \Phi(\psi_{01}) \circ
\Phi(\psi_{12}) .$$
In terms of Lagrangian correspondences this functor is
$L\mapsto L\circ{\rm graph}\,\psi$ on objects. This suggests
that one should extend the functor to more general Lagrangian 
correspondences $L_{01}\subset M_0^-\times M_1$ by 
$L\mapsto L\circ L_{01}$ on objects.
However, these compositions are generically only immersed, so
one would have to allow for singular Lagrangians as objects in $\Don(M_1)$.
Moreover, it is not clear how to extend the functor on the level
of morphisms, that is Floer cohomology groups.
In the following sections we propose some alternative definitions 
of functors associated to general Lagrangian correspondences.

\subsection{First functor associated to Lagrangian correspondences} 

We now define a first functor associated to a Lagrangian correspondence. 
Fix an integer $N > 0$ and let $\Ab_{N}$ be the category of
$\Z_{N}$-graded abelian groups.  Let $\Don(M)^\dual$ be the category
whose objects are functors from $\Don(M)$ to $\Ab_{N}$, and whose
morphisms are natural transformations.

Let $(M_0,\omega_0)$ and $(M_1,\omega_1)$ be symplectic manifolds satisfying (M1-2), equipped with $N$-fold Maslov coverings $\Lag^N(M_j)$ as in (G1) and background classes $b_j\in H^2(M_j,\Z_2)$, and let $L_{01} \subset M_0^- \times M_1$ be an admissible Lagrangian correspondence in the sense of Definition~\ref{def:adm},
equipped with a grading as in (G2) and a
relative spin structure with background class $- \pi_{0}^* b_0 + \pi_{1}^* b_1$.

\begin{definition}
The contravariant functor $ \Phi_{L_{01}}: \Don(M_0) \to
\Don(M_1)^\dual $ associated to $L_{01}$ is defined as follows:
\begin{enumerate}
\item 
On the level of objects, for every Lagrangian $L_0\subset M_0$ we define
a functor $ \Phi_{L_{01}}(L_0): \Don(M_1) \to \Ab_{N}$ by
$$  L_1 \mapsto HF(L_0, L_{01}, L_1) = HF(L_0 \times L_1,L_{01})$$
on objects $L_1\subset M_1$, and on morphisms
\begin{align*} 
 HF(L_1,L_1') &\to
\Hom( HF(L_0 ,L_{01}, L_1 ),  HF(L_0, L_{01}, L_1') ) \\
f \qquad&\mapsto \qquad\quad \bigl\{ g \mapsto \Phi_{\ul{S}_1}(g\otimes f)  \bigr\}
\end{align*}
is defined by the relative invariant for the quilted surface
$\ul{S}_1$ shown in Figure \ref{dual},
$$
\Phi_{\ul{S}_1}: HF( L_0 ,L_{01} , L_1 ) \otimes HF(L_1,L_1') 
\to HF( L_0, L_{01}, L_1') .
$$
\item
The functor on the level of morphisms associates to every $f\in HF(L_0,L_0')$
a natural transformation
$$ \Phi_{L_{01}}(f):  \Phi_{L_{01}}(L_0') 
\to \Phi_{L_{01}}(L_0), $$
which maps objects $L_1\subset M_1$ to the $\Ab_{N}$-morphism
$$
\Phi_{L_{01}}(f)(L_1) :\quad
\begin{aligned}
 HF(L_0 ',L_{01}, L_1 ) &\to HF(L_0, L_{01}, L_1) \\
 g \qquad&\mapsto \quad \Phi_{\ul{S}_0}(f\otimes g) 
\end{aligned}
$$
defined by the relative invariant for the quilted surface
$\ul{S}_0$ shown in Figure \ref{dual},
$$
\Phi_{\ul{S}_0}: HF(L_0,L_0') \otimes HF(L_0', L_{01}, L_1 )  
\to HF(L_0, L_{01}, L_1 ) .
$$
\end{enumerate}
\end{definition}

\begin{figure}[ht]
\begin{picture}(0,0)%
\includegraphics{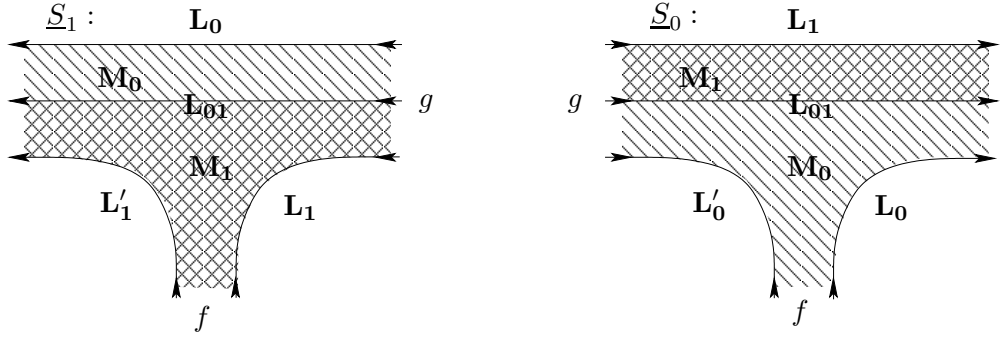}%
\end{picture}%
\setlength{\unitlength}{3108sp}%
\begingroup\makeatletter\ifx\SetFigFont\undefined%
\gdef\SetFigFont#1#2#3#4#5{%
  \reset@font\fontsize{#1}{#2pt}%
  \fontfamily{#3}\fontseries{#4}\fontshape{#5}%
  \selectfont}%
\fi\endgroup%
\begin{picture}(7944,2860)(3814,-4475)
\put(6031,-3481){\makebox(0,0)[lb]{$\mathbf{L_1}$}}
\put(5231,-2671){\makebox(0,0)[lb]{$\mathbf{L_{01}}$}}
\put(5276,-1996){\makebox(0,0)[lb]{$\mathbf{L_0}$}}
\put(4566,-3481){\makebox(0,0)[lb]{$\mathbf{L_1'}$}}
\put(5311,-4381){\makebox(0,0)[lb]{$f$}}
\put(7111,-2626){\makebox(0,0)[lb]{$g$}}
\put(4141,-1976){\makebox(0,0)[lb]{$\ul{S}_1 :$}}
\put(4546,-2456){\makebox(0,0)[lb]{$\mathbf{M_0}$}}
\put(5276,-3166){\makebox(0,0)[lb]{$\mathbf{M_1}$}}
\put(8300,-2626){\makebox(0,0)[lb]{$g$}}
\put(10081,-4326){\makebox(0,0)[lb]{$f$}}
\put(10746,-3481){\makebox(0,0)[lb]{$\mathbf{L_0}$}}
\put(10046,-1996){\makebox(0,0)[lb]{$\mathbf{L_1}$}}
\put(9291,-3481){\makebox(0,0)[lb]{$\mathbf{L_0'}$}}
\put(8956,-1971){\makebox(0,0)[lb]{$\ul{S}_0 :$}}
\put(10046,-3166){\makebox(0,0)[lb]{$\mathbf{M_0}$}}
\put(9181,-2456){\makebox(0,0)[lb]{$\mathbf{M_1}$}}
\put(10056,-2671){\makebox(0,0)[lb]{$\mathbf{L_{01}}$}}
\end{picture}%
\caption{Lagrangian functor for morphisms}
\label{dual}
\end{figure}

The composition axiom for the functors $\Phi_{L_{01}}(L_0)$ and the
commutation axiom for the natural transformations 
follow from the quilted gluing theorem \cite[Theorem 3.13]{quilts}\footnote{
A complete account of the gluing analysis for quilted surfaces can be found in \cite{mau:gluing}.
}
applied to Figures \ref{comp} and \ref{comp nat}.
\begin{figure}[ht]
\begin{picture}(0,0)%
\includegraphics{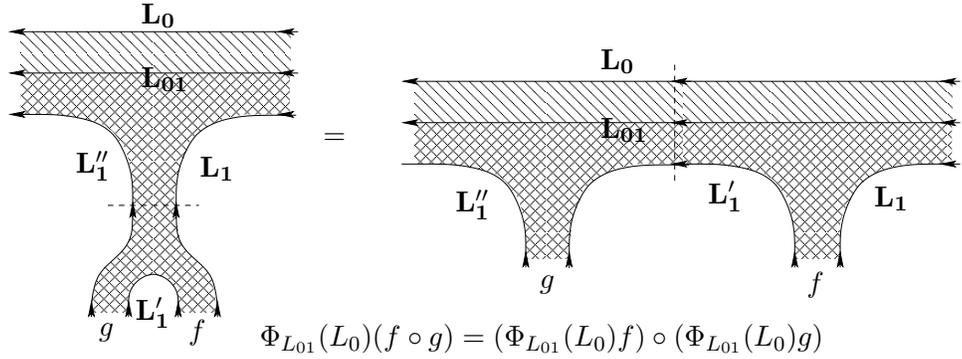}%
\end{picture}%
\setlength{\unitlength}{2279sp}%
\begingroup\makeatletter\ifx\SetFigFont\undefined%
\gdef\SetFigFont#1#2#3#4#5{%
  \reset@font\fontsize{#1}{#2pt}%
  \fontfamily{#3}\fontseries{#4}\fontshape{#5}%
  \selectfont}%
\fi\endgroup%
\begin{picture}(10374,3679)(-1541,-4034)
\put(-89,-1250){\makebox(0,0)[lb]{$\mathbf{L_{01}}$}}
\put(4250,-3446){\makebox(0,0)[lb]{$g$}}
\put(7150,-3496){\makebox(0,0)[lb]{$f$}}
\put(4906,-1800){\makebox(0,0)[lb]{$\mathbf{L_{01}}$}}
\put(3331,-2646){\makebox(0,0)[lb]{$\mathbf{L_1''}$}}
\put(6076,-2556){\makebox(0,0)[lb]{$\mathbf{L_1'}$}}
\put(4906,-1051){\makebox(0,0)[lb]{$\mathbf{L_0}$}}
\put(7886,-2546){\makebox(0,0)[lb]{$\mathbf{L_1}$}}
\put(-89,-555){\makebox(0,0)[lb]{$\mathbf{L_0}$}}
\put(-550,-3976){\makebox(0,0)[lb]{$g$}}
\put(426,-4016){\makebox(0,0)[lb]{$f$}}
\put(-150,-3816){\makebox(0,0)[lb]{$\mathbf{L_1'}$}}
\put(1891,-1816){\makebox(0,0)[lb]{$=$}}
\put(-800,-2221){\makebox(0,0)[lb]{$\mathbf{L_1''}$}}
\put(551,-2221){\makebox(0,0)[lb]{$\mathbf{L_1}$}}
\put(1200,-4100){\makebox(0,0)[lb]{$\Phi_{L_{01}}(L_0) (f\circ g) 
=(\Phi_{L_{01}}(L_0) f) \circ (\Phi_{L_{01}}(L_0) g)$}}
\end{picture}%
\caption{Composition axiom for Lagrangian functors}
\label{comp}
\end{figure}
\begin{figure}[ht]
\begin{picture}(0,0)%
\includegraphics{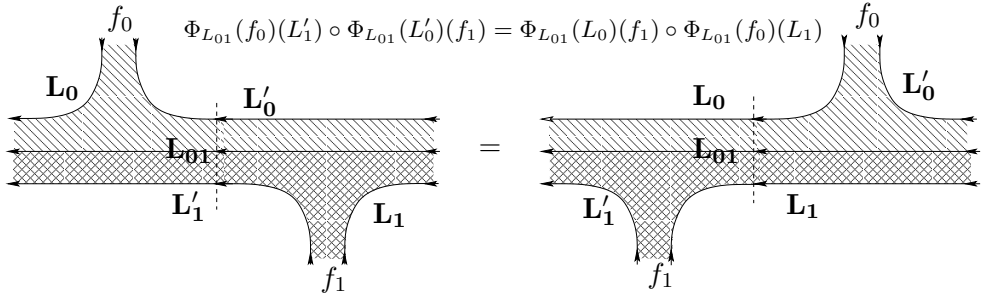}%
\end{picture}%
\setlength{\unitlength}{1782sp}%
\begingroup\makeatletter\ifx\SetFigFont\undefined%
\gdef\SetFigFont#1#2#3#4#5{%
  \reset@font\fontsize{#1}{#2pt}%
  \fontfamily{#3}\fontseries{#4}\fontshape{#5}%
  \selectfont}%
\fi\endgroup%
\begin{picture}(13569,3724)(-1721,-3449)
\put(3386,-2646){\makebox(0,0)[lb]{$\mathbf{L_1}$}}
\put(2656,-3556){\makebox(0,0)[lb]{$f_1$}}
\put(-1169,-881){\makebox(0,0)[lb]{$\mathbf{L_0}$}}
\put(1576,-1096){\makebox(0,0)[lb]{$\mathbf{L_0'}$}}
\put(586,-2566){\makebox(0,0)[lb]{$\mathbf{L_1'}$}}
\put(496,-1751){\makebox(0,0)[lb]{$\mathbf{L_{01}}$}}
\put(-314, 74){\makebox(0,0)[lb]{$f_0$}}
\put(6301,-2621){\makebox(0,0)[lb]{$\mathbf{L_1'}$}}
\put(7201,-3501){\makebox(0,0)[lb]{$f_1$}}
\put(10746,-881){\makebox(0,0)[lb]{$\mathbf{L_0'}$}}
\put(10081,119){\makebox(0,0)[lb]{$f_0$}}
\put(4861,-1636){\makebox(0,0)[lb]{$=$}}
\put(7831,-1751){\makebox(0,0)[lb]{$\mathbf{L_{01}}$}}
\put(7831,-1006){\makebox(0,0)[lb]{$\mathbf{L_0}$}}
\put(9136,-2521){\makebox(0,0)[lb]{$\mathbf{L_1}$}}
\put(500,0){\makebox(0,0)[lb]{
\smash{{\SetFigFont{9}{11.0}{\familydefault}{\mddefault}{\updefault}
$\Phi_{L_{01}}(f_0)(L_1')\circ\Phi_{L_{01}}(L_0')(f_1)
=\Phi_{L_{01}}(L_0)(f_1)\circ \Phi_{L_{01}}(f_0)(L_1) $}}}}
\end{picture}%
\caption{Commutation axiom for Lagrangian functors}
\label{comp nat}
\end{figure}

Clearly the functor $\Phi_{L_{01}}$ is unsatisfactory, since given two
Lagrangian correspondences $L_{01} \subset M_0^- \times M_1,
\ \ L_{12} \subset M_1^- \times M_2 $
it is not clear how to define the composition of the associated functors
$ \Phi_{L_{01}}: \Don(M_0) \to \Don(M_1)^\dual$ and 
$\Phi_{L_{12}}: \Don(M_1) \to \Don(M_2)^\dual$.
As a solution (perhaps not the only one) we will define in Section \ref{Don}
a category sitting in between $\Don(M)$ and $\Don(M)^\dual$.
This will allow for the definition of composable functors for general 
Lagrangian correspondences in Section \ref{Funk}.

\section{Donaldson-Fukaya category of generalized Lagrangians}
\label{Don}

In this section we extend the Donaldson-Fukaya category $\Don(M)$ to a
category $\Don^\#(M)$ which has generalized Lagrangian submanifolds as
objects.  Hence $\Don^\#(M)$ sits in between $\Don(M)$ and
$\Don(M)^\dual$.  One might draw an analogy here with the way
square-integrable functions sit between smooth functions and
distributions.  $\Don^\#(M)$ admits a functor to $\Don(M)^\dual$,
whose image is roughly speaking the subcategory of $\Don(M)^\dual$
generated by objects of geometric origin.
This extension of the Donaldson-Fukaya category is particularly
natural in our application to $2+1$-dimensional topological field
theory: One expects to associate a Lagrangian submanifold to any
three-manifold with boundary, but our constructions in fact yield generalized Lagrangian submanifolds that arise naturally from a decomposition into elementary cobordisms (or compression bodies).

Let $(M,\omega)$ be a symplectic manifold satisfying (M1-2) with
monotonicity constant $\tau\geq 0$.  We fix a Maslov cover $\Lag^N(M)
\to M$ as in (G1) and a background class $b \in H^2(M,\Z_2)$.  

\begin{definition}\label{def:gen adm lag}
\begin{enumerate}
\item
A generalized Lagrangian submanifold of $M$ is a generalized Lagrangian correspondence $\ul{L}$ from $\{\pt\}$ to $M$, in the sense of Definition \ref{Lag cor}.
That is, $\ul{L}=(L_{(-r)(-r+1)},\ldots,L_{(-1)0})$ is a sequence of
Lagrangian correspondences $L_{(i-1)i}\subset N_{i-1}^- \times
N_i$ for a sequence $N_{-r},\ldots,N_0$ of any length $r\geq 0$ of
symplectic manifolds with $N_{-r}=\{\pt\}$ a point and $N_0 = M$.  
\item 
We call a generalized Lagrangian $\ul{L}$ {\em admissible} if each $N_i$ satisfies (M1-2) with the monotonicity constant $\tau\geq 0$, each $L_{(i-1)i}$ satisfies (L1-3), and
the image of each $\pi_1(L_{(i-1)i})$ in $\pi_1(N_{i-1}^- \times N_i)$ is torsion.
\end{enumerate}
\end{definition}

Again, one could replace the torsion assumption on fundamental groups by Bohr-Sommerfeld monotonicity as described in \cite{quiltfloer}.
Note that an (admissible) Lagrangian submanifold $L\subset M$ is an (admissible)
generalized Lagrangian with $r=0$.  We picture a generalized Lagrangian 
$\ul{L}$ as a sequence
$$ \begin{diagram} \node{\{\pt\}} \arrow{e,t}{L_{(-r)(-r+1)}}
\node{N_{-r}} \arrow{e,t}{L_{(-r+1)(-r+2)}} 
\node{\ldots} \arrow{e,t}{L_{(-2)(-1)}}
\node{N_{-1}} \arrow{e,t}{L_{(-1)0}} \node{N_0 = M}
\end{diagram} .$$
Given two generalized Lagrangians $\ul{L},\ul{L}'$ of $M$ 
we can transpose one and concatenate them to a
sequence of Lagrangian correspondences from 
$\{\pt\}$ to $\{\pt\}$,
$$ \begin{diagram} \node{\{\pt\}} \arrow{e,t}{L_{(-r)(-r+1)}} 
\node{\ldots} \arrow{e,t}{L_{(-1)0}} \node{N_0 = M = N_0'}
\arrow{e,t}{(L_{(-1)0}')^t} \node{\ldots}
\arrow{e,t}{(L_{(-r')(-r'+1)}')^t} \node{\{\pt\}}
\end{diagram} .$$
The Floer cohomology of this sequence 
(as defined in \cite{quiltfloer}) is the natural generalization of the Floer cohomology
for pairs of Lagrangian submanifolds.  Hence we define
\begin{equation}\label{def HF}
HF(\ul{L},\ul{L}') :=
HF(L_{(-r)(-r+1)},\ldots,L_{(-1)0},(L_{(-1)0}')^t,\ldots,
(L_{(-r')(-r'+ 1)}')^t ) .
\end{equation}
Note here that every such sequence arising from a pair of admissible
generalized Lagrangians is automatically monotone 
by a Lemma of \cite{quiltfloer}.

\begin{definition}
The {\em generalized Donaldson-Fukaya category}
$$\Don^\#(M) := \Don^\#(M,\Lag^N(M),\omega,b) $$ is defined as follows: 
\begin{enumerate}
\item 
Objects of $\Don^\#(M)$ are admissible generalized Lagrangians of
$M$, equipped with orientations, a grading, and a relative spin
structure (see \cite{quiltfloer}).
%
\item
Morphism spaces of $\Don^\#(M)$ are the $\Z_N$-graded Floer
cohomology groups (see \eqref{def HF})
$$ \Hom(\ul{L},\ul{L}') := HF(\ul{L},\ul{L}')[d], \quad d = \hh \Bigl(
\sum_k \dim(N_k) + \sum_{k'} \dim(N'_{k'}) \Bigr) ,$$
given by choices of a perturbation datum and widths as described in
\cite{quiltfloer} and degree shift $d$.  For $\Z$-coefficients the Floer
cohomology groups are modified by the inclusion of additional
determinant lines as below in \eqref{canoncf}.
\item
Composition of morphisms in $\Don^\#(M)$,
\begin{align*} 
\Hom(\ul{L},\ul{L}') \times  \Hom(\ul{L}',\ul{L}'') &\longrightarrow \;\;\Hom(\ul{L},\ul{L}'') \\ 
 (f,g)\qquad\qquad\quad &\longmapsto  f\circ g:=\Phi_{\ul{P}}(f\otimes g) 
\end{align*}
is defined by the relative invariant $\Phi_{\ul{P}}$ associated to the
quilted half-pair of pants surface $\ul{P}$ in Figure
\ref{fancypants}, with the following orderings: The relative
invariant is independent of the ordering of the patches with one
outgoing end by a Remark in \cite{quilts}.  
The remaining patches with two incoming ends are ordered from the top down, that is,
starting with those furthest from the boundary.
\end{enumerate}
\end{definition}

\begin{figure}[ht]
\begin{picture}(0,0)%
\includegraphics{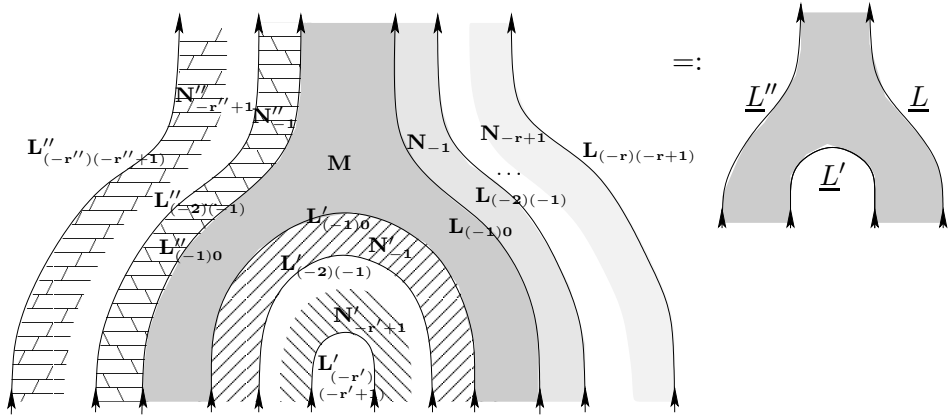}%
\end{picture}%
\setlength{\unitlength}{2901sp}%
\begingroup\makeatletter\ifx\SetFigFont\undefined%
\gdef\SetFigFont#1#2#3#4#5{%
  \reset@font\fontsize{#1}{#2pt}%
  \fontfamily{#3}\fontseries{#4}\fontshape{#5}%
  \selectfont}%
\fi\endgroup%
\begin{picture}(8048,3552)(-1390,-4546)  
\put(1351,-2446){\makebox(0,0)[lb]{\smash{{\SetFigFont{8}{9.6}{\familydefault}{\mddefault}{\updefault}$\mathbf{M}$}}}}
\put(2791,-2491){\makebox(0,0)[lb]{\smash{{\SetFigFont{8}{9.6}{\familydefault}{\mddefault}{\updefault}$\ldots$}}}}
\put(3531,-2336){\makebox(0,0)[lb]{\smash{{\SetFigFont{8}{9.6}{\familydefault}{\mddefault}{\updefault}$\mathbf{L_{\scriptscriptstyle(-r)(-r+1)}}$}}}}
\put(200,-2750){\makebox(0,0)[lb]{\smash{{\SetFigFont{8}{9.6}{\familydefault}{\mddefault}{\updefault}$\ldots$}}}}
\put(-134,-2761){\makebox(0,0)[lb]{\smash{{\SetFigFont{8}{9.6}{\familydefault}{\mddefault}{\updefault}$\mathbf{L''_{\scriptscriptstyle(-2)(-1)}}$}}}}
 \put(946,-3301){\makebox(0,0)[lb]{\smash{{\SetFigFont{8}{9.6}{\familydefault}{\mddefault}{\updefault}$\mathbf{L'_{\scriptscriptstyle(-2)(-1)}}$}}}}
\put(1171,-2896){\makebox(0,0)[lb]{\smash{{\SetFigFont{8}{9.6}{\familydefault}{\mddefault}{\updefault}$\mathbf{L'_{\scriptscriptstyle(-1)0}}$}}}}
\put(1266,-4156){\makebox(0,0)[lb]{\smash{{\SetFigFont{8}{9.6}{\familydefault}{\mddefault}{\updefault}$\mathbf{L'_{\hspace{-2mm}\scriptscriptstyle\overset{(-r')}{(-r'+1)}}}$}}}}
 \put(1700,-3150){\makebox(0,0)[lb]{\smash{{\SetFigFont{8}{9.6}{\familydefault}{\mddefault}{\updefault}$\mathbf{N'_{\scriptscriptstyle-1}}$}}}}
\put(706,-2041){\makebox(0,0)[lb]{\smash{{\SetFigFont{8}{9.6}{\familydefault}{\mddefault}{\updefault}$\mathbf{N''_{\scriptscriptstyle-1}}$}}}}
\put(46,-1861){\makebox(0,0)[lb]{\smash{{\SetFigFont{8}{9.6}{\familydefault}{\mddefault}{\updefault}$\mathbf{N''_{\scriptscriptstyle-r''+1}}$}}}}
\put(1401,-3751){\makebox(0,0)[lb]{\smash{{\SetFigFont{8}{9.6}{\familydefault}{\mddefault}{\updefault}$\mathbf{N'_{\scriptscriptstyle-r'+1}}$}}}}
\put(2386,-2986){\makebox(0,0)[lb]{\smash{{\SetFigFont{8}{9.6}{\familydefault}{\mddefault}{\updefault}$\mathbf{L_{\scriptscriptstyle(-1)0}}$}}}}
\put(2026,-2266){\makebox(0,0)[lb]{\smash{{\SetFigFont{8}{9.6}{\familydefault}{\mddefault}{\updefault}$\mathbf{N_{\scriptscriptstyle -1}}$}}}}
\put(2656,-2186){\makebox(0,0)[lb]{\smash{{\SetFigFont{8}{9.6}{\familydefault}{\mddefault}{\updefault}$\mathbf{N_{\scriptscriptstyle-r+1}}$}}}}
\put(-1214,-2311){\makebox(0,0)[lb]{\smash{{\SetFigFont{8}{9.6}{\familydefault}{\mddefault}{\updefault}$\mathbf{L''_{\scriptscriptstyle(-r'')(-r''+1)}}$}}}}
\put(-89,-3166){\makebox(0,0)[lb]{\smash{{\SetFigFont{8}{9.6}{\familydefault}{\mddefault}{\updefault}$\mathbf{L''_{\scriptscriptstyle(-1)0}}$}}}}
\put(2611,-2716){\makebox(0,0)[lb]{\smash{{\SetFigFont{8}{9.6}{\familydefault}{\mddefault}{\updefault}$\mathbf{L_{\scriptscriptstyle(-2)(-1)}}$}}}}
\put(4920,-1951){\makebox(0,0)[lb]{$\ul{L}''$}}
\put(6321,-1951){\makebox(0,0)[lb]{$\ul{L}$}}
\put(5556,-2626){\makebox(0,0)[lb]{$\ul{L}'$}}
\put(4296,-1591){\makebox(0,0)[lb]{$=:$}}
\end{picture}%
\caption{Quilted pair of pants} \label{quiltpop}
\label{fancypants}
\end{figure}

\begin{remark}
\begin{enumerate}
\item
Identities $1_{\ul{L}}\in \Hom(\ul{L},\ul{L})$ are furnished by
relative invariants $1_{\ul{L}}:=\Phi_{\ul{S}}\in \Hom(\ul{L},\ul{L})$
associated to the quilted disk $\ul{S}$ in Figure \ref{fancy id}, 
with patches ordered from the bottom up,
that is, starting with those closest to the boundary.
\item
The identity and associativity axioms are
satisfied with $\Z_2$ coefficients by the quilted gluing theorem \cite[Theorem 3.13]{quilts}
applied to the quilted versions of Figures \ref{assoc}, \ref{ident}.
\item
Both the identity and composition are degree $0$ by
a Remark of \cite{quilts}.
\item
$\Don^\#(M)$ is a small category. The objects form a set by the same arguments as in Remark~\ref{small}~(c); the morphisms are evidently constructed as set.
\end{enumerate}
\end{remark}

\begin{figure}[ht]
\begin{picture}(0,0)%
\includegraphics{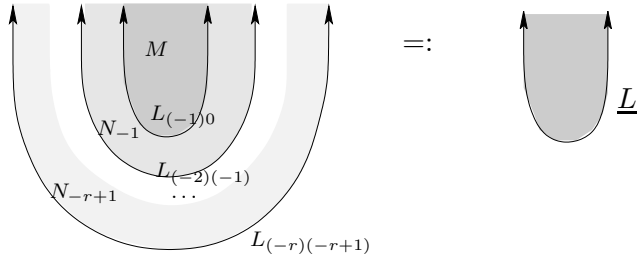}%
\end{picture}%
\setlength{\unitlength}{2901sp}%
\begingroup\makeatletter\ifx\SetFigFont\undefined%
\gdef\SetFigFont#1#2#3#4#5{%
  \reset@font\fontsize{#1}{#2pt}%
  \fontfamily{#3}\fontseries{#4}\fontshape{#5}%
  \selectfont}%
\fi\endgroup%
\begin{picture}(5169,2140)(4594,-3224)
\put(9800,-2041){\makebox(0,0)[lb]{$\ul{L}$}}
\put(5851,-2581){\makebox(0,0)[lb]{\smash{{\SetFigFont{8}{9.6}{\familydefault}{\mddefault}{\updefault}$L_{(-2)(-1)}$}}}}
\put(6661,-3166){\makebox(0,0)[lb]{\smash{{\SetFigFont{8}{9.6}{\familydefault}{\mddefault}{\updefault}$L_{(-r)(-r+1)}$}}}}
\put(5986,-2761){\makebox(0,0)[lb]{\smash{{\SetFigFont{8}{9.6}{\familydefault}{\mddefault}{\updefault}$\ldots$}}}}
\put(5806,-2086){\makebox(0,0)[lb]{\smash{{\SetFigFont{8}{9.6}{\familydefault}{\mddefault}{\updefault}$L_{(-1)0}$}}}}
\put(5761,-1546){\makebox(0,0)[lb]{\smash{{\SetFigFont{8}{9.6}{\familydefault}{\mddefault}{\updefault}$M$}}}}
\put(5356,-2221){\makebox(0,0)[lb]{\smash{{\SetFigFont{8}{9.6}{\familydefault}{\mddefault}{\updefault}$N_{-1}$}}}}
\put(4951,-2761){\makebox(0,0)[lb]{\smash{{\SetFigFont{8}{9.6}{\familydefault}{\mddefault}{\updefault}$N_{-r+1}$}}}}
\put(7966,-1501){\makebox(0,0)[lb]{$=:$}}
\end{picture}%
\caption{Quilted identity} \label{quiltid}
\label{fancy id}
\end{figure}

\begin{remark} \label{zmod}
To obtain the axioms with $\Z$ coefficients requires a modification of
the Floer cohomology groups, incorporating the determinant lines in a
more canonical way. This will be treated in detail in \cite{orient}, so
we only give a sketch here: For each intersection point $\ul{x} \in \cI(\ul{L},\ul{L}')$
we say that an {\em orientation} for $\ul{x}$ consists of the
following data: A partially quilted surface\footnote{
See \cite{orient} for the definition of partial quilts.
For example, the standard cup orientation for $\ul{x}=(x_1,\ldots,x_N)$
will use unquilted cups $S_i$ associated to each $T_{x_i} M_i$, and
identified via seams on the strip-like ends.
}
$\ul{S}$ with a single end,
complex vector bundles $\ul{E}$ over $\ul{S}$, and 
totally real subbundles $\cF$ over the boundaries and seams, 
such that near infinity on the strip-like ends 
$\ul{E}$ and $\cF$ are given by $(T_{x_i} M_i)$ and $T_{\ul{x}}\ul{L},T_{\ul{x}}\ul{L}'$ ; 
a real Cauchy-Riemann operator $D_{\ul{E},\cF}$; 
an orientation on the determinant line $\det(D_{\ul{E},\cF})$.  
We say that two orientations for $\ul{x}$ are {\em isomorphic} if 
the two problems have isomorphic bundles $\ul{E}$, 
and the surfaces, boundary and seam conditions
are deformation equivalent after a possible re-ordering of boundary
components etc., and the orientations are related by the isomorphism
of determinant lines arising from re-ordering. Let $O(\ul{x})$ denote
the space of isomorphism classes of orientations for $\ul{x}$.  Define
\begin{equation} \label{canoncf}
 \widetilde{CF}(\ul{L},\ul{L}') = \bigoplus_{\ul{x} \in \cI(\ul{L},\ul{L}')} O(\ul{x})
\otimes_{\Z_2} \Z .\end{equation}
The Floer coboundary operator extends canonically to an operator of
degree $1$ on $\widetilde{CF}(\ul{L},\ul{L}')$, and let
$\widetilde{HF}(\ul{L},\ul{L}')$ denote its cohomology.  This is similar to
the definition given in e.g. Seidel \cite[(12f)]{se:bo}, except that
we allow more general surfaces.  The group $\widetilde{HF}(\ul{L},\ul{L}')$
is of infinite rank over $\Z$, but it has finite rank over a suitable
graded-commutative Novikov ring generated by determinant lines.

The relative invariants extend to
operators $\widetilde{\Phi}_{\ul{S}}$ operating on the tensor product
of (extended) Floer cohomologies.  In particular, the quilted pair of
pants defines an operator
$$ \widetilde{\Phi}_{\ul{P}} : \widetilde{HF}(\ul{L},\ul{L}')
\otimes \widetilde{HF}(\ul{L}',\ul{L}'') \to
\widetilde{HF}(\ul{L},\ul{L}'') .$$
%
%
If we fix orientations for each generator $\langle\ul{x}\rangle$, as in 
the definition of $HF$, then the gluing sign for the first gluing 
(to the second incoming end) in the proof of associativity, 
Figure \ref{assoc}, is $+1$. For the second gluing (to the first incoming end) 
when applied to 
$\langle \ul{x}_1 \rangle \otimes \langle \ul{x}_2 \rangle\otimes \langle \ul{x}_3 \rangle$ the sign is
$(-1)^{|\ul{x}_3| \hh \sum_i \dim(N^{(1)}_i) } $.
Here $N^{(j)}_i$ denotes the sequence of symplectic manifolds underlying the generalized Lagrangian correspondence $\ul{L}_j$.
In addition, the two gluings induce different orderings of patches in the
glued quilted surface, which are related by the additional sign
$(-1)^{\bigl(\hh \sum_i \dim(N^{(1)}_i) \bigr) \bigl(\hh \sum_i \dim(N^{(2)}_i) \bigr) } $.
Combined together, these factors cancel the sign arising 
from the re-ordering of determinants in the definitions of 
$\widetilde{\Phi}_{\ul{P}}(\widetilde{\Phi}_{\ul{P}}(\langle \ul{x}_1\rangle \otimes
\langle \ul{x}_2 \rangle) \otimes \langle \ul{x}_3 \rangle )$ and
$\widetilde{\Phi}_{\ul{P}}( \langle \ul{x}_1 \rangle \otimes \widetilde{\Phi}_{\ul{P}}(\langle \ul{x}_2\rangle \otimes \langle \ul{x}_3 \rangle))$.

%
The identity axiom involves gluing a quilted cup with a quilted pair
of pants; the orderings of the patches for the quilted cup and quilted
pants above are chosen so that the gluing sign for gluing the quilted
cup with quilted pants to obtain a quilted strip is $+1$ for
gluing into the second argument, and 
$(-1)^{|\ul{x}|\hh \sum_i \dim(N_i)}$ for gluing into the first argument. 
Again, the additional sign is absorbed into the
isomorphism of determinant lines induced by gluing.
\end{remark}

\begin{convention} \label{convention}
To simplify pictures of quilts we will use the following conventions
indicated in Figure \ref{convent} : A generalized Lagrangian submanifold 
$\ul{L}$ of $M$ can be used as ``boundary condition'' for a surface mapping to
$M$ in the sense that the boundary arc that is labeled by the sequence
$\ul{L}=(L_{(-r)(-r+1)},\ldots,L_{(-1)0})$ of Lagrangian
correspondences from $\{{\rm pt}\}$ to $M$ is replaced by a sequence
of strips mapping to $N_{-1},\ldots,N_{-r+1}$, with seam conditions in
$L_{(-1)0},\ldots,L_{(-r+2)(-r+1)}$ and a final boundary condition in
$L_{(-r)(-r+1)}$.
Similarly, a generalized Lagrangian correspondence 
$\ul{L}$ between $M_-$ and $M_+$ can be used as ``seam condition'' between 
surfaces mapping to $M_\pm$ in the sense that the seam
that is labeled by the sequence $\ul{L}=(L_{01},\ldots,L_{(r-1)r})$ of
Lagrangian correspondences from $M_-$ to $M_+$ is replaced by a
sequence of strips mapping to $M_{1},\ldots,M_{r-1}$ with seam
conditions in $L_{01},\ldots,L_{(r-1)r}$.
\end{convention}

\begin{figure}[ht]
\begin{picture}(0,0)%
\includegraphics{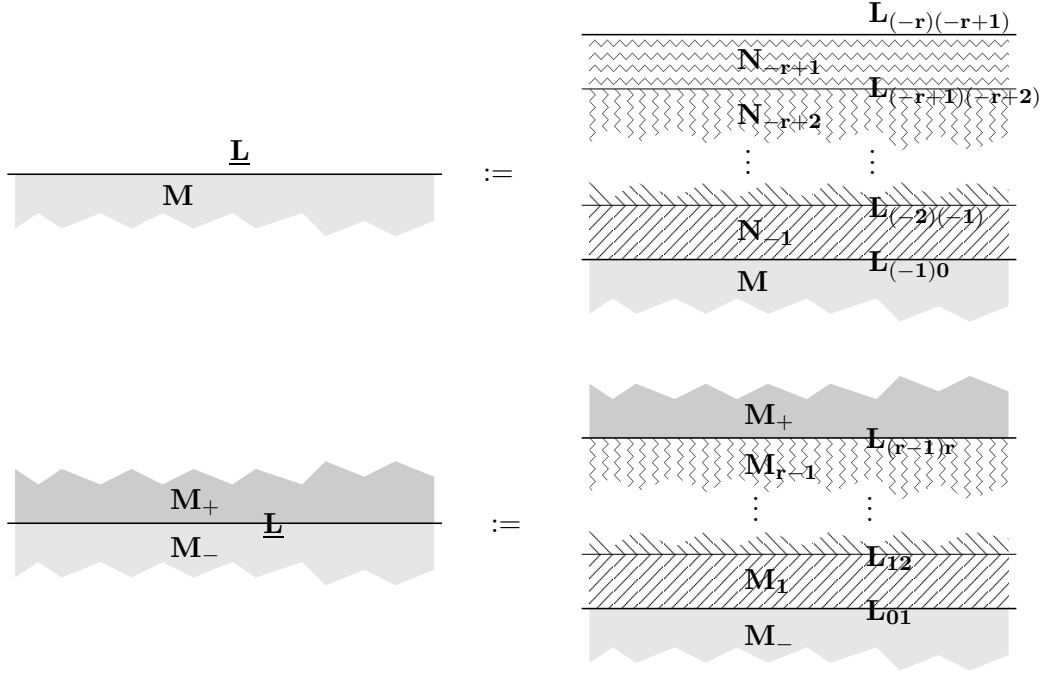}%
\end{picture}%
\setlength{\unitlength}{2565sp}%
\begingroup\makeatletter\ifx\SetFigFont\undefined%
\gdef\SetFigFont#1#2#3#4#5{%
 \reset@font\fontsize{#1}{#2pt}%
 \fontfamily{#3}\fontseries{#4}\fontshape{#5}%
 \selectfont}%
\fi\endgroup%
\begin{picture}(9794,6382)(1179,-7337)
\put(5801,-2600){\makebox(0,0)[lb]{$:=$}}
\put(3351,-2450){\makebox(0,0)[lb]{$\mathbf{\ul{L}}$}}
\put(9526,-1900){\makebox(0,0)[lb]{$\mathbf{L_{(-r+1)(-r+2)}}$}}
\put(8251,-1586){\makebox(0,0)[lb]{$\mathbf{N_{-r+1}}$}}
\put(9526,-1170){\makebox(0,0)[lb]{$\mathbf{L_{(-r)(-r+1)}}$}}
\put(8251,-2111){\makebox(0,0)[lb]{$\mathbf{N_{-r+2}}$}}
\put(2701,-2836){\makebox(0,0)[lb]{$\mathbf{M}$}}
\put(9526,-3580){\makebox(0,0)[lb]{$\mathbf{L_{(-1)0}}$}}
\put(9526,-3061){\makebox(0,0)[lb]{$\mathbf{L_{(-2)(-1)}}$}}
\put(8251,-3236){\makebox(0,0)[lb]{$\mathbf{N_{-1}}$}}
\put(8251,-3661){\makebox(0,0)[lb]{$\mathbf{M}$}}
\put(8326,-2536){\makebox(0,0)[lb]{$\mathbf{\vdots}$}}
\put(9526,-2536){\makebox(0,0)[lb]{$\mathbf{\vdots}$}}
\put(5876,-6000){\makebox(0,0)[lb]{$:=$}}
\put(8326,-5500){\makebox(0,0)[lb]{$\mathbf{M_{r-1}}$}}
\put(2776,-6300){\makebox(0,0)[lb]{$\mathbf{M_-}$}}
\put(9501,-6886){\makebox(0,0)[lb]{$\mathbf{L_{01}}$}}
\put(8326,-6586){\makebox(0,0)[lb]{$\mathbf{M_1}$}}
\put(8326,-7151){\makebox(0,0)[lb]{$\mathbf{M_-}$}}
\put(8401,-5886){\makebox(0,0)[lb]{$\mathbf{\vdots}$}}
\put(9501,-5886){\makebox(0,0)[lb]{$\mathbf{\vdots}$}}
\put(3676,-6061){\makebox(0,0)[lb]{$\mathbf{\ul{L}}$}}
\put(2776,-5836){\makebox(0,0)[lb]{$\mathbf{M_+}$}}
\put(8326,-5011){\makebox(0,0)[lb]{$\mathbf{M_+}$}}
\put(9501,-5290){\makebox(0,0)[lb]{$\mathbf{L_{(r-1)r}}$}}
\put(9501,-6361){\makebox(0,0)[lb]{$\mathbf{L_{12}}$}}
\end{picture}%
\caption{Conventions on using generalized Lagrangians and Lagrangian 
correspondences as boundary and seam conditions}
\label{convent}
\end{figure}

\begin{remark} \label{Donsharp indep}
As for $\Don(M)$, the category $\Don^\#(M)$ is independent of the
choices of perturbation data and widths up to isomorphism of
categories, see Remark \ref{Don indep} and the proofs of independence of
quilted Floer cohomology and relative quilt invariants in \cite{quiltfloer,quilts}.
\end{remark}

\begin{proposition}  The map
$\ul{L} \mapsto \ul{L}^\dual$, for a generalized Lagrangian $\ul{L}$ of $M$ 
given by
$$
\ul{L}^\dual(L_0) := \Hom(\ul{L},L_0) =
HF(L_{-r(-r+1)},\ldots,L_{(-1)0},L_0) [d]
$$
for all Lagrangian submanifolds $L_0\subset M$ and 
with degree shift $d=\hh  \sum_k \dim(N_k)$, 
extends to a contravariant functor
$ \Don^\#(M) \to \Don(M)^\dual .$
\end{proposition} 

\begin{proof}  
The functor $\ul{L}^\dual:\Don(M)\to\Ab_{N}$ can be defined on morphisms
by
$$
\ul{L}^\dual : 
\begin{aligned}
\Hom(L_1,L_1') &\to \Hom(\Hom(\ul{L},L_1),\Hom(\ul{L},L_1'))  \\
f \quad & \mapsto \quad\bigl\{  
g \mapsto g\circ f = \Phi_{\ul{P}}(g\otimes f)  \bigr\}
\end{aligned}
$$
using the composition on $ \Don^\#(M)$.  To morphisms $f\in
\Hom(\ul{L},\ul{L}')$ of $\Don^\#(M)$ we can then associate the natural
transformation $f^\dual: \ul{L}'^\dual \to \ul{L}^\dual$, which maps
every object $L_1\subset M$ of $\Don(M)$ to the following
$\Ab_{N}$-morphism $f^\dual(L_1)$:
$$ \Hom(\ul{L}',L_1) \to  \Hom(\ul{L},L_1), \quad 
g  \mapsto f\circ g = \Phi_{\ul{P}}(f\otimes g) ,$$
again given by composition on $ \Don^\#(M)$.
The axioms follow from the quilted gluing theorem \cite[Theorem 3.13]{quilts}
applied to jazzed-up versions of Figures \ref{comp} and \ref{comp nat} 
(which show the example $\ul{L}=(L_0,L_{01})$, $\ul{L}'=(L_0',L_{01})$).
In this case the orientations are independent of the ordering of patches
since all have one boundary component and one outgoing end.
\end{proof}

\section{Composable functors associated to Lagrangian correspondences}
\label{Funk}

Let $M_0$ and $M_1$ be two symplectic manifolds satisfying (M1-2) with
the same monotonicity constant $\tau\geq 0$.  We fix Maslov covers $\Lag^N(M_i)
\to M_i$ as in (G1) and background classes $b_i \in H^2(M_i,\Z_2)$.
Given an admissible Lagrangian correspondence $L_{01} \subset M_0^- \times M_1$ in the sense of Section~\ref{corcat}, we can now define a functor
$\Phi(L_{01}): \Don^\#(M_0) \to \Don^\#(M_1)$.
More precisely, we assume that $L_{01}$ satisfies (L1-3), and
the image of $\pi_1(L_{01})$ in $\pi_1(M_0^- \times M_1)$ is torsion.

\begin{definition} The functor $ \Phi(L_{01}): \Don^\#(M_0) \to \Don^\#(M_1) $ 
is defined as follows: 
\begin{enumerate}
\item
On the level of objects, $ \Phi(L_{01}) $ is concatenation of the 
Lagrangian correspondence to the sequence of Lagrangian correspondences:
For a generalized Lagrangian $\ul{L} = (L_{-r(-r+1)},\ldots,L_{(-1)0})$ of $M_0$
with corresponding sequence of symplectic manifolds 
$(\{\pt\},N_{-r+1},\ldots,N_{-1},M_0)$ we put
$$ \Phi(L_{01})(\ul{L}) := ( \ul{L}, L_{01} ) 
:= ( L_{(-r)(-r+1)},\ldots,L_{(-1)0},L_{01}) $$
with the corresponding symplectic manifolds
$(\{\pt\},N_{-r+1},\ldots,N_{-1},M_0,M_1)$;
\item
On the level of morphisms, for any pair $\ul{L},\ul{L}'$
of generalized Lagrangians in $M_0$, 
$$ \Phi(L_{01}) := \Phi_{\ul{S}} \,: \; \Hom(\ul{L},\ul{L}') \to
\Hom(\Phi(L_{01})(\ul{L}),\Phi(L_{01})(\ul{L}')) $$
is the relative invariant associated to the quilted surface
$\ul{S}$ with two punctures and one interior circle, as in Figure
\ref{morphism}.
\end{enumerate}
\end{definition}

\begin{figure}[ht]
\begin{picture}(0,0)%
\includegraphics{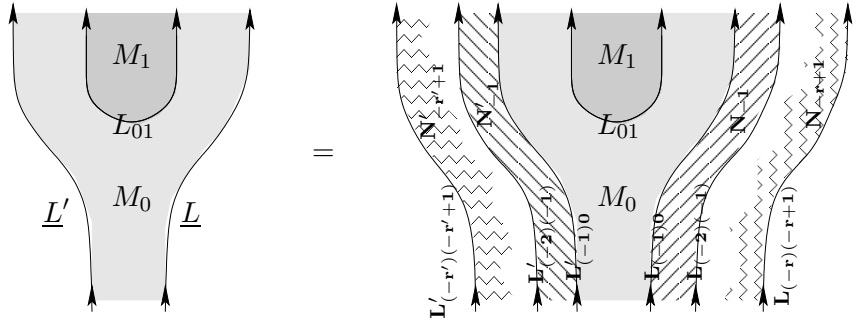}%
\end{picture}%
\setlength{\unitlength}{3108sp}%
\begingroup\makeatletter\ifx\SetFigFont\undefined%
\gdef\SetFigFont#1#2#3#4#5{%
  \reset@font\fontsize{#1}{#2pt}%
  \fontfamily{#3}\fontseries{#4}\fontshape{#5}%
  \selectfont}%
\fi\endgroup%
\begin{picture}(6743,2499)(-580,-3358)
\put(4120,-2536){\makebox(0,0)[lb]{$M_0$}}
\put(4120,-1401){\makebox(0,0)[lb]{$M_1$}}
\put(4120,-1956){\makebox(0,0)[lb]{$L_{01}$}}
\put(3691,-3121){\rotatebox{90.0}{\makebox(0,0)[lb]{\smash{{\SetFigFont{8}{9.6}{\familydefault}{\mddefault}{\updefault}$\mathbf{L'_{(-2)(-1)}}$}}}}}
\put(3991,-3101){\rotatebox{90.0}{\makebox(0,0)[lb]{\smash{{\SetFigFont{8}{9.6}{\familydefault}{\mddefault}{\updefault}$\mathbf{L'_{(-1)0}}$}}}}}
\put(4616,-3101){\rotatebox{90.0}{\makebox(0,0)[lb]{\smash{{\SetFigFont{8}{9.6}{\familydefault}{\mddefault}{\updefault}$\mathbf{L_{(-1)0}}$}}}}}
\put(4956,-3101){\rotatebox{90.0}{\makebox(0,0)[lb]{\smash{{\SetFigFont{8}{9.6}{\familydefault}{\mddefault}{\updefault}$\mathbf{L_{(-2)(-1)}}$}}}}}
\put(5650,-3301){\rotatebox{90.0}{\makebox(0,0)[lb]{\smash{{\SetFigFont{8}{9.6}{\familydefault}{\mddefault}{\updefault}$\mathbf{L_{(-r)(-r+1)}}$}}}}}
\put(2906,-3401){\rotatebox{90.0}{\makebox(0,0)[lb]{\smash{{\SetFigFont{8}{9.6}{\familydefault}{\mddefault}{\updefault}$\mathbf{L'_{(-r')(-r'+1)}}$}}}}}
\put(5906,-1886){\rotatebox{90.0}{\makebox(0,0)[lb]{\smash{{\SetFigFont{8}{9.6}{\familydefault}{\mddefault}{\updefault}$\mathbf{N_{-r+1}}$}}}}}
\put(5300,-1931){\rotatebox{90.0}{\makebox(0,0)[lb]{\smash{{\SetFigFont{8}{9.6}{\familydefault}{\mddefault}{\updefault}$\mathbf{N_{-1}}$}}}}}
\put(3280,-1876){\rotatebox{90.0}{\makebox(0,0)[lb]{\smash{{\SetFigFont{8}{9.6}{\familydefault}{\mddefault}{\updefault}$\mathbf{N'_{-1}}$}}}}}
\put(2826,-1976){\rotatebox{90.0}{\makebox(0,0)[lb]{\smash{{\SetFigFont{8}{9.6}{\familydefault}{\mddefault}{\updefault}$\mathbf{N'_{-r'+1}}$}}}}}
\put(800,-2671){\makebox(0,0)[lb]{$\ul{L}$}}
\put(250,-2536){\makebox(0,0)[lb]{$M_0$}}
\put(250,-1401){\makebox(0,0)[lb]{$M_1$}}
\put(-300,-2671){\makebox(0,0)[lb]{$\ul{L}'$}}
\put(250,-1956){\makebox(0,0)[lb]{$L_{01}$}}
\put(1846,-2131){\makebox(0,0)[lb]{$=$}}
\end{picture}%
\caption{The Lagrangian correspondence functor $\Phi(L_{01})$ on morphisms}
\label{morphism}
\end{figure}

\begin{remark}  In the case that $M_1$ is a point, the map for morphisms
is the dual of the pair of pants product.  
\end{remark}

For composable morphisms $f\in \Hom(\ul{L},\ul{L}')$, $g\in
\Hom(\ul{L}',\ul{L}'')$ one shows $\Phi_{L_{01}}(f \circ g) =
\Phi_{L_{01}}(f) \circ \Phi_{L_{01}}(g)$ by applying the quilted gluing theorem \cite[Theorem 3.13]{quilts}
to the gluings shown in Figure \ref{functoraxiom} (simplifying
the picture by Convention \ref{convention}), which yield
homotopic quilted surfaces.  
The gluing signs for both gluings are positive.  Similarly, the second
gluing shows that $\Phi(L_{01})(1_{\ul{L}}) = 1_{\Phi(L_{01})(\ul{L})}$,
since we have ordered the patches of the quilted cup from the outside in.

\begin{figure}[ht]
\begin{picture}(0,0)%
\includegraphics{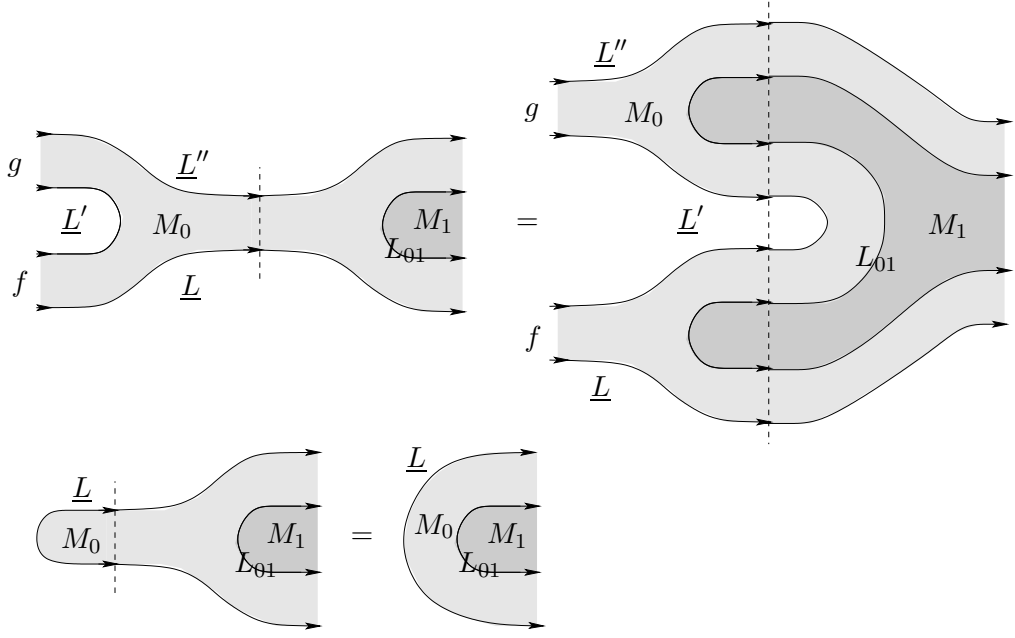}%
\end{picture}%
\setlength{\unitlength}{2279sp}%
\begingroup\makeatletter\ifx\SetFigFont\undefined%
\gdef\SetFigFont#1#2#3#4#5{%
  \reset@font\fontsize{#1}{#2pt}%
  \fontfamily{#3}\fontseries{#4}\fontshape{#5}%
  \selectfont}%
\fi\endgroup%
\begin{picture}(10947,6849)(-3419,-9238)
\put(991,-4921){\makebox(0,0)[lb]{$M_1$}}
\put(-1850,-4976){\makebox(0,0)[lb]{$M_0$}}
\put(-1529,-5656){\makebox(0,0)[lb]{$\ul{L}$}}
\put(676,-5236){\makebox(0,0)[lb]{$L_{01}$}}
\put(-1574,-4336){\makebox(0,0)[lb]{$\ul{L}''$}}
\put(3286,-3756){\makebox(0,0)[lb]{$M_0$}}
\put(6591,-4976){\makebox(0,0)[lb]{$M_1$}}
\put(5800,-5326){\makebox(0,0)[lb]{$L_{01}$}}
\put(3871,-4931){\makebox(0,0)[lb]{$\ul{L}'$}}
\put(2926,-6756){\makebox(0,0)[lb]{$\ul{L}$}}
\put(2971,-3121){\makebox(0,0)[lb]{$\ul{L}''$}}
\put(2116,-4876){\makebox(0,0)[lb]{$=$}}
\put(-2834,-4931){\makebox(0,0)[lb]{$\ul{L}'$}}
\put(2206,-6206){\makebox(0,0)[lb]{$f$}}
\put(2206,-3701){\makebox(0,0)[lb]{$g$}}
\put(-3419,-4301){\makebox(0,0)[lb]{$g$}}
\put(-3374,-5606){\makebox(0,0)[lb]{$f$}}
\put(-2834,-8396){\makebox(0,0)[lb]{$M_0$}}
\put(-2714,-7851){\makebox(0,0)[lb]{$\ul{L}$}}
\put(-600,-8341){\makebox(0,0)[lb]{$M_1$}}
\put(-950,-8656){\makebox(0,0)[lb]{$L_{01}$}}
\put(1801,-8341){\makebox(0,0)[lb]{$M_1$}}
\put(1486,-8656){\makebox(0,0)[lb]{$L_{01}$}}
\put(316,-8296){\makebox(0,0)[lb]{$=$}}
\put(1001,-8241){\makebox(0,0)[lb]{$M_0$}}
\put(921,-7531){\makebox(0,0)[lb]{$\ul{L}$}}
\end{picture}%
\caption{The functor axioms for $\Phi_{L_{01}}$}
\label{functoraxiom}
\end{figure}

\begin{remark}  
The surfaces of the first gluing in Figure \ref{functoraxiom} can equivalently 
be represented as degenerations of one quilted disk.
The corresponding one-parameter family in Figure \ref{functoraxiom sta} 
is the one-dimensional {\em multiplihedron} of Stasheff, see \cite{st:hs}, 
\cite[p. 113]{ma:op}, to which we will return in \cite{Ainfty}.
\end{remark}

\begin{figure}[ht]
\includegraphics[width=4in,height=1in]{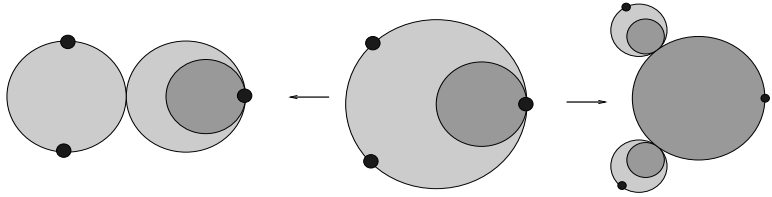}
\caption{Degeneration view of the first functor axiom}
\label{functoraxiom sta}
\end{figure}

With this new definition, any two functors associated to smooth, 
compact, admissible Lagrangian correspondences,
$\Phi(L_{01}):\Don^\#(M_0) \to \Don^\#(M_1) $
and $\Phi(L_{12}):\Don^\#(M_1) \to \Don^\#(M_2) $,
are clearly composable.
More generally, consider a sequence 
$$\ul{L}_{0r} = (L_{01},\ldots,L_{(r-1)r})$$
of Lagrangian correspondences $L_{(j-1)j} \subset M_{j-1}^- \times M_{j}$.
(That is, $\ul{L}_{0r}$ is a generalized Lagrangian correspondence from $M_0$
to $M_r$ in the sense of Definition~\ref{Lag cor}.)
Assume that $\ul{L}_{0r}$ is admissible in the sense of Section~\ref{corcat} below.
We can then define a functor by composition
\begin{equation}\label{funklong}
 \Phi(\ul{L}_{0r}):=\Phi(L_{01})\circ\ldots\circ \Phi(L_{(r-1)r}) : 
\Don^\#(M_0) \to \Don^\#(M_r) .
\end{equation}

\begin{figure}[ht]
\begin{picture}(0,0)%
\includegraphics{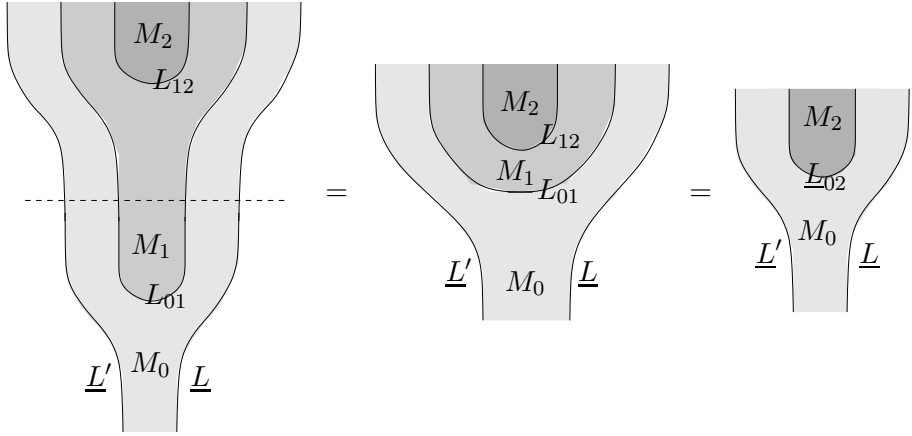}%
\end{picture}%
\setlength{\unitlength}{2279sp}%
\begingroup\makeatletter\ifx\SetFigFont\undefined%
\gdef\SetFigFont#1#2#3#4#5{%
  \reset@font\fontsize{#1}{#2pt}%
  \fontfamily{#3}\fontseries{#4}\fontshape{#5}%
  \selectfont}%
\fi\endgroup%
\begin{picture}(10715,4749)(-1181,-3358)
\put(201,839){\makebox(0,0)[lb]{$M_2$}}
\put(2296,-826){\makebox(0,0)[lb]{$=$}}
\put(6256,-826){\makebox(0,0)[lb]{$=$}}
\put(3600,-1771){\makebox(0,0)[lb]{$\ul{L}'$}}
\put(4250,-1861){\makebox(0,0)[lb]{$M_0$}}
\put(5041,-1771){\makebox(0,0)[lb]{$\ul{L}$}}
\put(6976,-1566){\makebox(0,0)[lb]{$\ul{L}'$}}
\put(7426,-1306){\makebox(0,0)[lb]{$M_0$}}
\put(8101,-1566){\makebox(0,0)[lb]{$\ul{L}$}}
\put(7500,-716){\makebox(0,0)[lb]{$\ul{L}_{02}$}}
\put(7486,-61){\makebox(0,0)[lb]{$M_2$}}
\put(4141,-656){\makebox(0,0)[lb]{$M_1$}}
\put(4596,-856){\makebox(0,0)[lb]{$L_{01}$}}
\put(4606,-256){\makebox(0,0)[lb]{$L_{12}$}}
\put(4191,99){\makebox(0,0)[lb]{$M_2$}}
\put(406,344){\makebox(0,0)[lb]{$L_{12}$}}
\put(-314,-2916){\makebox(0,0)[lb]{$\ul{L}'$}}
\put(170,-2746){\makebox(0,0)[lb]{$M_0$}}
\put(826,-2916){\makebox(0,0)[lb]{$\ul{L}$}}
\put(181,-1456){\makebox(0,0)[lb]{$M_1$}}
\put(321,-2000){\makebox(0,0)[lb]{$L_{01}$}}
\end{picture}
\caption{The composition $\Phi(L_{01})\circ\ldots\circ \Phi(L_{(r-1)r})$
is given by a relative invariant for the sequence 
$\ul{L}_{0r} = (L_{01},\ldots,L_{(r-1)r})$. (Here $r=2$.)}
\label{funkcomp}
\end{figure}

\begin{remark} \label{rmk funkcomp}
On the level of morphisms, the functor $\Phi(\ul{L}_{0r})$ is given by
the relative invariant associated to the quilted surface $\ul{S}$ in 
Figure \ref{funkcomp},
$$
\Phi(\ul{L}_{0r}) = \Phi_{\ul{S}} :
\Hom(\ul{L},\ul{L}') \to \Hom(\Phi(\ul{L}_{0r})(\ul{L}),\Phi(\ul{L}_{0r})(\ul{L}')) 
$$ 
for all generalized Lagrangian submanifolds
$\ul{L},\ul{L}'\in\Obj(\Don^\#(M_0))$, with patches with two outgoing
ends ordered from bottom up.  This follows from the quilted gluing theorem applied to the gluing shown in Figure \ref{funkcomp}.
\end{remark}

\subsection{Functors associated to composed Lagrangian correspondences and graphs}

The next two strip-shrinking results are summarized from
\cite{isom,quiltfloer,quilts}.  The first theorem describes the
isomorphism of Floer cohomology under geometric composition, while the
second describes the behavior of the relative invariants.  

\begin{theorem} \label{main2}  
Let $\ul{L}= (L_{01},\ldots,L_{r(r+1)})$ be a cyclic sequence of Lagrangian correspondences between symplectic manifolds $M_0,\ldots,M_{r+1}=M_0$.
Suppose that
\begin{enumerate} 
\item the symplectic manifolds all satisfy (M1-2) with 
the same monotonicity constant $\tau$,
\item  the Lagrangian correspondences all satisfy (L1-3),
\item  the sequence $\ul{L}$ is monotone, relatively spin, and graded;
%
\item for some $1\leq j\leq r$ the composition $L_{(j-1)j} \circ
L_{j(j+1)}$ is embedded in the sense of Definition \ref{embedded},
\end{enumerate}
Then with respect to the induced relative spin structure, orientation,
and grading on the modified sequence $\ul{L}'=(L_{01},\ldots, L_{(j-1)j} \circ L_{j(j+1)}, \ldots,L_{r(r+1)})$ there exists a canonical isomorphism of graded groups
\begin{equation*}
HF(\ul{L})= HF(\ldots L_{(j-1)j} , L_{j(j+1)} \ldots)
\overset{\sim}{\to} HF(\ldots L_{(j-1)j} \circ L_{j(j+1)}
\ldots)=HF(\ul{L}') ,
\end{equation*}
induced by the canonical identification of intersection points.
\end{theorem}

\begin{theorem} \label{intertwine}  
Consider a quilted surface $\ul{S}$ containing a patch $S_{\ell_1}$ that
is diffeomorphic to $\R \times [0,1]$ and attached via seams
$\sigma_{01}=\{(\ell_0,I_0),(\ell_1,\R\times\{0\})\}$ and 
$\sigma_{12}=\{(\ell_1,\R\times\{1\}),(\ell_2,I_2)\}$ to boundary components 
$I_0,I_2$ of other surfaces $S_{\ell_0}, S_{\ell_2}$. 
Let $\ul{M}$ be symplectic manifolds (satisfying (M1-2), (G1) with the same $\tau\geq 0$ and $N\in\N$) labeling the patches of $\ul{S}$, and $\L$ be Lagrangian boundary and seam conditions for $\ul{S}$ such that all Lagrangians in $\cL$ 
satisfy (L1-3), (G2), and $\cL$ is monotone and relative spin in the sense of \cite{quilts}.
Suppose that the Lagrangian correspondences $L_{\sigma_{01}} \subset
M_{\ell_0}^-\times M_{\ell_1}$, $L_{\sigma_{12}} \subset M_{\ell_1}^-
\times M_{\ell_2}$ associated to the boundary components of $S_{\ell_1}$ are
such that $L_{\sigma_{01}} \circ L_{\sigma_{12}}$ is embedded.  

Let $\ul{S}'$ denote the quilted
surface obtained by removing the patch $S_{\ell_1}$ and corresponding seams 
and replacing it by a new seam $\sigma_{02}:=\{(\ell_0,I_0),(\ell_2,I_2)\}$.
We define Lagrangian boundary conditions $\L'$ for $\ul{S}'$ by
$L_{\sigma_{02}}:=L_{\sigma_{01}} \circ L_{\sigma_{12}}$ . 
Then the isomorphisms in Floer cohomology 
$ \Psi_{\ul{e}} : HF(\ul{L}_{\ul{e}}) \to HF(\ul{L}'_{\ul{e}}) $
for each end $\ul{e}\in\E(\ul{S})\cong\E(\ul{S}')$
intertwine with the relative invariants:
$$\Phi_{\ul{S}'} \circ \biggl( \bigotimes_{\ul{e}\in \E_-}
\Psi_{\ul{e}} \biggr) = \biggl( \bigotimes_{\ul{e} \in \E_+}
\Psi_{\ul{e}} \biggr)\circ \Phi_{\ul{S}} \, [n_{\ell_1} d] . $$
Here $2n_{\ell_1}$ is the dimension of $M_{\ell_1}$, 
and $d=1,0,\,\text{or}\; -1$ according to whether the removed strip $S_{\ell_1}$ 
has two outgoing ends, one in- and one outgoing, or two incoming ends.
\end{theorem}

As first application of these results we will show that the composed
functor $\Phi(L_{01})\circ \Phi(L_{12}):\Don^\#(M_0) \to \Don^\#(M_2)$
is isomorphic to the functor $\Phi(L_{01}\circ L_{12})$ of the
geometric composition $L_{01}\circ L_{12}\subset M_0^-\times M_2$, if
the latter is embedded.  More precisely and more generally, we have
the following result.

\begin{theorem} \label{thm compiso}
Let $\ul{L}_{0r}= (L_{01},\ldots,L_{(r-1)r})$ and $\ul{L}'_{0r'}=
(L'_{01},\ldots,L'_{(r'-1)r'})$ be two admissible generalized
Lagrangian correspondence from $M_0$ to $M_{r}=M_{r'}$.  Suppose that
they are equivalent in the sense of Section~\ref{symp cat} 
through a series of embedded compositions of consecutive Lagrangian
correspondences and such that each intermediate generalized
Lagrangian correspondence is admissible.
Then for any two admissible generalized Lagrangian submanifolds
$\ul{L},\ul{L}'\in\Obj(\Don^\#(M_0))$ there is an isomorphism
$$\Psi:\Hom(\Phi(\ul{L}_{0r})(\ul{L}),\Phi(\ul{L}_{0r})(\ul{L}'))
\to \Hom(\Phi(\ul{L}'_{0r'})(\ul{L}),\Phi(\ul{L}'_{0r'})(\ul{L}'))$$
which intertwines the functors on the morphism level,
$$
\Psi \circ \Phi(\ul{L}_{0r}) = \Phi(\ul{L}'_{0r'}) \;: \;
\Hom(\ul{L},\ul{L}') \to \Hom(\Phi(\ul{L}'_{0r'})(\ul{L}),\Phi(\ul{L}'_{0r'})(\ul{L}')) .
$$
\end{theorem}
\begin{proof}
By assumption there exists a sequence of admissible generalized
Lagrangian correspondences $\ul{L}^{j}$ connecting
$\ul{L}^0=\ul{L}_{0r}$ to $\ul{L}^N=\ul{L}'_{0r'}$.  In each step two
consecutive Lagrangian correspondences $L_-$, $L_+$ in the sequence
$\ul{L}^j=(\ldots,L_-,L_+,\ldots)$ are replaced by their embedded composition 
$L_-\circ L_+$ in $\ul{L}^{j\pm
1}=(\ldots,L_-\circ L_+,\ldots)$.  To each $\ul{L}^j$ we associate
seam conditions for the quilted surface $\ul{S}^j$ on the right of
Figure \ref{funkcomp}. Replacing the consecutive correspondences by
their composition corresponds to shrinking a strip in this surface.
So Theorem \ref{intertwine} provides an isomorphism $\Psi_{\ul{e}^j_+}$
associated to the outgoing end $\ul{e}^j_+$ of each surface
$\ul{S}^j$ such that $\Psi_{\ul{e}^j_+}\circ\Phi_{\ul{S}^j} =
\Phi_{\ul{S}^{j\pm 1}}$.  Figure \ref{funkshrink} shows an example of
this degeneration.  The isomorphism $\Psi$ is given by concatenation of
the isomorphisms $\Psi_{\ul{e}^j_+}$ (and their inverses in case the
composition is between $\ul{L}^j$ and $\ul{L}^{j-1}$). It intertwines
$\Phi_{\ul{S}^0}=\Phi(\ul{L}_{0r})$ and
$\Phi_{\ul{S}^N}=\Phi(\ul{L}'_{0r'})$ as claimed.
\end{proof}

\begin{figure}[ht]
\begin{picture}(0,0)%
\includegraphics{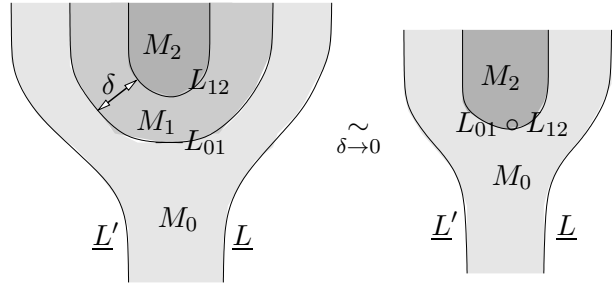}%
\end{picture}%
\setlength{\unitlength}{2486sp}%
\begingroup\makeatletter\ifx\SetFigFont\undefined%
\gdef\SetFigFont#1#2#3#4#5{%
  \reset@font\fontsize{#1}{#2pt}%
  \fontfamily{#3}\fontseries{#4}\fontshape{#5}%
  \selectfont}%
\fi\endgroup%
\begin{picture}(6710,2814)(2824,-2098)
\put(6076,-826){\makebox(0,0)[lb]{$\underset{\delta\to 0}{\sim}$}}
\put(3646,-1771){\makebox(0,0)[lb]{$\ul{L}'$}}
\put(4296,-1571){\makebox(0,0)[lb]{$M_0$}}
\put(5041,-1771){\makebox(0,0)[lb]{$\ul{L}$}}
\put(7026,-1666){\makebox(0,0)[lb]{$\ul{L}'$}}
\put(7626,-1156){\makebox(0,0)[lb]{$M_0$}}
\put(8271,-1666){\makebox(0,0)[lb]{$\ul{L}$}}
\put(7261,-626){\makebox(0,0)[lb]{$L_{01}\circ L_{12}$}}
\put(7516,-161){\makebox(0,0)[lb]{$M_2$}}
\put(4081,-606){\makebox(0,0)[lb]{$M_1$}}
\put(3741,-200){\makebox(0,0)[lb]{$\delta$}}
\put(4556,-806){\makebox(0,0)[lb]{$L_{01}$}}
\put(4591,-200){\makebox(0,0)[lb]{$L_{12}$}}
\put(4141,150){\makebox(0,0)[lb]{$M_2$}}
\end{picture}%

\caption{Isomorphism between the functors $\Phi(L_{01})\circ\Phi(L_{12})$
and $\Phi(L_{01}\circ L_{12})$}
\label{funkshrink}
\end{figure}

Next, let $\psi: M_0 \to M_1$ be a symplectomorphism and $\graph\psi
\subset M_0^- \times M_1$ its graph.  The functor $\Phi(\psi)$
defined in Section \ref{sym} extends to a functor 
$$ \Phi(\psi): \Don^\#(M_0) \to \Don^\#(M_1) $$
defined on the level of objects by 
$$ \ul{L}=(L_{-r(-r+1)},\ldots,L_{-10}) \mapsto (L_{-r(-r+1)},\ldots,(1_{N_{-1}}
 \times \psi)(L_{-10})) =: \psi(\ul{L}).$$
On the level of morphisms, the functor
$\Phi(\psi):\Hom(\ul{L},\ul{L}') \to
\Hom(\Phi(\psi)(\ul{L}),\Phi(\psi)(\ul{L}'))$ is defined by
$\bra{(x_{-r},\ldots,x_{-1},x_0,x'_{-1}, \ldots, x'_{-r'})} \mapsto
\bra{(x_{-r},\ldots,x_{-1},\psi(x_0),x'_{-1}, \ldots, x'_{-r'}} $ on the generators
$\cI(\ul{L},\ul{L}')$ of the chain complex.  As another application of
Theorem \ref{intertwine} we will show that this functor is in fact
isomorphic to the functor $\Phi(\graph\psi) : \Don^\#(M_0) \to
\Don^\#(M_1)$ that we defined for the Lagrangian correspondence
$\graph\psi$.

\begin{proposition}  \label{sym2}
$\Phi(\psi)$ and $\Phi(\graph\psi)$ are canonically isomorphic as
functors from $\Don^\#(M_0)$ to $\Don^\#(M_1)$. 
More precisely, there exists a canonical natural
transformation $\alpha: \Phi(\psi) \to \Phi(\graph\psi)$, that is 
$\alpha(\ul{L})\in\Hom(\Phi(\psi)(\ul{L}) , \Phi(\graph\psi)(\ul{L}))$ 
for every $\ul{L}\in\Obj(\Don^\#(M_0))$ such that
$\alpha(\ul{L})\circ\Phi(\graph\psi)(f) = \Phi(\psi)(f) \circ
\alpha(\ul{L}')$ for all $f\in\Hom(\ul{L},\ul{L}')$, and all
$\alpha(\ul{L})$ are isomorphisms in $\Don^\#(M_1)$.
\end{proposition} 

\begin{proof}  
Consider a generalized Lagrangian submanifold
$\ul{L}=(L_{(-r)(-r+1)},\ldots,L_{(-1)0}) \in \Obj(\Don^\#(M_0))$.
By Theorem \ref{main2} we have canonical isomorphisms from
\begin{multline*} 
\Hom(\Phi(\psi)\ul{L}, \Phi(\graph\psi)\ul{L}) = \Hom(\psi(\ul{L}),(\ul{L},\graph\psi)) \\
=\Hom(\ldots (1\times\psi)(L_{(-1)0}),(\graph\psi)^t,(L_{(-1)0})^t \ldots)
\end{multline*}
to all three of
\begin{align*}
& \Hom(\ldots (1\times\psi)(L_{(-1)0}),(L_{(-1)0}\circ(\graph\psi))^t \ldots)
= \Hom(\psi(\ul{L}),\psi(\ul{L})), \\
&\Hom(\ldots L_{(-1)0},\graph\psi,(\graph\psi)^t,(L_{(-1)0})^t \ldots)
= \Hom((\ul{L},\graph\psi)(\ul{L},\graph\psi)), \\
&\Hom(\ldots (1\times\psi)(L_{(-1)0})\circ\graph(\psi^{-1}),(L_{(-1)0})^t \ldots) 
= \Hom(\ul{L},\ul{L}) ,
\end{align*}
see Figure \ref{murks}.\footnote{
Strictly speaking, one has to apply the shift functor $\Psi_{M_0}$ of
Definition~\ref{dfn shift} to
adjust the relative spin structure on $\ul{L}$. However, 
$HF(\Psi_{M_0}(\ul{L}), \Psi_{M_0}(\ul{L}))$ is canonically
isomorphic to $HF(\ul{L},\ul{L})$.
}
The isomorphisms are by $(\psi(\ul{x}),\ul{x})\mapsto \psi(\ul{x})$,
$(\ul{x},\psi(x_0),\ul{x})$, or $\ul{x}$, respectively, on the level
of perturbed intersection points
$\ul{x}=(x_{-r},\ldots,x_0)\in\cI(\ul{L},\ul{L})$.  The first two
isomorphisms also intertwine the identity morphisms $1_{\psi(\ul{L})}
\cong 1_{(\ul{L},\graph\psi)}$ by Theorem \ref{main2} and the
degeneration of the quilted identity indicated in Figure \ref{murks};
this is the identity axiom for the functor $\Phi(\graph\psi)$.  The
identity axiom for $\Phi(\psi)$ implies that the above isomorphisms
(their composition which coincides with
$\Phi(\psi):\Hom(\ul{L},\ul{L}) \to \Hom(\psi(\ul{L}),\psi(\ul{L}))$)
also intertwine $1_{\ul{L}}$ with $1_{\psi(\ul{L})}$.  We define
$\alpha(\ul{L})\in \Hom(\Phi(\psi)\ul{L}, \Phi(\graph(\psi))\ul{L})$
to be the element corresponding to the identities
$1_{\Phi(\psi)(\ul{L})}\cong 1_{\Phi(\graph\psi)(\ul{L})} \cong
1_{\ul{L}}$ under these isomorphisms.

\begin{figure}[ht]
\begin{picture}(0,0)%
\includegraphics{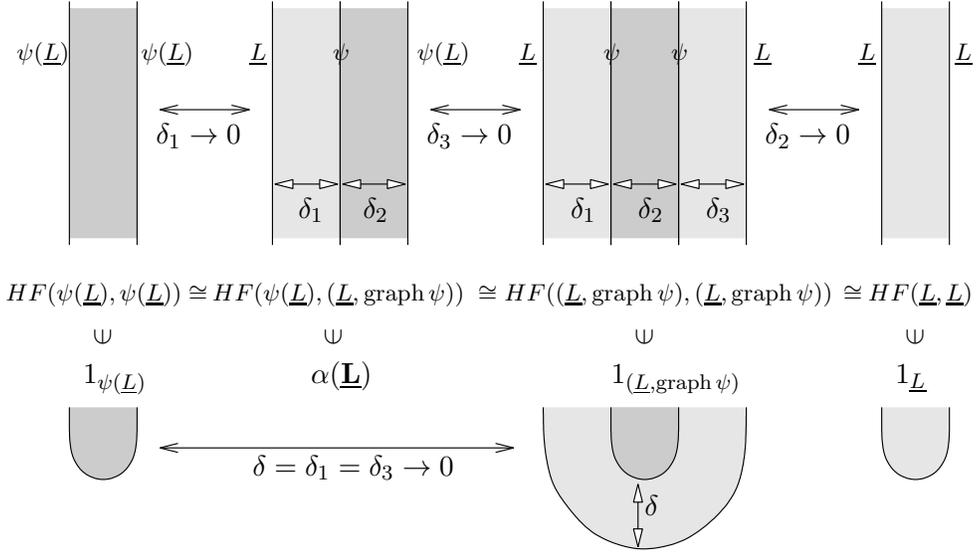}%
\end{picture}%
\setlength{\unitlength}{3729sp}%
\begingroup\makeatletter\ifx\SetFigFont\undefined%
\gdef\SetFigFont#1#2#3#4#5{%
  \reset@font\fontsize{#1}{#2pt}%
  \fontfamily{#3}\fontseries{#4}\fontshape{#5}%
  \selectfont}%
\fi\endgroup%
\begin{picture}(6087,3669)(226,-3268)
\put(1666,-2791){\makebox(0,0)[lb]{$\delta=\delta_1=\delta_3\to 0$}}
\put(4276,-3031){\makebox(0,0)[lb]{\smash{{\SetFigFont{11}{13.2}{\familydefault}{\mddefault}{\updefault}$\delta$}}}}
\put(721,-1901){\rotatebox{90.0}{\makebox(0,0)[lb]{\smash{{\SetFigFont{11}{13.2}{\familydefault}{\mddefault}{\updefault}$\in$}}}}}
\put(2251,-1901){\rotatebox{90.0}{\makebox(0,0)[lb]{\smash{{\SetFigFont{11}{13.2}{\familydefault}{\mddefault}{\updefault}$\in$}}}}}
\put(4321,-1901){\rotatebox{90.0}{\makebox(0,0)[lb]{\smash{{\SetFigFont{11}{13.2}{\familydefault}{\mddefault}{\updefault}$\in$}}}}}
\put(26,-1600){\makebox(0,0)[lb]{\smash{{\SetFigFont{9}{13.2}{\familydefault}{\mddefault}{\updefault}$HF(\psi(\ul{L}),\psi(\ul{L}))\cong$}}}}
\put(5761,-1600){\makebox(0,0)[lb]{\smash{{\SetFigFont{9}{13.2}{\familydefault}{\mddefault}{\updefault}$HF(\ul{L},\ul{L})$}}}}
\put(1401,-1600){\makebox(0,0)[lb]{\smash{{\SetFigFont{9}{13.2}{\familydefault}{\mddefault}{\updefault}$HF(\psi(\ul{L}),(\ul{L},\graph\psi))\;\cong$}}}}
\put(3341,-1600){\makebox(0,0)[lb]{\smash{{\SetFigFont{9}{13.2}{\familydefault}{\mddefault}{\updefault}$HF((\ul{L},\graph\psi),(\ul{L},\graph\psi))\,\cong$}}}}
\put(4051,-2150){\makebox(0,0)[lb]{\smash{{\SetFigFont{11}{13.2}{\familydefault}{\mddefault}{\updefault}$1_{(\ul{L},\graph\psi)}$}}}}
\put(5941,-2150){\makebox(0,0)[lb]{\smash{{\SetFigFont{11}{13.2}{\familydefault}{\mddefault}{\updefault}$1_{\ul{L}}$}}}}
\put(541,-2150){\makebox(0,0)[lb]{\smash{{\SetFigFont{11}{13.2}{\familydefault}{\mddefault}{\updefault}$1_{\psi(\ul{L})}$}}}}
\put(1026,-556){\makebox(0,0)[lb]{\smash{{\SetFigFont{11}{13.2}{\familydefault}{\mddefault}{\updefault}$\delta_1\to 0$}}}}
\put(2826,-556){\makebox(0,0)[lb]{\smash{{\SetFigFont{11}{13.2}{\familydefault}{\mddefault}{\updefault}$\delta_3\to 0$}}}}
\put(5076,-556){\makebox(0,0)[lb]{\smash{{\SetFigFont{11}{13.2}{\familydefault}{\mddefault}{\updefault}$\delta_2\to 0$}}}}
\put(1971,-1051){\makebox(0,0)[lb]{\smash{{\SetFigFont{11}{13.2}{\familydefault}{\mddefault}{\updefault}$\delta_1$}}}}
\put(2401,-1051){\makebox(0,0)[lb]{\smash{{\SetFigFont{11}{13.2}{\familydefault}{\mddefault}{\updefault}$\delta_2$}}}}
\put(4681,-1051){\makebox(0,0)[lb]{\smash{{\SetFigFont{11}{13.2}{\familydefault}{\mddefault}{\updefault}$\delta_3$}}}}
\put(4231,-1051){\makebox(0,0)[lb]{\smash{{\SetFigFont{11}{13.2}{\familydefault}{\mddefault}{\updefault}$\delta_2$}}}}
\put(3781,-1051){\makebox(0,0)[lb]{\smash{{\SetFigFont{11}{13.2}{\familydefault}{\mddefault}{\updefault}$\delta_1$}}}}
\put(6121,-1901){\rotatebox{90.0}{\makebox(0,0)[lb]{\smash{{\SetFigFont{11}{13.2}{\familydefault}{\mddefault}{\updefault}$\in$}}}}}
\put(2056,-2150){\makebox(0,0)[lb]{\smash{{\SetFigFont{11}{13.2}{\familydefault}{\mddefault}{\updefault}$\mathbf{\alpha(\ul{L})}$}}}}
\put(100,0){\makebox(0,0)[lb]{\smash{{\SetFigFont{9}{13.2}{\familydefault}{\mddefault}{\updefault}$\psi(\ul{L})$}}}}
\put(926,0){\makebox(0,0)[lb]{\smash{{\SetFigFont{9}{13.2}{\familydefault}{\mddefault}{\updefault}$\psi(\ul{L})$}}}}
\put(5700,0){\makebox(0,0)[lb]{\smash{{\SetFigFont{9}{13.2}{\familydefault}{\mddefault}{\updefault}$\ul{L}$}}}}
\put(6340,0){\makebox(0,0)[lb]{\smash{{\SetFigFont{9}{13.2}{\familydefault}{\mddefault}{\updefault}$\ul{L}$}}}}
\put(1651,0){\makebox(0,0)[lb]{\smash{{\SetFigFont{9}{13.2}{\familydefault}{\mddefault}{\updefault}$\ul{L}$}}}}
\put(2201,0){\makebox(0,0)[lb]{\smash{{\SetFigFont{9}{13.2}{\familydefault}{\mddefault}{\updefault}$\psi$}}}}
\put(2761,0){\makebox(0,0)[lb]{\smash{{\SetFigFont{9}{13.2}{\familydefault}{\mddefault}{\updefault}$\psi(\ul{L})$}}}}
\put(3441,0){\makebox(0,0)[lb]{\smash{{\SetFigFont{9}{13.2}{\familydefault}{\mddefault}{\updefault}$\ul{L}$}}}}
\put(4001,0){\makebox(0,0)[lb]{\smash{{\SetFigFont{9}{13.2}{\familydefault}{\mddefault}{\updefault}$\psi$}}}}
\put(4451,0){\makebox(0,0)[lb]{\smash{{\SetFigFont{9}{13.2}{\familydefault}{\mddefault}{\updefault}$\psi$}}}}
\put(5001,0){\makebox(0,0)[lb]{\smash{{\SetFigFont{9}{13.2}{\familydefault}{\mddefault}{\updefault}$\ul{L}$}}}}
\end{picture}%
\caption{Natural isomorphisms of Floer cohomology groups and definition of the natural
transformation $\alpha$ : The light and dark shaded surfaces are mapped to $M_0$ and
$M_1$ respectively and we abbreviate $\graph\psi$ by $\psi$
and $\Phi(\psi)(\ul{L})$ by $\psi(\ul{L})$.}
\label{murks}
\end{figure}

Each $\alpha(\ul{L})$ is an isomorphism since $\alpha(\ul{L})\circ f = I_1(f)$
for all $f\in HF(\Phi(\graph\psi)\ul{L},\ul{L}'')$ and
$f\circ\alpha(\ul{L})=I_2(f)$ for all  $f\in HF(\ul{L}'',\Phi(\psi)\ul{L})$,
with the isomorphisms from Theorem \ref{main2}
\begin{align*}
I_1 &:\, HF((\ul{L},\graph\psi),\ul{L}'')  \to HF(\psi(\ul{L}),\ul{L}'') , \\
I_2 &:\, HF(\ul{L}'',\psi(\ul{L})) \to HF(\ul{L}'',(\ul{L},\graph\psi)) .
\end{align*}
These identities can be seen from the gluing theorem in \cite{quilts} and Theorem~\ref{main2}, applied to the gluings and
degenerations indicated in Figure \ref{murks2}. The quilted surfaces
can be deformed to a strip resp.\ a quilted strip (which corresponds
to a strip in $M_0^-\times M_1$).  These relative invariants both are
the identity since the solutions are counted without quotienting by
$\R$, see the strip example \cite[Example 2.5]{quilts}.
\begin{figure}[ht]
\begin{picture}(0,0)%
\includegraphics{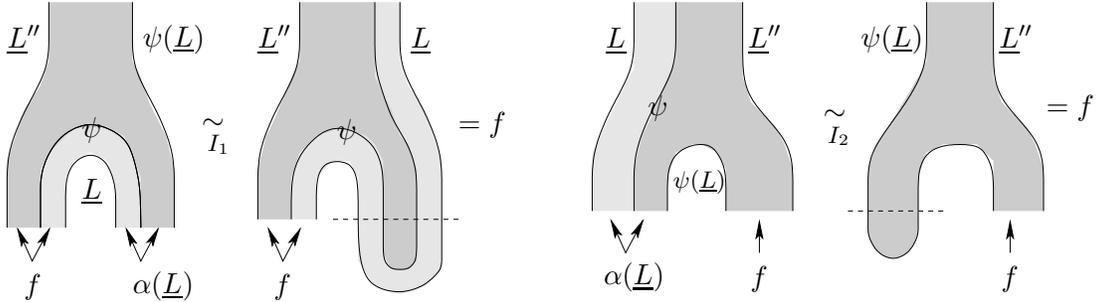}%
\end{picture}%
\setlength{\unitlength}{2763sp}%
\begingroup\makeatletter\ifx\SetFigFont\undefined%
\gdef\SetFigFont#1#2#3#4#5{%
  \reset@font\fontsize{#1}{#2pt}%
  \fontfamily{#3}\fontseries{#4}\fontshape{#5}%
  \selectfont}%
\fi\endgroup%
\begin{picture}(9921,2624)(1339,-9423)
\put(3126,-7261){\makebox(0,0)[rb]{$\psi(\ul{L})$}}
\put(1656,-7261){\makebox(0,0)[rb]{$\ul{L}''$}}
\put(2200,-8061){\makebox(0,0)[rb]{$\psi$}}
\put(2201,-8611){\makebox(0,0)[rb]{$\ul{L}$}}
\put(1661,-9471){\makebox(0,0)[rb]{$f$}}
\put(3026,-9471){\makebox(0,0)[rb]{$\alpha(\ul{L})$}}
\put(3901,-9471){\makebox(0,0)[rb]{$f$}}
\put(3900,-7261){\makebox(0,0)[rb]{$\ul{L}''$}}
\put(4490,-8081){\makebox(0,0)[rb]{$\psi$}}
\put(5151,-7261){\makebox(0,0)[rb]{$\ul{L}$}}
\put(5400,-8011){\makebox(0,0)[lb]{$= f$}}
\put(10551,-7261){\makebox(0,0)[rb]{$\ul{L}''$}}
\put(8301,-7261){\makebox(0,0)[rb]{$\ul{L}''$}}
\put(8176,-9396){\makebox(0,0)[rb]{$f$}}
\put(7251,-9396){\makebox(0,0)[rb]{$\alpha(\ul{L})$}}
\put(7281,-7861){\makebox(0,0)[rb]{$\psi$}}
\put(6901,-7261){\makebox(0,0)[rb]{$\ul{L}$}}
\put(9571,-7261){\makebox(0,0)[rb]{$\psi(\ul{L})$}}
\put(7800,-8461){\makebox(0,0)[rb]{\smash{{\SetFigFont{9}{9.6}{\familydefault}{\mddefault}{\updefault}$\psi(\ul{L})$}}}}
\put(10676,-7861){\makebox(0,0)[lb]{$=f$}}
\put(10426,-9396){\makebox(0,0)[rb]{$f$}}
\put(8676,-7861){\makebox(0,0)[lb]{$\sim$}}
\put(3101,-7936){\makebox(0,0)[lb]{$\sim$}}
\put(8716,-8061){\makebox(0,0)[lb]\smash{\SetFigFont{9}{9.6}{\familydefault}{\mddefault}{\updefault}$I_2$}}
\put(3151,-8136){\makebox(0,0)[lb]\smash{\SetFigFont{9}{9.6}{\familydefault}{\mddefault}{\updefault}$I_1$}}\end{picture}%
\caption{$\alpha(\ul{L})$ is an isomorphism in $\Don^\#(M_1)$}
\label{murks2}
\end{figure}
For $f\in\Hom(\ul{L},\ul{L}')$ this already shows the first equality in
$ \Phi(\psi)(f) \circ \alpha(\ul{L}') = I(f) = \alpha(\ul{L})\circ\Phi(\graph\psi)(f)$ 
with the isomorphism $I:HF(\ul{L},\ul{L}') \to HF(\psi(\ul{L}),(\ul{L}',\graph\psi))$.
More precisely, on the chain level for $x\in\cI(\ul{L},\ul{L}')$
$$
\Phi(\psi)(x) \circ \alpha(\ul{L}') = (\psi(x),x) 
=\alpha(\ul{L})\circ\Phi(\graph\psi)(x).
$$ 
The second identity is proven by repeatedly using Theorem \ref{intertwine} 
and the quilted gluing theorem \cite[Theorem 3.13]{quilts}, see Figure \ref{graph}.
\end{proof}
\begin{figure}[ht]
\begin{picture}(0,0)%
\includegraphics{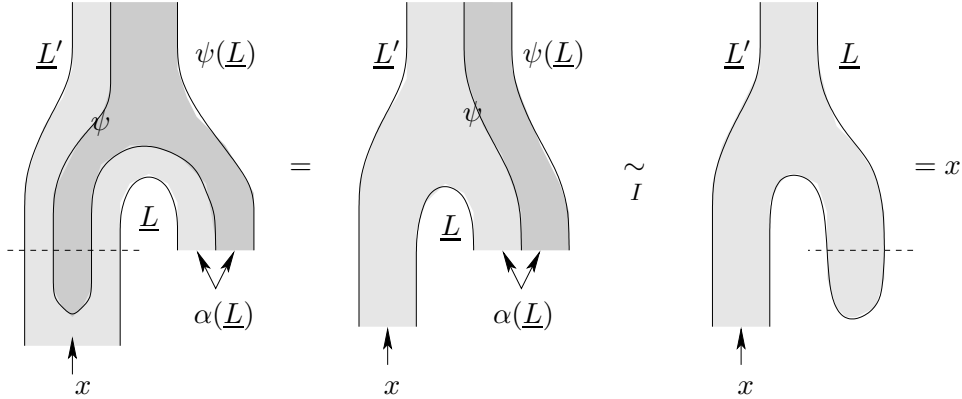}%
\end{picture}%
\setlength{\unitlength}{3158sp}%
\begingroup\makeatletter\ifx\SetFigFont\undefined%
\gdef\SetFigFont#1#2#3#4#5{%
  \reset@font\fontsize{#1}{#2pt}%
  \fontfamily{#3}\fontseries{#4}\fontshape{#5}%
  \selectfont}%
\fi\endgroup%
\begin{picture}(7512,3136)(2989,-9410)
\put(3676,-9361){\makebox(0,0)[rb]{$x$}}
\put(3826,-7336){\makebox(0,0)[rb]{$\psi$}}
\put(4206,-8086){\makebox(0,0)[rb]{$\ul{L}$}}
\put(4976,-6811){\makebox(0,0)[rb]{$\psi(\ul{L})$}}
\put(3456,-6811){\makebox(0,0)[rb]{$\ul{L}'$}}
\put(6101,-6811){\makebox(0,0)[rb]{$\ul{L}'$}}
\put(7551,-6811){\makebox(0,0)[rb]{$\psi(\ul{L})$}}
\put(6746,-7261){\makebox(0,0)[rb]{$\psi$}}
\put(6571,-8161){\makebox(0,0)[rb]{$\ul{L}$}}
\put(5406,-7636){\makebox(0,0)[rb]{$=$}}
\put(4956,-8861){\makebox(0,0)[rb]{$\alpha(\ul{L})$}}
\put(7306,-8861){\makebox(0,0)[rb]{$\alpha(\ul{L})$}}
\put(8026,-7636){\makebox(0,0)[rb]{$\sim$}}
\put(7906,-7836){\makebox(0,0)[lb]\smash{\SetFigFont{9}{9.6}{\familydefault}{\mddefault}{\updefault}$I$}}
\put(6101,-9361){\makebox(0,0)[rb]{$x$}}
\put(8876,-9361){\makebox(0,0)[rb]{$x$}}
\put(8856,-6811){\makebox(0,0)[rb]{$\ul{L}'$}}
\put(9701,-6811){\makebox(0,0)[rb]{$\ul{L}$}}
\put(10501,-7636){\makebox(0,0)[rb]{$=x$}}
\end{picture}%
\caption{Isomorphism of functors for a symplectomorphism and its
graph, using shrinking strips}
\label{graph}
\end{figure}

\begin{remark} 
There is an analytically easier proof of the previous Proposition
\ref{sym2} since it deals only with the special case when one of the Lagrangian correspondences is the graph of a symplectomorphism: Instead of shrinking a strip as in
Theorems~\ref{main2} and Theorem ~\ref{intertwine} one can apply the
symplectomorphism to the whole strip; for a suitable choice of
perturbation data it then attaches smoothly to the other surface in
the quilt, and the seam can be removed.
\end{remark} 

The functor $\Phi({\rm Id_{M_0}})$ associated to the identity map on
$M_0$ clearly is the identity functor on $\Don^\#(M_0)$. So
Proposition~\ref{sym2} gives a (rather indirect) isomorphism between
the functor for the diagonal and the identity functor. To be more
precise, taking into account the relative spin structure of the
diagonal, we need to introduce the following shift functor.

\begin{definition} \label{dfn shift}
We define a shift functor 
$$
\Psi_{M_0}: \ \Don^\#(M_0,\Lag^N(M_0),\omega_0,b_0) 
\to \Don^\#(M_0,\Lag^N(M_0),\omega_0,b_0-w_2(M_0)) .
$$
\begin{enumerate}
\item
On the level of objects, $\Psi_{M_0}$ maps every generalized Lagrangian
$\ul{L}\in\Don^\#(M_0)$ to itself but shifts the relative spin structure 
to one with background class $b_0-w_2(M_0)$, as explained in 
\cite{orient}.
\item
On the level of morphisms, 
$\Psi_{M_0}:\Hom(\ul{L},\ul{L}')\to \Hom(\Psi_{M_0}(\ul{L}),\Psi_{M_0}(\ul{L}'))$
is the canonical isomorphism for shifted spin structures from 
\cite{orient}.
\end{enumerate}
\end{definition}

\begin{remark} \label{rmk diag}
Let $\Delta\subset M_0^-\times M_0$ denote the diagonal.  Throughout,
we will equip $\Delta$ with the orientation and relative spin
structure that are induced by the projection to the second factor (see
\cite{orient}).
Then $\Delta$ is an admissible Lagrangian correspondence from $M_0$ to
$M_1$, where $M_1=M_0$ with the same symplectic structure
$\omega_1=\omega_0$ and Maslov cover $\Lag^N(M_1)=\Lag^N(M_0)$, but
with a shifted background class $b_1=b_0-w_2(M_0)$.  In other words,
$\Delta$ is an object in the category $\Don^\#\bigl(M_0,M_1)$ that is
introduced in Section~\ref{corcat} below.
\end{remark}

In the following, we will drop the Maslov cover and symplectic form from the notation.

\begin{corollary} \label{phidelta}
The functor $\Phi(\Delta):\Don^\#(M_0,b_0)\to\Don^\#(M_0,b_0-w_2(M_0))$
associated to the diagonal is canonically isomorphic to the shift functor $\Psi_{M_0}$.
\end{corollary}

\section{Composition functor for categories of correspondences}
\label{corcat}

The set of generalized Lagrangian correspondences forms a category in
its own right, which we define in close analogy to the generalized Donaldson
category in Section \ref{Don}. We will then be able to define a composition
functor for these categories.

Let $M_a$ and $M_b$ be symplectic manifolds satisfying (M1-2) with
the same monotonicity constant $\tau\geq 0$.  We fix an integer $N > 0$,
$N$-fold Maslov covers $\Lag^N(M_{(\cdot)}) \to M_{(\cdot)}$ as in (G1), and background
classes $b_{(\cdot)} \in H^2(M_{(\cdot)},\Z_2)$.
Recall from  Definition~\ref{Lag cor} that a generalized Lagrangian 
correspondence from $M_a$ to $M_b$ is a sequence 
$\ul{L}=(L_{01},L_{12},\ldots,L_{(r-1)r})$ of 
Lagrangian correspondences $L_{(i-1)i}\subset N_{i-1}^- \times N_i$ 
for a sequence $N_{0},\ldots,N_{r}$ of any length $r\geq 0$ of
symplectic manifolds with $N_{0}=M_a$ and $N_{r} = M_b$.  
We picture $\ul{L}$ as sequence
$$ \begin{diagram} \node{M_a=N_0} \arrow{e,t}{L_{01}}
\node{N_1} \arrow{e,t}{L_{12}} 
\node{\ldots} \arrow{e,t}{L_{(r-1)r}} \node{N_r = M_b}
\end{diagram} .$$
As in Definition~\ref{def:gen adm lag} we call a generalized Lagrangian correspondence $\ul{L}$ from $M_a$ to $M_b$ {\em admissible} if each $N_i$ satisfies (M1-2) with the
monotonicity constant $\tau\geq 0$, each $L_{(i-1)i}$ satisfies (L1-3), and
the image of each $\pi_1(L_{(i-1)i})$ in $\pi_1(N_{i-1}^- \times N_i)$ is torsion.


\begin{definition} \label{dfn corcat}
The {\em Donaldson-Fukaya category of correspondences}
$$\Don^\#(M_a,M_b) := \Don^\#(M_a,M_b,\Lag^N(M_a),\Lag^N(M_b),\omega_a,\omega_b,b_a,b_b) $$
is defined as follows:
\begin{enumerate}
\item 
The objects of $\Don^\#(M_a,M_b)$ are admissible generalized
Lagrangian correspondences from $M_a$ to $M_b$, equipped with
orientations, gradings, and relative spin structures.\footnote{ In the previous notation, a
grading on $\ul{L}$ is a collection of $N$-fold Maslov covers
$\Lag^N(N_j) \to N_j$ for $j = 0,\ldots,r$ and gradings of the
Lagrangian correspondences $L_{(j-1)j}$.  Here the gradings on
$N_0=M_a$ and $N_r=M_b$ are the fixed ones.  A relative spin structure
on $\ul{L}$ is a collection of background classes $b_j \in
H^2(N_j,\Z_2)$ for $j = 0,\ldots,r$ and relative spin structures on
$L_{(j-1)j}$ with background classes $-\pi_{j-1}^* b_{j-1} + \pi_{j}^*
b_{j}$.  Here $b_0=b_a$ and $b_r=b_b$ are the fixed background classes
in $M_a$ and $M_b$. See \cite{quiltfloer} for more details. }
\item
The morphism spaces of $\Don^\#(M_a,M_b)$ are the $\Z_N$-graded Floer
cohomology groups (defined in \cite{quiltfloer})
%
\begin{align*} 
\Hom(\ul{L},\ul{L}') &:= HF(\ul{L},\ul{L}') [d],
\end{align*}
where the second group is shifted by $d = \hh ( \sum_k \dim(N_k) +
\sum_{k'} \dim(N'_{k'})) $.  
For $\Z$-coefficients one has to introduce determinant lines as in 
Remark \ref{zmod}.
See Figure \ref{corhom} for views of the quilted holomorphic cylinders 
which are counted (modulo $\R$-shift) as Floer
trajectories.
\item
The composition of morphisms in $\Don^\#(M_a,M_b)$,
\begin{align*} 
\Hom(\ul{L},\ul{L}') \times  \Hom(\ul{L}',\ul{L}'') &\longrightarrow \;\;\Hom(\ul{L},\ul{L}'') \\ 
 (f,g)\qquad\qquad\quad &\longmapsto  f\circ g:= \Phi_{\ul{P}}(f\otimes g) 
\end{align*}
is defined by the relative invariant $\Phi_{\ul{P}}$ associated to the
quilted pair of pants surface $\ul{P}$ (this time the pair of pants is
an honest one, not just the front) in Figure \ref{quiltpants2}, where
the patches without outgoing ends are ordered from $M_a$ to $M_b$.
\end{enumerate}
\end{definition}

\begin{figure}[ht]
\begin{picture}(0,0)%
\includegraphics{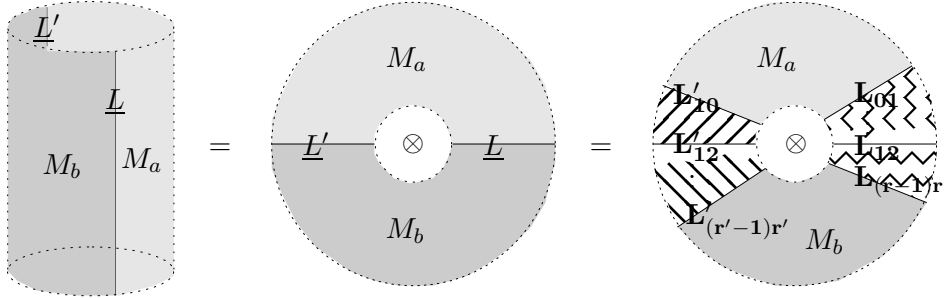}%
\end{picture}%
\setlength{\unitlength}{2693sp}%
\begingroup\makeatletter\ifx\SetFigFont\undefined%
\gdef\SetFigFont#1#2#3#4#5{%
  \reset@font\fontsize{#1}{#2pt}%
  \fontfamily{#3}\fontseries{#4}\fontshape{#5}%
  \selectfont}%
\fi\endgroup%
\begin{picture}(9429,3399)(7950,-3850)
\put(8201,-1511){\makebox(0,0)[lb]{$\ul{L}'$}}
\put(8861,-2176){\makebox(0,0)[lb]{$\ul{L}$}}
\put(8281,-2761){\makebox(0,0)[lb]{$M_b$}}
\put(9006,-2701){\makebox(0,0)[lb]{$M_a$}}
\put(11441,-3356){\makebox(0,0)[lb]{$M_b$}}
\put(11441,-1771){\makebox(0,0)[lb]{$M_a$}}
\put(12341,-2600){\makebox(0,0)[lb]{$\ul{L}$}}
\put(10676,-2600){\makebox(0,0)[lb]{$\ul{L}'$}}
\put(11590,-2530){\makebox(0,0)[lb]{$\otimes$}}
\put(15301,-3456){\makebox(0,0)[lb]{$M_b$}}
\put(14851,-1771){\makebox(0,0)[lb]{$M_a$}}
\put(13321,-2536){\makebox(0,0)[lb]{$=$}}
\put(9811,-2536){\makebox(0,0)[lb]{$=$}}
\put(14086,-2600){\makebox(0,0)[lb]{$\mathbf{L_{12}'}$}}
\put(14086,-2150){\makebox(0,0)[lb]{$\mathbf{L_{10}'}$}}
\put(14201,-3321){\makebox(0,0)[lb]{$\mathbf{L_{(r'-1)r'}'}$}}
\put(15751,-2936){\makebox(0,0)[lb]{$\mathbf{L_{(r-1)r}}$}}
\put(15751,-2100){\makebox(0,0)[lb]{$\mathbf{L_{01}}$}}
\put(14226,-2851){\makebox(0,0)[lb]{$\mathbf{\vdots}$}}
\put(15841,-2366){\makebox(0,0)[lb]{$\mathbf{\vdots}$}}
\put(15751,-2586){\makebox(0,0)[lb]{$\mathbf{L_{12}}$}}
\put(15115,-2530){\makebox(0,0)[lb]{$\otimes$}}
\end{picture}
\caption{Floer trajectories for pairs of generalized Lagrangian correspondences} 
\label{corhom}
\end{figure}

\begin{convention}
In Figure~\ref{corhom} and 
the following pictures, the outer circles
will always be outgoing ends. The inner circles are usually incoming
ends, indicated by a $\otimes$ or marked with the incoming morphism.
Ends at the top resp.\ bottom of pictures will always be outgoing
resp.\ incoming, unless otherwise indicated by arrows.
\end{convention}

\begin{figure} 
\begin{picture}(0,0)%
\includegraphics{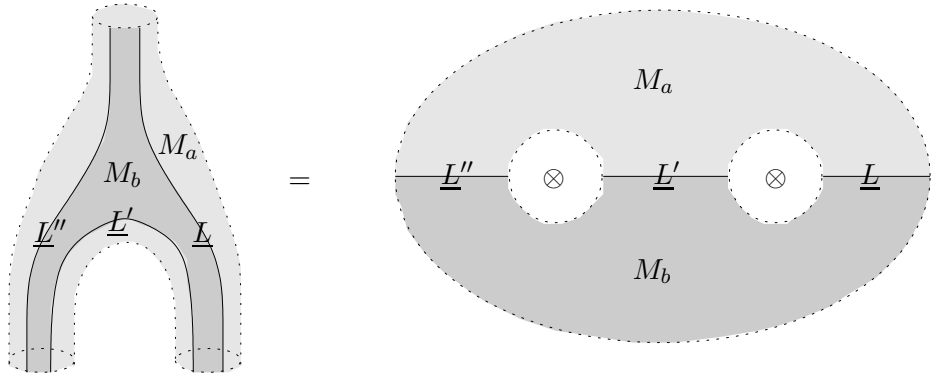}%
\end{picture}%
\setlength{\unitlength}{3274sp}%
\begingroup\makeatletter\ifx\SetFigFont\undefined%
\gdef\SetFigFont#1#2#3#4#5{%
  \reset@font\fontsize{#1}{#2pt}%
  \fontfamily{#3}\fontseries{#4}\fontshape{#5}%
  \selectfont}%
\fi\endgroup%
\begin{picture}(6999,2809)(7549,-3943)
\put(8271,-2526){\makebox(0,0)[lb]{$M_b$}}
\put(8686,-2311){\makebox(0,0)[lb]{$M_a$}}
\put(8306,-2901){\makebox(0,0)[lb]{$\ul{L}'$}}
\put(7751,-2986){\makebox(0,0)[lb]{$\ul{L}''$}}
\put(11600,-2536){\makebox(0,0)[lb]{$\otimes$}}
\put(8951,-2986){\makebox(0,0)[lb]{$\ul{L}$}}
\put(14006,-2550){\makebox(0,0)[lb]{$\ul{L}$}}
\put(13281,-2536){\makebox(0,0)[lb]{$\otimes$}}
\put(10831,-2550){\makebox(0,0)[lb]{$\ul{L}''$}}
\put(12441,-2550){\makebox(0,0)[lb]{$\ul{L}'$}}
\put(12286,-1816){\makebox(0,0)[lb]{$M_a$}}
\put(12286,-3256){\makebox(0,0)[lb]{$M_b$}}
\put(9676,-2536){\makebox(0,0)[lb]{$=$}}
\end{picture}%
\caption{Quilted pair of pants:
Composition of morphisms for Lagrangian correspondences} \label{quiltpants2}
\end{figure}

\begin{remark}
\begin{enumerate}
\item
The identity $ 1_{\ul{L}}\in \Hom(\ul{L},\ul{L})$ for a generalized
Lagrangian correspondence $\ul{L}$ is given by the relative invariant
$1_{\ul{L}}:=\Phi_{\ul{S}}$ associated to the quilted cap in Figure
\ref{quiltcap}, where the patches without outgoing ends are ordered 
from $M_b$ to $M_a$.
\item
The associativity and identity axiom for $\Don^\#(M_a,M_b)$ follow from the quilted gluing theorem \cite[Theorem 3.13]{quilts}
applied to the gluings (indicated by dashed lines) in Figure \ref{cor
axioms}.  Note that -- in contrast to Figure \ref{corhom} -- the
solutions on the quilted annulus (i.e.\ cylinder) are counted without
quotienting by $\R$, hence as in the strip example \cite[Example 2.5]{quilts}
this relative invariant is the identity.
\item
$\Don^\#(M_a,M_b)$ is a small category by the same arguments as in Remark~\ref{small}~(c).
\end{enumerate}
\end{remark}

\begin{figure} 
\begin{picture}(0,0)%
\includegraphics{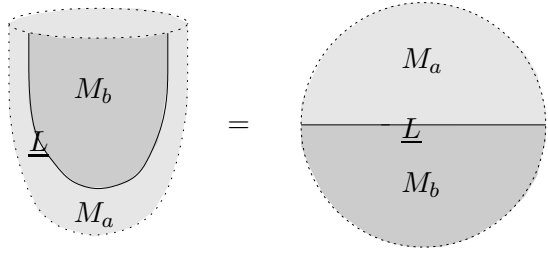}%
\end{picture}%
\setlength{\unitlength}{2735sp}%
\begingroup\makeatletter\ifx\SetFigFont\undefined%
\gdef\SetFigFont#1#2#3#4#5{%
  \reset@font\fontsize{#1}{#2pt}%
  \fontfamily{#3}\fontseries{#4}\fontshape{#5}%
  \selectfont}%
\fi\endgroup%
\begin{picture}(4891,2234)(7774,-3585)
\put(11341,-2636){\makebox(0,0)[lb]{$\ul{L}$}}
\put(9766,-2536){\makebox(0,0)[lb]{$=$}}
\put(8371,-3391){\makebox(0,0)[lb]{$M_a$}}
\put(7966,-2761){\makebox(0,0)[lb]{$\ul{L}$}}
\put(8371,-2266){\makebox(0,0)[lb]{$M_b$}}
\put(11341,-3121){\makebox(0,0)[lb]{$M_b$}}
\put(11341,-1996){\makebox(0,0)[lb]{$M_a$}}
\end{picture}%
\caption{Quilted cap: Identity for Lagrangian correspondences}
\label{quiltcap}
\end{figure}

\begin{figure}[htbp]
\begin{picture}(0,0)%
\includegraphics{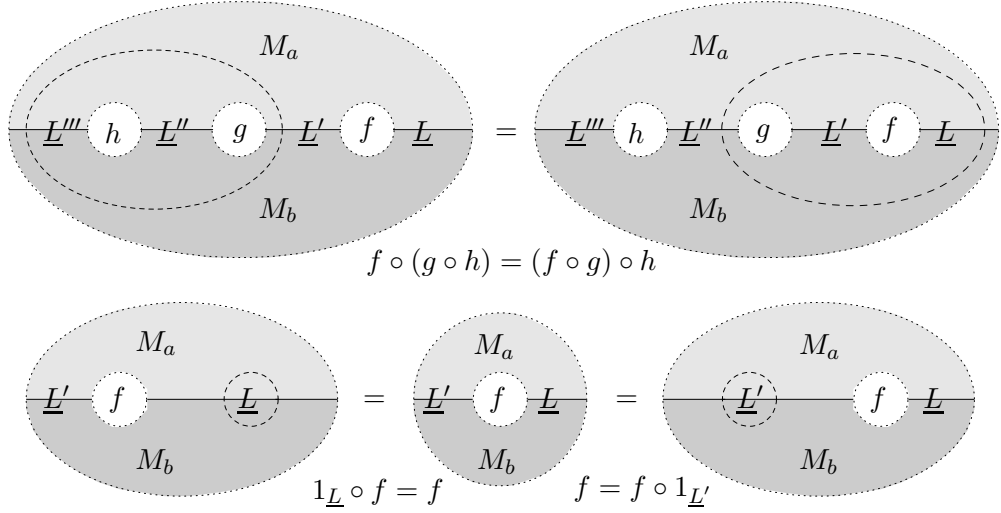}%
\end{picture}%
\setlength{\unitlength}{1906sp}%
\begingroup\makeatletter\ifx\SetFigFont\undefined%
\gdef\SetFigFont#1#2#3#4#5{%
  \reset@font\fontsize{#1}{#2pt}%
  \fontfamily{#3}\fontseries{#4}\fontshape{#5}%
  \selectfont}%
\fi\endgroup%
\begin{picture}(12354,6635)(10339,-3899)
\put(14901,850){\makebox(0,0)[lb]{$f$}}
\put(13271,850){\makebox(0,0)[lb]{$g$}}
\put(11606,850){\makebox(0,0)[lb]{$h$}}
\put(13601,1934){\makebox(0,0)[lb]{$M_a$}}
\put(13601,-186){\makebox(0,0)[lb]{$M_b$}}
\put(10801,800){\makebox(0,0)[lb]{$\ul{L}'''$}}
\put(12281,800){\makebox(0,0)[lb]{$\ul{L}''$}}
\put(14120,800){\makebox(0,0)[lb]{$\ul{L}'$}}
\put(15606,800){\makebox(0,0)[lb]{$\ul{L}$}}
\put(21701,850){\makebox(0,0)[lb]{$f$}}
\put(20071,850){\makebox(0,0)[lb]{$g$}}
\put(18406,850){\makebox(0,0)[lb]{$h$}}
\put(19201,1934){\makebox(0,0)[lb]{$M_a$}}
\put(19201,-186){\makebox(0,0)[lb]{$M_b$}}
\put(17601,800){\makebox(0,0)[lb]{$\ul{L}'''$}}
\put(19081,800){\makebox(0,0)[lb]{$\ul{L}''$}}
\put(20920,800){\makebox(0,0)[lb]{$\ul{L}'$}}
\put(22406,800){\makebox(0,0)[lb]{$\ul{L}$}}
\put(12000,-1916){\makebox(0,0)[lb]{$M_a$}}
\put(12000,-3456){\makebox(0,0)[lb]{$M_b$}}
\put(10800,-2681){\makebox(0,0)[lb]{$\ul{L}'$}}
\put(13336,-2681){\makebox(0,0)[lb]{$\ul{L}$}}
\put(19826,-2681){\makebox(0,0)[lb]{$\ul{L}'$}}
\put(20631,-1961){\makebox(0,0)[lb]{$M_a$}}
\put(20631,-3456){\makebox(0,0)[lb]{$M_b$}}
\put(16446,-3456){\makebox(0,0)[lb]{$M_b$}}
\put(16401,-1961){\makebox(0,0)[lb]{$M_a$}}
\put(15756,-2681){\makebox(0,0)[lb]{$\ul{L}'$}}
\put(17250,-2681){\makebox(0,0)[lb]{$\ul{L}$}}
\put(22251,-2681){\makebox(0,0)[lb]{$\ul{L}$}}
\put(14961,-2581){\makebox(0,0)[lb]{$=$}}
\put(18246,-2581){\makebox(0,0)[lb]{$=$}}
\put(16600,-2666){\makebox(0,0)[lb]{$f$}}
\put(21550,-2666){\makebox(0,0)[lb]{$f$}}
\put(11650,-2666){\makebox(0,0)[lb]{$f$}}
\put(15000,-900){\makebox(0,0)[lb]{$f\circ(g\circ h)=(f\circ g)\circ h$}}
\put(16680,900){\makebox(0,0)[lb]{$=$}}
\put(14300,-3900){\makebox(0,0)[lb]{$1_{\ul{L}}\circ f=f$}}
\put(17700,-3900){\makebox(0,0)[lb]{$f= f\circ 1_{\ul{L}'}$}} 
\end{picture}%
\caption{Axioms for Donaldson-Fukaya category of correspondences}
\label{cor axioms}
\end{figure}

\medskip

\begin{remark} \label{ring iso}
Consider the case where the symplectic manifolds $M_a=M_b=M$ agree
(including Maslov cover and background class).  Then for any
admissible generalized Lagrangian correspondence
$\ul{L}\in\Obj(\Don^\#(M,M))$ the composition of morphisms in (c)
defines a ring structure on $\Hom(\ul{L},\ul{L})$, and (d) provides an
identity element.  Another application of the strip shrinking theorems shows that
this ring structure is isomorphic under embedded compositions of
correspondences: Let $\ul{L}$ and $\ul{L}'$ be two admissible
generalized Lagrangian correspondences from $M$ to itself.  Suppose
that they are equivalent in the sense of Section~\ref{symp cat}
through a series of embedded compositions of consecutive Lagrangian
correspondences, and such that each intermediate generalized
Lagrangian correspondence is admissible.  Then there is a canonical
ring isomorphism
$
\bigl(\Hom(\ul{L},\ul{L}),\circ\bigr) \simeq \bigl(\Hom(\ul{L}',\ul{L}'),\circ\bigr) 
$
which intertwines the identity elements $1_{\ul{L}}$ and
$1_{\ul{L}'}$.  

Indeed, by assumption there exists a sequence of admissible
generalized Lagrangian correspondences $\ul{L}^{j}$ connecting
$\ul{L}^0=\ul{L}$ to $\ul{L}^N=\ul{L}'$ as in the proof of
Theorem~\ref{thm compiso}.  In each step two consecutive Lagrangian
correspondences in the sequence $\ul{L}^j=(\ldots,L_-,L_+,\ldots)$ are
replaced by their embedded, monotone composition in $\ul{L}^{j\pm
1}=(\ldots,L_-\circ L_+,\ldots)$.  Theorem \ref{main2} provides
isomorphisms $\Psi^j:HF(\ul{L}^j,\ul{L}^j)\to HF(\ul{L}^{j\pm
1},\ul{L}^{j\pm 1})$ by shrinking the strip between $L_-$ and $L_+$.
Theorem \ref{intertwine} applies to the corresponding strips in the pair of
pants surface and the quilted cap surface of Definition~\ref{dfn
corcat}~(c) and (d) and shows that the isomorphisms $\Psi^j$
intertwine the ring structures and identity morphisms.  The full ring
isomorphism is given by a composition of these isomorphisms or their
inverses.
\end{remark}  

Next, consider a triple of symplectic manifolds $M_a,M_b,M_c$
satisfying (M1-2) with the same monotonicity constant $\tau$, equipped
with Maslov covers $\Lag^N(M_{(\cdot)}) \to M_{(\cdot)}$ (with the
same $N$) and background classes $b_{(\cdot)} \in
H^2(M_{(\cdot)},\Z_2)$.  
We denote by $ \Don^\#(M_a,M_b) \times \Don^\#(M_b,M_c)$ the product
category.  That is, objects are pairs $(\ul{L}_{ab},\ul{L}_{bc})$ of
objects of $ \Don^\#(M_a,M_b)$ and $ \Don^\#(M_b,M_c)$.  Morphisms are
pairs $(f,g)$ with $f \in \Hom(\ul{L}_{ab},\ul{L}_{ab}'), g \in
\Hom(\ul{L}_{bc},\ul{L}_{bc}')$.  Composition is given by
$$ (f,g) \circ (f',g') := (-1)^{|f'||g|} (f \circ f', g \circ g') $$
for $f \in
\Hom(\ul{L}_{ab},\ul{L}_{ab}'),f' \in
\Hom(\ul{L}_{ab}',\ul{L}_{ab}''), g \in
\Hom(\ul{L}_{bc},\ul{L}_{bc}'), g' \in
\Hom(\ul{L}_{bc}',\ul{L}_{bc}'')$.

\begin{definition}
The {\em composition functor}
\begin{equation} \label{concat}
\#: \Don^\#(M_a,M_b) \times \Don^\#(M_b,M_c) \to \Don^\#(M_a,M_c) 
\end{equation}
is defined as follows.
\begin{enumerate}
\item
On the level of objects $\#$ is defined by concatenation:
\begin{align*}
\Obj(\Don^\#(M_a,M_b)) \times \Obj(\Don^\#(M_b,M_c)) 
&\to
\Obj(\Don^\#(M_a,M_c)) \\
(\ul{L}_{ab} , \ul{L}_{bc}) &\mapsto \ul{L}_{ab} \# \ul{L}_{bc} , 
\end{align*}
where 
$$ (L_{01}^{ab},\ldots,L_{(r-1)r}^{ab}) \# (L_{01}^{bc},\ldots,L_{(r'-1)r'}^{bc}) 
 := (L_{01}^{ab},\ldots,L_{(r-1)r}^{ab},L_{01}^{bc},\ldots,L_{(r'-1)r'}^{bc}) .$$
\item
On the level of morphisms, $\#$ is defined for
$\ul{L}_{ab},\ul{L}_{ab}'\in \Obj(\Don^\#(M_a,M_b))$
and $\ul{L}_{bc},\ul{L}_{bc}'\in \Obj(\Don^\#(M_b,M_c))$
by
\begin{align*}
\Hom(\ul{L}_{ab},\ul{L}_{ab}') \times \Hom(\ul{L}_{bc},\ul{L}_{bc}')
&\to \Hom(\ul{L}_{ab}\#\ul{L}_{bc} , \ul{L}_{ab}'\#\ul{L}_{bc}') \\
(f, g ) \qquad\qquad &\mapsto f \# g := 
\Phi_{\ul{P}}(f\otimes g) ,
\end{align*}
where $\Phi_{\ul{P}}$ is the relative invariant associated to the
quilted pair of pants $\ul{P}$, where now every seam connects one of
the incoming cylindrical ends to the outgoing cylindrical end, as in
Figure \ref{quiltpants3}.  
\end{enumerate}
\end{definition}

\begin{figure}[ht]
\begin{picture}(0,0)%
\includegraphics{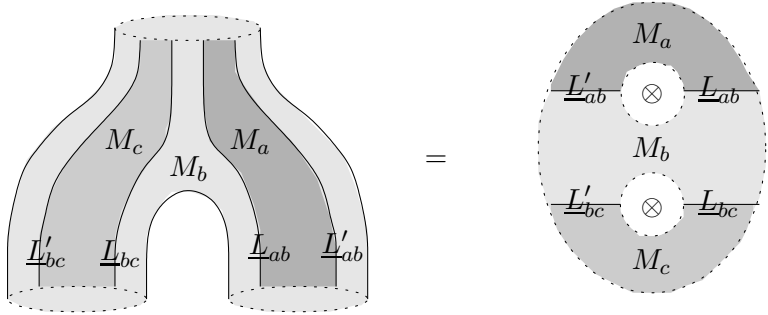}%
\end{picture}%
\setlength{\unitlength}{3481sp}%
\begingroup\makeatletter\ifx\SetFigFont\undefined%
\gdef\SetFigFont#1#2#3#4#5{%
  \reset@font\fontsize{#1}{#2pt}%
  \fontfamily{#3}\fontseries{#4}\fontshape{#5}%
  \selectfont}%
\fi\endgroup%
\begin{picture}(5420,2220)(7369,-5648)
\put(8531,-4711){\makebox(0,0)[lb]{$M_b$}}
\put(9616,-5281){\makebox(0,0)[lb]{$\ul{L}_{ab}'$}}
\put(8040,-5326){\makebox(0,0)[lb]{$\ul{L}_{bc}$}}
\put(7511,-5326){\makebox(0,0)[lb]{$\ul{L}_{bc}'$}}
\put(9100,-5281){\makebox(0,0)[lb]{$\ul{L}_{ab}$}}
\put(10351,-4606){\makebox(0,0)[lb]{$=$}}
\put(11890,-4156){\makebox(0,0)[lb]{$\otimes$}}
\put(11890,-4966){\makebox(0,0)[lb]{$\otimes$}}
\put(11826,-4571){\makebox(0,0)[lb]{$M_b$}}
\put(11826,-5371){\makebox(0,0)[lb]{$M_c$}}
\put(11826,-3761){\makebox(0,0)[lb]{$M_a$}}
\put(11350,-4966){\makebox(0,0)[lb]{$\ul{L}'_{bc}$}}
\put(12300,-4966){\makebox(0,0)[lb]{$\ul{L}_{bc}$}}
\put(11350,-4156){\makebox(0,0)[lb]{$\ul{L}_{ab}'$}}
\put(12300,-4156){\makebox(0,0)[lb]{$\ul{L}_{ab}$}}
\put(8070,-4516){\makebox(0,0)[lb]{$M_c$}}
\put(8961,-4516){\makebox(0,0)[lb]{$M_a$}}
\end{picture}%
\caption{Composition functor on Donaldson categories of correspondences}
\label{quiltpants3}

\end{figure}

The composition axiom for the functor $\#$ follows from the quilted gluing theorem \cite[Theorem 3.13]{quilts} applied to the two degenerations of the
five-holed sphere shown in Figure \ref{threepants}:  For all
$f \in \Hom(\ul{L}_{ab},\ul{L}_{ab}'),f' \in
\Hom(\ul{L}_{ab}',\ul{L}_{ab}''), g \in \Hom(\ul{L}_{bc},\ul{L}_{bc}'), 
g' \in \Hom(\ul{L}_{bc}',\ul{L}_{bc}'')$
we obtain
$$ \#\bigl( (f,g)\circ(f',g') \bigr) =
(-1)^{|f'||g|} 
(f \circ f') \# (g \circ g') = 
(f \# g) \circ (f' \# g') .$$
The identity axiom for the concatenation functor, 
$1_{\ul{L}_{ab}} \# 1_{\ul{L}_{bc}} =
1_{\ul{L}_{ab}\#\ul{L}_{bc}}$, follows similarly from the quilted gluing theorem applied to the degenerations shown in
Figure~\ref{morepants}.

\begin{figure}[htb]
\begin{picture}(0,0)%
\includegraphics{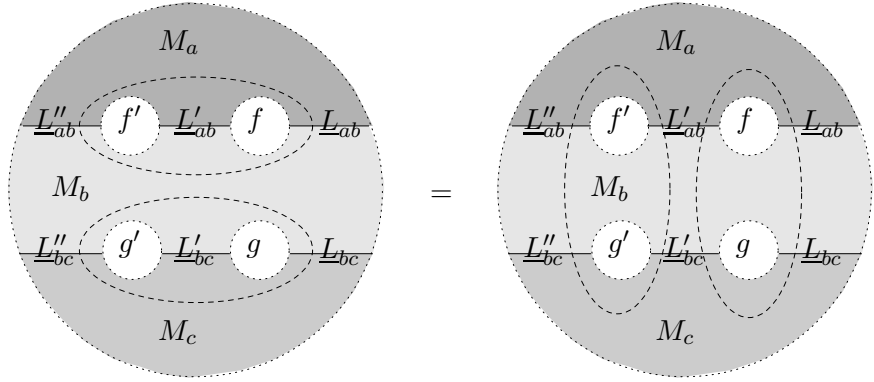}%
\end{picture}%
\setlength{\unitlength}{2072sp}%
\begingroup\makeatletter\ifx\SetFigFont\undefined%
\gdef\SetFigFont#1#2#3#4#5{%
  \reset@font\fontsize{#1}{#2pt}%
  \fontfamily{#3}\fontseries{#4}\fontshape{#5}%
  \selectfont}%
\fi\endgroup%
\begin{picture}(10352,4502)(11700,-1952)
\put(13006,974){\makebox(0,0)[lb]{$f'$}}
\put(14560,971){\makebox(0,0)[lb]{$f$}}
\put(14556,-514){\makebox(0,0)[lb]{$g$}}
\put(13030,-514){\makebox(0,0)[lb]{$g'$}}
\put(12016,929){\makebox(0,0)[lb]{$\ul{L}_{ab}''$}}
\put(12016,-601){\makebox(0,0)[lb]{$\ul{L}_{bc}''$}}
\put(15426,-601){\makebox(0,0)[lb]{$\ul{L}_{bc}$}}
\put(15426,929){\makebox(0,0)[lb]{$\ul{L}_{ab}$}}
\put(13696,929){\makebox(0,0)[lb]{$\ul{L}_{ab}'$}}
\put(13696,-601){\makebox(0,0)[lb]{$\ul{L}_{bc}'$}}
\put(13496,-1546){\makebox(0,0)[lb]{$M_c$}}
\put(13496,1919){\makebox(0,0)[lb]{$M_a$}}
\put(12216,164){\makebox(0,0)[lb]{$M_b$}}
\put(18856,974){\makebox(0,0)[lb]{$f'$}}
\put(20410,971){\makebox(0,0)[lb]{$f$}}
\put(20406,-514){\makebox(0,0)[lb]{$g$}}
\put(18880,-514){\makebox(0,0)[lb]{$g'$}}
\put(17866,929){\makebox(0,0)[lb]{$\ul{L}_{ab}''$}}
\put(17866,-601){\makebox(0,0)[lb]{$\ul{L}_{bc}''$}}
\put(21196,-601){\makebox(0,0)[lb]{$\ul{L}_{bc}$}}
\put(21196,929){\makebox(0,0)[lb]{$\ul{L}_{ab}$}}
\put(19546,929){\makebox(0,0)[lb]{$\ul{L}_{ab}'$}}
\put(19546,-601){\makebox(0,0)[lb]{$\ul{L}_{bc}'$}}
\put(19446,-1546){\makebox(0,0)[lb]{$M_c$}}
\put(19446,1919){\makebox(0,0)[lb]{$M_a$}}
\put(18666,164){\makebox(0,0)[lb]{$M_b$}}
\put(16751,164){\makebox(0,0)[lb]{$=$}}
\end{picture}%
\caption{Composition axiom for the concatenation functor}
\label{threepants}
\end{figure}

\begin{figure}[ht]
\begin{picture}(0,0)%
\includegraphics{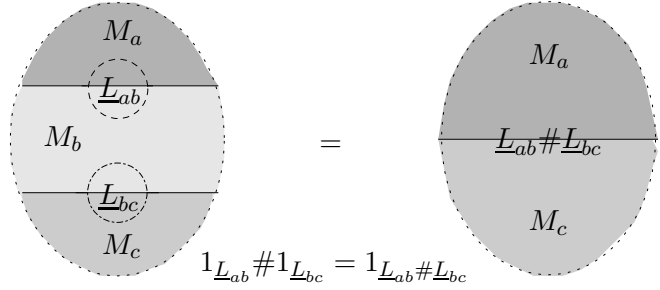}%
\end{picture}%
\setlength{\unitlength}{3274sp}%
\begingroup\makeatletter\ifx\SetFigFont\undefined%
\gdef\SetFigFont#1#2#3#4#5{%
  \reset@font\fontsize{#1}{#2pt}%
  \fontfamily{#3}\fontseries{#4}\fontshape{#5}%
  \selectfont}%
\fi\endgroup%
\begin{picture}(4925,2189)(10163,-5609)
\put(10866,-5371){\makebox(0,0)[lb]{$M_c$}}
\put(10841,-4206){\makebox(0,0)[lb]{$\ul{L}_{ab}$}}
\put(10841,-5016){\makebox(0,0)[lb]{$\ul{L}_{bc}$}}
\put(10866,-3751){\makebox(0,0)[lb]{$M_a$}}
\put(10421,-4561){\makebox(0,0)[lb]{$M_b$}}
\put(13841,-4591){\makebox(0,0)[lb]{$\ul{L}_{ab}\#\ul{L}_{bc}$}}
\put(14096,-5191){\makebox(0,0)[lb]{$M_c$}}
\put(14096,-3931){\makebox(0,0)[lb]{$M_a$}}
\put(12511,-4561){\makebox(0,0)[lb]{$=$}}
\put(11606,-5551){\makebox(0,0)[lb]{$1_{\ul{L}_{ab}} \# 1_{\ul{L}_{bc}} = 1_{\ul{L}_{ab}\#\ul{L}_{bc}}$}}
\end{picture}%
\caption{Identity axiom for the concatenation functor}
\label{morepants}
\end{figure}

\begin{remark}
The construction of functors associated to Lagrangian correspondences
in Section~\ref{Funk} has an obvious extension (\ref{funklong}) for
generalized Lagrangian correspondences.  For
$\ul{L}_{ab}\in\Don^\#(M_a,M_b)$ the functor
$\Phi(\ul{L}_{ab}):\Don^\#(M_a)\to\Don^\#(M_b)$ acts on objects
$\ul{L}\in\Obj(\Don^\#(M_a))$ by concatenation
$\Phi(\ul{L}_{ab})=\ul{L}\#\ul{L}_{ab}$, and on morphisms
$\Phi(\ul{L}_{ab}):HF(\ul{L},\ul{L}')\to
HF(\ul{L}\#\ul{L}_{ab},\ul{L}'\#\ul{L}_{ab})$ is defined by
composition $\Phi(L_{01})\circ\ldots \circ \Phi(L_{(r-1)r})$ of the
functors associated to the elementary Lagrangian correspondences
$(L_{01},\ldots,L_{(r-1)r})=\ul{L}_{ab}$.  Alternatively, the map
$\Phi(\ul{L}_{ab})$ on morphisms can be defined directly by the
relative invariant in Figure~\ref{funkcomp}, see Remark~\ref{rmk
funkcomp}.  Using the first definition, we have a tautological
equality of functors
\begin{equation} \label{concateq}
\Phi(\ul{L}_{ab}) \circ \Phi(\ul{L}_{bc})
= \Phi(\ul{L}_{ab} \# \ul{L}_{bc}) 
\end{equation}
for any two objects $\ul{L}_{ab} \in \Don^\#(M_a,M_b)$ and
$\ul{L}_{bc} \in \Don^\#(M_b,M_c)$.
\end{remark}

\section{Natural transformations associated to Floer cohomology classes}  
\label{nattran}

Let $M_a$ and $M_b$ be as in the previous section and
let $\ul{L}_{ab},\ul{L}_{ab}'$ be objects in $\Don^\#(M_a,M_b)$.  

\begin{definition}
Given a morphism $T \in \Hom(\ul{L}_{ab},\ul{L}_{ab}')$ we define a 
natural transformation
$$\Phi_T: \ \Phi(\ul{L}_{ab}) \to \Phi(\ul{L}_{ab}') $$ 
as follows: To any object $\ul{L}$ in $\Don^\#(M_a)$ we assign the morphism
$$ \Phi_T(\ul{L}) \in \Hom(\Phi(\ul{L}_{ab})(\ul{L}),\Phi(\ul{L}_{ab}')(\ul{L})) $$
given by the relative invariant associated to the surface in Figure
\ref{natural}, which is independent of the ordering of the patches.
(Note that the end where $T$ is inserted is cylindrical in the sense
that the strip-like ends glue together to a cylindrical end.)
\end{definition}

\begin{figure}[ht]
\begin{picture}(0,0)%
\includegraphics{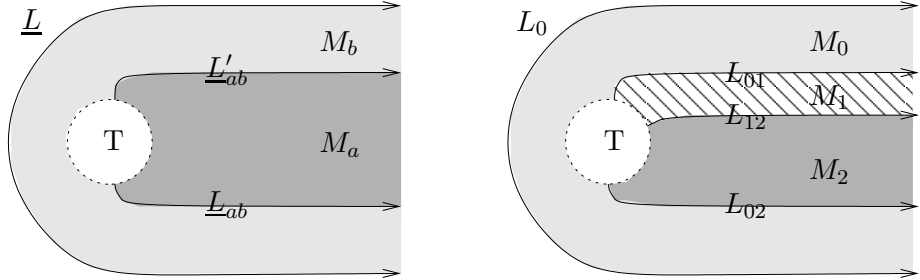}%
\end{picture}%
\setlength{\unitlength}{3356sp}%
\begingroup\makeatletter\ifx\SetFigFont\undefined%
\gdef\SetFigFont#1#2#3#4#5{%
  \reset@font\fontsize{#1}{#2pt}%
  \fontfamily{#3}\fontseries{#4}\fontshape{#5}%
  \selectfont}%
\fi\endgroup%
\begin{picture}(6750,2063)(-9212,-553)
\put(-3284,1100){\makebox(0,0)[lb]{$M_0$}}
\put(-3284,700){\makebox(0,0)[lb]{$M_1$}}
\put(-3285,150){\makebox(0,0)[lb]{$M_2$}}
\put(-3914,-100){\makebox(0,0)[lb]{$L_{02}$}}
\put(-3914,870){\makebox(0,0)[lb]{$L_{01}$}}
\put(-3914,560){\makebox(0,0)[lb]{$L_{12}$}}
\put(-5450,1250){\makebox(0,0)[lb]{$L_0$}}
\put(-8500,400){\makebox(0,0)[lb]{T}}
\put(-4800,400){\makebox(0,0)[lb]{T}}
\put(-6900,350){\makebox(0,0)[lb]{$M_a$}}
\put(-6900,1100){\makebox(0,0)[lb]{$M_b$}}
\put(-9100,1250){\makebox(0,0)[lb]{$\ul{L}$}}
\put(-7739,880){\makebox(0,0)[lb]{$\ul{L}'_{ab}$}}
\put(-7739,-100){\makebox(0,0)[lb]{$\ul{L}_{ab}$}}
\end{picture}%
\caption{Natural transformation associated to a Floer cohomology class: 
General case and an example, where $\ul{L}$ consists of a single
Lagrangian $L_0$, $\ul{L}_{ab}$ consists of a single
Lagrangian $L_{02}$, and $\ul{L}_{ab}'$ consists of a pair
$(L_{01},L_{12})$.}
\label{natural}
\end{figure}
To see that $\Phi_T$ is a natural transformation of functors
$\Phi(\ul{L}_{ab}) \to \Phi(\ul{L}'_{ab})$ we must show that for any two
objects $\ul{L},\ul{L}'$ in $\Don^\#(M_a)$ and any morphism $f \in
\Hom(\ul{L},\ul{L}')$ we have
\begin{equation}\label{nattr}
\Phi(\ul{L}_{ab})(f) \circ \Phi_T(\ul{L}') = (-1)^{|T||f|} \Phi_T(\ul{L}) \circ \Phi({\ul{L}_{ab}'})(f) .
\end{equation}
This identity follows from the quilted gluing theorem \cite[Theorem 3.13]{quilts} applied to the gluing shown
in Figure \ref{Natfig}.
\begin{figure}[ht]
\begin{picture}(0,0)%
\includegraphics{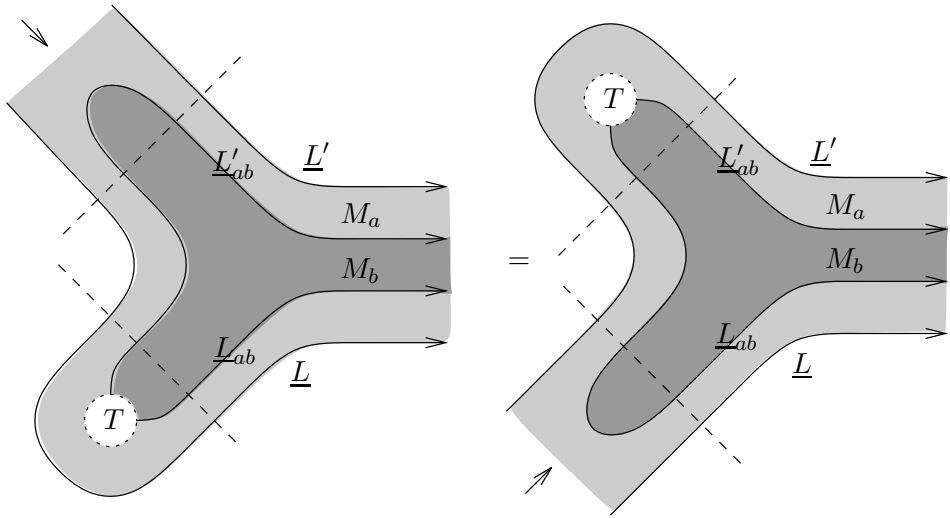}%
\end{picture}%
\setlength{\unitlength}{4144sp}%
\begingroup\makeatletter\ifx\SetFigFont\undefined%
\gdef\SetFigFont#1#2#3#4#5{%
  \reset@font\fontsize{#1}{#2pt}%
  \fontfamily{#3}\fontseries{#4}\fontshape{#5}%
  \selectfont}%
\fi\endgroup%
\begin{picture}(5660,3079)(-11,-3578)
\put(2000,-1840){\makebox(0,0)[lb]{$M_a$}}
\put(2000,-2175){\makebox(0,0)[lb]{$M_b$}}
\put(1776,-1506){\makebox(0,0)[lb]{$\ul{L}'$}}
\put(1700,-2800){\makebox(0,0)[lb]{$\ul{L}$}}
\put(1232,-1552){\makebox(0,0)[lb]{$\ul{L}_{ab}'$}}
\put(1242,-2660){\makebox(0,0)[lb]{$\ul{L}_{ab}$}}
\put(580,-3050){\makebox(0,0)[lb]{$T$}}
\put(3000,-2086){\makebox(0,0)[lb]{$=$}}
\put(4900,-1790){\makebox(0,0)[lb]{$M_a$}}
\put(4900,-2115){\makebox(0,0)[lb]{$M_b$}}
\put(4810,-1480){\makebox(0,0)[lb]{$\ul{L}'$}}
\put(4701,-2750){\makebox(0,0)[lb]{$\ul{L}$}}
\put(4253,-1532){\makebox(0,0)[lb]{$\ul{L}_{ab}'$}}
\put(4243,-2580){\makebox(0,0)[lb]{$\ul{L}_{ab}$}}
\put(3574,-1129){\makebox(0,0)[lb]{$T$}}
\end{picture}%
\caption{Natural transformation axiom}
\label{Natfig}
\end{figure}

\begin{proposition} \label{functor}
The maps $ \ul{L}_{ab} \mapsto \Phi(\ul{L}_{ab})$ and  
$T \mapsto \Phi_T $ define a functor 
$$
\Don^\#(M_a,M_b) \to \Fun(\Don^\#(M_a),\Don^\#(M_b)) .
$$
\end{proposition}

\begin{proof} 
We apply the quilted gluing theorem \cite[Theorem 3.13]{quilts} to the quilted surfaces in Figure \ref{Tcomp} to deduce the composition axiom 
$
\Phi_T(\ul{L}) \circ \Phi_{T'}(\ul{L}) = \Phi_{T \circ T'}(\ul{L})
$
for all $T\in\Hom(\ul{L}_{ab},\ul{L}_{ab}')$,
$T'\in\Hom(\ul{L}_{ab}',\ul{L}_{ab}'')$, and
$\ul{L}\in\Obj(\Don^\#(M_a))$. The identity axiom
$
\Phi_{1_{\ul{L}_{ab}}}(\ul{L})=1_{\Phi(\ul{L}_{ab})(\ul{L})}
$
for $T=1_{\ul{L}_{ab}}\in\Hom(\ul{L}_{ab},\ul{L}_{ab})$ and $\ul{L}\in\Obj(\Don^\#(M_a))$
follows from the quilted gluing theorem applied to the quilted surface in Figure \ref{Tid}.
\end{proof}

\begin{figure}[ht]
\begin{picture}(0,0)%
\includegraphics{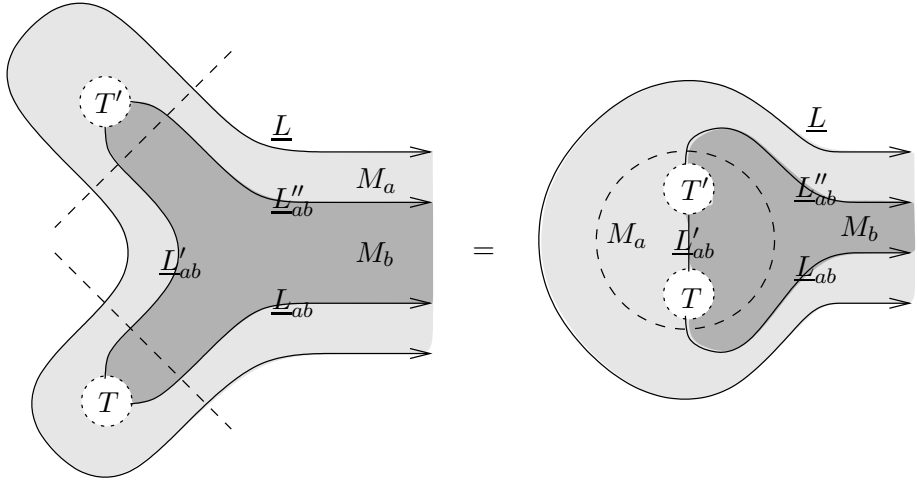}%
\end{picture}%
\setlength{\unitlength}{4144sp}%
\begingroup\makeatletter\ifx\SetFigFont\undefined%
\gdef\SetFigFont#1#2#3#4#5{%
  \reset@font\fontsize{#1}{#2pt}%
  \fontfamily{#3}\fontseries{#4}\fontshape{#5}%
  \selectfont}%
\fi\endgroup%
\begin{picture}(5889,2861)(6,-3374)
\put(561,-2970){\makebox(0,0)[lb]{{{{$T$}}}}}
\put(540,-1180){\makebox(0,0)[lb]{{{{$T'$}}}}}
\put(2100,-1660){\makebox(0,0)[lb]{{{{$M_a$}}}}}
\put(2100,-2100){\makebox(0,0)[lb]{{{{$M_b$}}}}}
\put(1600,-1350){\makebox(0,0)[lb]{{{{$\ul{L}$}}}}}
\put(1600,-1800){\makebox(0,0)[lb]{{{{$\ul{L}_{ab}''$}}}}}
\put(930,-2170){\makebox(0,0)[lb]{{{{$\ul{L}_{ab}'$}}}}}
\put(1600,-2400){\makebox(0,0)[lb]{{{{$\ul{L}_{ab}$}}}}}
\put(2800,-2050){\makebox(0,0)[lb]{{{{$=$}}}}}
\put(3600,-1987){\makebox(0,0)[lb]{{{{$M_a$}}}}}
\put(5000,-1950){\makebox(0,0)[lb]{{{{$M_b$}}}}}
\put(4800,-1300){\makebox(0,0)[lb]{{{{$\ul{L}$}}}}}
\put(4730,-1730){\makebox(0,0)[lb]{{{{$\ul{L}_{ab}''$}}}}}
\put(4730,-2200){\makebox(0,0)[lb]{{{{$\ul{L}_{ab}$}}}}}
\put(4000,-2050){\makebox(0,0)[lb]{{{{$\ul{L}_{ab}'$}}}}}
\put(4050,-2350){\makebox(0,0)[lb]{{{{$T$}}}}}
\put(4050,-1700){\makebox(0,0)[lb]{{{{$T'$}}}}}
\end{picture}%
\caption{Composition axiom for natural transformations}
\label{Tcomp}
\end{figure}
\begin{figure}[ht]
\begin{picture}(0,0)%
\includegraphics{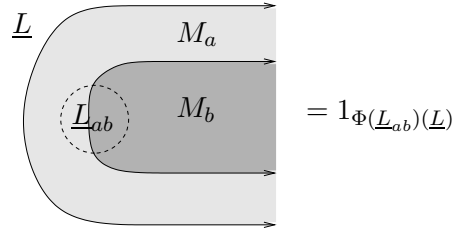}%
\end{picture}%
\setlength{\unitlength}{2196sp}%
\begingroup\makeatletter\ifx\SetFigFont\undefined%
\gdef\SetFigFont#1#2#3#4#5{%
  \reset@font\fontsize{#1}{#2pt}%
  \fontfamily{#3}\fontseries{#4}\fontshape{#5}%
  \selectfont}%
\fi\endgroup%
\begin{picture}(3240,2559)(3601,-3703)
\put(7000,-2600){\makebox(0,0)[lb]{$=1_{\Phi(\ul{L}_{ab})(\ul{L})}$}}
\put(4350,-2600){\makebox(0,0)[lb]{$\ul{L}_{ab}$}}
\put(5536,-2500){\makebox(0,0)[lb]{$M_b$}}
\put(5536,-1646){\makebox(0,0)[lb]{$M_a$}}
\put(3681,-1561){\makebox(0,0)[lb]{$\ul{L}$}}
\end{picture}%
\caption{Identity axiom for natural transformations}
\label{Tid}
\end{figure}
\begin{remark}   \label{inject}  
In this remark we discuss the special case of the diagonal $\Delta
\subset M^- \times M$, which gives rise to the so-called open-closed
maps in $2D$ TQFT.  By \cite{pss} there is a ring isomorphism between
the Floer cohomology of the diagonal $HF(\Delta,\Delta)$ and the quantum
cohomology $HF(\Id)$.  Our construction gives for any element $\alpha \in
HF(\Delta,\Delta) \simeq HF(\Id)$ an automorphism of the identity
functor $\Phi(\Delta)$ (more precisely, of the shift functor
$\Phi(\Delta)\simeq\Psi_M$ in case $w_2(M)\ne 0$).  In particular, we
obtain elements $\Phi_{\alpha}(L) \in HF((L,\Delta),(L,\Delta)) \simeq
HF(L,L)$ for each admissible Lagrangian submanifold $L\subset M$.
(Here $HF((L,\Delta),(L,\Delta)) \simeq HF(L,L)$ is a ring isomorphism
by Remark \ref{ring iso} .)  Proposition \ref{functor} gives
$ \Phi_{\alpha \circ \beta}(L) = \Phi_\alpha(L) \circ \Phi_\beta(L).$
That is, the closed-open map $HF(\Id) \to HF(L,L)$ is a ring
homomorphism.  The closed-open maps in Floer theory are discussed in
more detail in Albers \cite[Theorem 3.1]{alb:ext}.  

For any pair of Lagrangians $L^0,L^1 \subset M$, combining the ring
homomorphism $HF(\Id) \to HF(L^k,L^k)$ with the composition
$HF(L^0,L^0)\times HF(L^0,L^1)\to HF(L^0,L^1)$
resp.\ $HF(L^0,L^1)\times HF(L^1,L^1)\to HF(L^0,L^1)$
gives a module structure on $HF(L^0,L^1)$ over $HF(\Id)$.  The module structure is
independent of $k=0,1$, by the natural transformation axiom
\eqref{nattr} with $\ul{L}_{ab}=\ul{L}'_{ab}=\Delta$.  It is equal to the module
structure induced by the isomorphism $HF(L^0 \times L^1, \Delta) \to
HF(L^0,L^1)$ of \cite{orient}.

Note that if $HF(\Id) \to HF(L,L)$ is a surjection and $HF(\Id)$ is
semisimple then $HF(L,L)$ is again semisimple, and in particular
nilpotent free.
\end{remark}

Next, we show that embedded composition of Lagrangian correspondences
gives rise to isomorphic objects in the Donaldson-Fukaya category.
For simplicity we restrict to the case of elementary Lagrangian
correspondences, i.e.\ sequences of length 1. The statement and
argument for the general case is analogous.

\begin{theorem} \label{main22}
Let $L_{01}\in\Obj(\Don^\#(M_0,M_1))$ and $L_{12} \in\Obj(\Don^\#(M_1,
M_2))$ be admissible Lagrangian correspondences.  Suppose that $L_{01}
\times_{M_1} L_{12}\to M_0^-\times M_2$
is cut out transversally and embeds to a smooth, admissible
Lagrangian correspondence $L_{02}:=L_{01}\circ L_{12} \in
\Obj(\Don^\#(M_0, M_2))$.  Then $\Delta_{M_0} \# L_{02}$, $L_{02} \#
\Delta_{M_2}$, and $L_{01} \# L_{12}$ are all isomorphic in
$\Don^\#(M_0,M_2)$.
\end{theorem} 

\begin{remark}
If in Theorem \ref{main22} we moreover assume $w_2(M_0)=0$ or
$w_2(M_2)=0$, then we in fact have an isomorphism between $L_{01} \#
L_{12}$ and $L_{01}\circ L_{12}$, by Proposition \ref{diagonal} below.
\end{remark}

\begin{proof}   
By Theorem \ref{main2}, $\Hom(L_{01}\# L_{12},\Delta_{M_0} \# L_{02})$
resp.\ $\Hom(\Delta_{M_0}\# L_{02},L_{01} \# L_{12})$ is isomorphic to
$\Hom(\Delta_{M_0}\#L_{02},\Delta_{M_0}\#L_{02})$; let $\phi$ resp.\
$\psi$ denote the inverse image of the identity
$1_{\Delta_{M_0}\#L_{02}}$.  To establish the isomorphism $L_{01}\#
L_{12} \simeq \Delta_{M_0} \# L_{02}$ we show that
$ \psi \circ \phi = 1_{\Delta_{M_0}\#L_{02}}$ and $\phi \circ \psi =
1_{L_{01} \# L_{12}} $
for the composition by the pair of pants products.  These are special
cases of Theorem \ref{intertwine} applied to the degenerations shown
in Figure \ref{shrinkpants}.  The isomorphism $L_{01}\# L_{12} \simeq
L_{02} \# \Delta_{M_2}$ is proven in the same way.
\end{proof}

\begin{figure}
\begin{picture}(0,0)%
\includegraphics{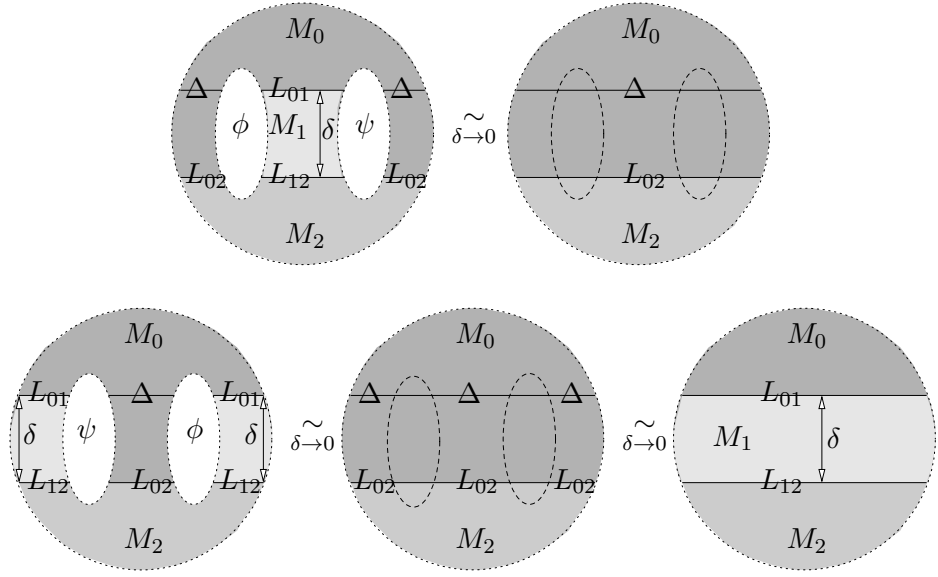}%
\end{picture}%
\setlength{\unitlength}{2403sp}%
\begingroup\makeatletter\ifx\SetFigFont\undefined%
\gdef\SetFigFont#1#2#3#4#5{%
  \reset@font\fontsize{#1}{#2pt}%
  \fontfamily{#3}\fontseries{#4}\fontshape{#5}%
  \selectfont}%
\fi\endgroup%
\begin{picture}(9559,5865)(-1225,-5469)
\put(3341,-1086){\makebox(0,0)[lb]{$\underset{\delta\to 0}{\sim}$}}
\put(1676,-4256){\makebox(0,0)[lb]{$\underset{\delta\to 0}{\sim}$}}
\put(5101,-4256){\makebox(0,0)[lb]{$\underset{\delta\to 0}{\sim}$}}
\put(7210,-4200){\makebox(0,0)[lb]{$\delta$}}
\put(1621,-16){\makebox(0,0)[lb]{$M_0$}}
\put(586,-585){\makebox(0,0)[lb]{$\Delta$}}
\put(2651,-1516){\makebox(0,0)[lb]{$L_{02}$}}
\put(1441,-1516){\makebox(0,0)[lb]{$L_{12}$}}
\put(2701,-586){\makebox(0,0)[lb]{$\Delta$}}
\put(586,-1516){\makebox(0,0)[lb]{$L_{02}$}}
\put(1621,-2131){\makebox(0,0)[lb]{$M_2$}}
\put(1441,-1006){\makebox(0,0)[lb]{$M_1$}}
\put(1996,-1006){\makebox(0,0)[lb]{$\delta$}}
\put(1441,-616){\makebox(0,0)[lb]{$L_{01}$}}
\put(-44,-3166){\makebox(0,0)[lb]{$M_0$}}
\put(-44,-5281){\makebox(0,0)[lb]{$M_2$}}
\put(-1040,-3766){\makebox(0,0)[lb]{$L_{01}$}}
\put(-1045,-4666){\makebox(0,0)[lb]{$L_{12}$}}
\put(24,-4666){\makebox(0,0)[lb]{$L_{02}$}}
\put(24,-3736){\makebox(0,0)[lb]{$\Delta$}}
\put(936,-3766){\makebox(0,0)[lb]{$L_{01}$}}
\put(936,-4666){\makebox(0,0)[lb]{$L_{12}$}}
\put(3376,-3166){\makebox(0,0)[lb]{$M_0$}}
\put(2371,-3736){\makebox(0,0)[lb]{$\Delta$}}
\put(4386,-4666){\makebox(0,0)[lb]{$L_{02}$}}
\put(4456,-3736){\makebox(0,0)[lb]{$\Delta$}}
\put(2321,-4666){\makebox(0,0)[lb]{$L_{02}$}}
\put(3376,-5281){\makebox(0,0)[lb]{$M_2$}}
\put(3396,-3736){\makebox(0,0)[lb]{$\Delta$}}
\put(3376,-4666){\makebox(0,0)[lb]{$L_{02}$}}
\put(6796,-3166){\makebox(0,0)[lb]{$M_0$}}
\put(6796,-5281){\makebox(0,0)[lb]{$M_2$}}
\put(6516,-3766){\makebox(0,0)[lb]{$L_{01}$}}
\put(6516,-4666){\makebox(0,0)[lb]{$L_{12}$}}
\put(5086,-16){\makebox(0,0)[lb]{$M_0$}}
\put(5086,-2131){\makebox(0,0)[lb]{$M_2$}}
\put(5106,-586){\makebox(0,0)[lb]{$\Delta$}}
\put(5106,-1516){\makebox(0,0)[lb]{$L_{02}$}}
\put(1201,-4156){\makebox(0,0)[lb]{$\delta$}}
\put(-1079,-4156){\makebox(0,0)[lb]{$\delta$}}
\put(6031,-4226){\makebox(0,0)[lb]{$M_1$}}
\put(1071,-1006){\makebox(0,0)[lb]{$\phi$}}
\put(-530,-4156){\makebox(0,0)[lb]{$\psi$}}
\put(2341,-1006){\makebox(0,0)[lb]{$\psi$}}
\put(600,-4156){\makebox(0,0)[lb]{$\phi$}}
\end{picture}%
\caption{Isomorphism of composition and concatenation}
\label{shrinkpants}
\end{figure}

\begin{proposition} \label{diagonal}
Suppose that $M_0$ satisfies $w_2(M_0)=0$.  Then the diagonal
$\Delta_{M_0}\in\Don^\#(M_0,M_0)$ is an identity of the composition
$\#$ up to isomorphism. That is, for every generalized Lagrangian
$\ul{L}\in\Obj(\Don^\#(M_0,M_1))$ the objects $\Delta_{M_0}\#\ul{L}$
and $\ul{L}$ are isomorphic in $\Don^\#(M_0,M_1)$, and for every
generalized Lagrangian $\ul{L}\in\Obj(\Don^\#(M_1,M_0))$ the objects
$\ul{L}\#\Delta_{M_0}$ and $\ul{L}$ are isomorphic in
$\Don^\#(M_1,M_0)$.
\end{proposition} 

\begin{proof}   
By Theorem \ref{main2}, both $\Hom(\Delta_{M_0}\# \ul{L},\ul{L})$ and
$\Hom(\ul{L},\Delta_{M_0}\# \ul{L})$ are isomorphic to
$\Hom(\ul{L},\ul{L})$; let $\phi$ resp.\ $\psi$ denote the inverse
image of the identity $1_{\ul{L}}$.  Then the identities $\phi \circ
\psi = 1_{\ul{L}}$ and $\phi \circ \psi = 1_{\Delta_{M_0} \# \ul{L} }$
follow from Theorem \ref{intertwine} applied to the degenerations
shown in Figure \ref{moreshrinkpants}.  (Alternatively, as mentioned
in Section \ref{sym2}, one could glue the strips instead of shrinking
them.)  This proves $\Delta_{M_0}\#\ul{L} \simeq \ul{L}$.  The
isomorphism $\ul{L}\#\Delta_{M_0} \simeq \ul{L}$ is proven in the same
way.
\end{proof}

\begin{figure}
\begin{picture}(0,0)%
\includegraphics{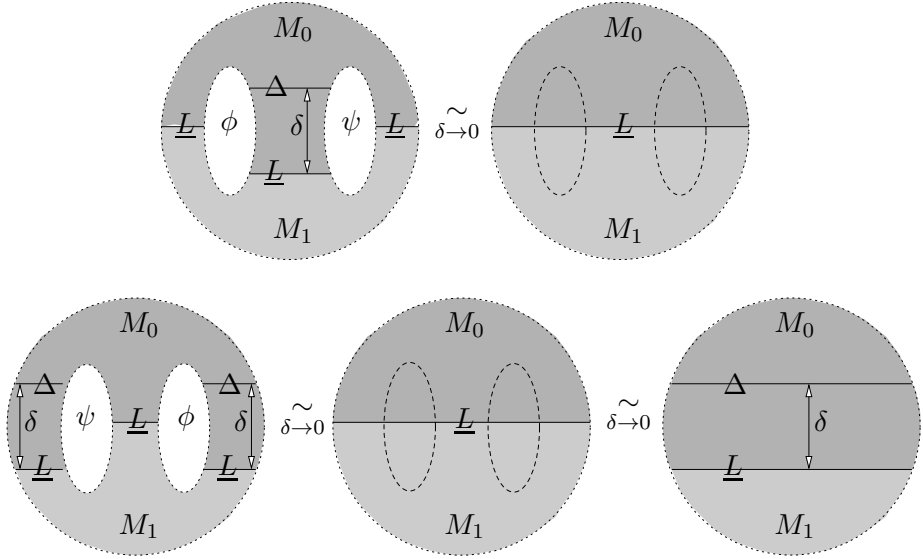}%
\end{picture}%
\setlength{\unitlength}{2362sp}%
\begingroup\makeatletter\ifx\SetFigFont\undefined%
\gdef\SetFigFont#1#2#3#4#5{%
  \reset@font\fontsize{#1}{#2pt}%
  \fontfamily{#3}\fontseries{#4}\fontshape{#5}%
  \selectfont}%
\fi\endgroup%
\begin{picture}(9598,5819)(-1177,-5423)
\put(3320,-1066){\makebox(0,0)[lb]{$\underset{\delta\to 0}{\sim}$}}
\put(1700,-4171){\makebox(0,0)[lb]{$\underset{\delta\to 0}{\sim}$}}
\put(5121,-4126){\makebox(0,0)[lb]{$\underset{\delta\to 0}{\sim}$}}
\put(5086,-16){\makebox(0,0)[lb]{$M_0$}}
\put(5086,-2131){\makebox(0,0)[lb]{$M_1$}}
\put(5200,-1041){\makebox(0,0)[lb]{$\ul{L}$}}
\put(1621,-16){\makebox(0,0)[lb]{$M_0$}}
\put(1800,-1006){\makebox(0,0)[lb]{$\delta$}}
\put(596,-1041){\makebox(0,0)[lb]{$\ul{L}$}}
\put(2800,-1041){\makebox(0,0)[lb]{$\ul{L}$}}
\put(1530,-1540){\makebox(0,0)[lb]{$\ul{L}$}}
\put(1531,-590){\makebox(0,0)[lb]{$\Delta$}}
\put(1621,-2131){\makebox(0,0)[lb]{$M_1$}}
\put(  1,-3121){\makebox(0,0)[lb]{$M_0$}}
\put(1041,-4641){\makebox(0,0)[lb]{$\ul{L}$}}
\put(1041,-3721){\makebox(0,0)[lb]{$\Delta$}}
\put(100,-4146){\makebox(0,0)[lb]{$\ul{L}$}}
\put(-900,-4641){\makebox(0,0)[lb]{$\ul{L}$}}
\put(-900,-3721){\makebox(0,0)[lb]{$\Delta$}}
\put(  1,-5236){\makebox(0,0)[lb]{$M_1$}}
\put(3421,-3121){\makebox(0,0)[lb]{$M_0$}}
\put(3421,-5236){\makebox(0,0)[lb]{$M_1$}}
\put(3531,-4146){\makebox(0,0)[lb]{$\ul{L}$}}
\put(6886,-3121){\makebox(0,0)[lb]{$M_0$}}
\put(7300,-4111){\makebox(0,0)[lb]{$\delta$}}
\put(6886,-5236){\makebox(0,0)[lb]{$M_1$}}
\put(1216,-4111){\makebox(0,0)[lb]{$\delta$}}
\put(-989,-4111){\makebox(0,0)[lb]{$\delta$}}
\put(6346,-3721){\makebox(0,0)[lb]{$\Delta$}}
\put(6346,-4641){\makebox(0,0)[lb]{$\ul{L}$}}
\put(1071,-1006){\makebox(0,0)[lb]{$\phi$}}
\put(-450,-4111){\makebox(0,0)[lb]{$\psi$}}
\put(2341,-1006){\makebox(0,0)[lb]{$\psi$}}
\put(621,-4111){\makebox(0,0)[lb]{$\phi$}}
\end{picture}%

\caption{Isomorphism of $\Delta_{M_0}\#\ul{L}$ and $\ul{L}$}
\label{moreshrinkpants}
\end{figure}

\begin{corollary} \label{main2cor} 
Under the assumptions of Theorem~\ref{main22} the functors
$\Psi_{M_0} \circ \Phi(L_{01} \circ L_{12})$,
$\Phi(L_{01} \circ L_{12})\circ \Psi_{M_2}$, and
$\Phi(L_{01}) \circ \Phi(L_{12})$ are all isomorphic 
in the category of functors from $\Don^\#(M_0)$ to $\Don^\#(M_2)$.
\end{corollary}

\begin{proof}
From Theorem \ref{main22} and \eqref{concateq} we obtain isomorphisms
between
$\Phi(\Delta_{M_0}\#L_{02})=\Phi(\Delta_{M_0})\circ\Phi(L_{02})$,
$\Phi(L_{02}\#\Delta_{M_2})=\Phi(L_{02})\circ\Phi(\Delta_{M_2})$, and
$\Phi(L_{01}\#L_{12})=\Phi(L_{01})\circ\Phi(L_{12})$.  By
Proposition~\ref{phidelta} the functors $\Phi(\Delta_{M_k})$ are
isomorphic to the shift functors $\Psi_{M_k}$.  Since isomorphisms
commute with composition of functors, this proves the corollary.
\end{proof}

\section{
$2$-category of monotone symplectic manifolds}
\label{2cat}

We can rephrase and summarize the constructions of the previous sections, 
using the language of $2$-categories.

\begin{definition}  A {\em $2$-category} $\cC$ consists of the following data:
\ben
\item  A class of objects $\Obj(\cC)$.
\item  For each pair of objects $X,Y \in \Obj(\cC)$, a small
category $\Hom(X,Y)$.  
\item  For each triple of objects $X,Y,Z \in \Obj(\cC)$, a composition 
functor 
$$ \circ: \ \Hom(X,Y) \times \Hom(Y,Z) \to \Hom(X,Z) .$$
\item For every $X \in \Obj(\cC)$ an identity functor $1_X \in
  \Hom(X,X)$.
\een
These data should satisfy the following axioms:
\smallskip

(Identity):\; For all $X,Y \in \Obj(\cC)$ and $f \in \Hom(X,Y)$
$$ 1_X \circ f = f, \ \ \ \ f \circ 1_Y = f .$$ 

(Associativity):\;
For all composable morphisms $f,g,h$ 
$$ f \circ (g \circ h) = (f \circ g) \circ h . $$
\end{definition}

Objects resp.\ morphisms in $\Hom(X,Y)$ are called $1$-morphisms
resp.\ $2$-morphisms.  We say that $\cC$ has {\em weak identities} if
equality in the identity axiom is replaced by $2$-isomorphism.

The basic example of a $2$-category is $\on{Cat}$, 
whose objects are small categories, $1$-morphisms are functors, 
and $2$-morphisms are natural transformations.  

\begin{definition}
A {\em $2$-functor} $\F:\cC_1\to\cC_2$ between $2$-categories
$\cC_1$ and $\cC_2$ consists of 
\ben 
\item a map $\F: \Obj(\cC_1) \to \Obj(\cC_2)$,
\item for each pair $X,Y \in \Obj(\cC_1)$, a functor 
$$\F(X,Y): \Hom(X,Y) \to \Hom(\F(X),\F(Y)) ,$$
\een
respecting composition and identities.
\end{definition}

In the following we restrict ourselves to symplectic manifolds that are spin,
i.e.\ ${w_2(M)=0}$. Their advantage is that the shift functor 
$\Psi_M:\Don^\#(M,b)\to\Don^\#(M,b)$ of Definition \ref{dfn shift} is trivial 
and the diagonal $\Delta_M\subset M^-\times M$ is an object of the
category of correspondences $\Don^\#(M,M)$ from $(M,b)$ to itself.
We moreover drop the Maslov cover from the data, thus working with 
ungraded Floer cohomology groups.

\begin{definition} \label{WF}
Fix a constant $\tau\geq 0$.
Let the {\em Weinstein-Floer $2$-category} $\Floer^\#_{\tau}$ 
be the category given as follows:
\begin{enumerate}
\item
Objects are symplectic manifolds $(M,\omega)$ that satisfy (M1-2) with
monotonicity constant $\tau$ and $w_2(M)=0$, and that are equipped
with a background class $b\in H^2(M,\Z_2)$. 
\item
The morphism categories of $\Floer^\#$ are the Donaldson categories of
Lagrangian correspondences, $\Hom(M_0,M_1):=\Don^\#(M_0,M_1)$; without grading.
\item Composition is defined by the functor \eqref{concat},
$$
\#: \Don^\#(M_0,M_1) \times \Don^\#(M_1,M_2) \to \Don^\#(M_0,M_2) .
$$
\item
The diagonal defines a weak identity $\Delta_M \in \Don^\#(M,M)$.
\end{enumerate}
\end{definition}

\begin{remark}
One could define $\Floer_\tau^\#$ by restricting to nonempty symplectic manifolds.
However, for future applications, we wish to include the empty set $\emptyset$ as object.
The only elementary Lagrangian correspondence from $\emptyset$ to $M$ is $L=\emptyset$, 
but in the sequence of a generalized Lagrangian correspondences, we must now allow any number of $\emptyset$ as symplectic manifolds as well as Lagrangian correspondences.
However, the Floer cohomology of any generalized Lagrangian correspondence containing 
$\emptyset$ is the trivial group $HF( \ldots\overset{\emptyset}{\longrightarrow}\ldots)=\{0\}$.
\end{remark}

The associativity axiom on $\Floer_\tau^\#$ is immediate on the level of objects: For any
triple $\ul{L}_{01}\in\Obj(\Don^\#(M_0,M_1))$,
$\ul{L}_{12}\in\Obj(\Don^\#(M_1,M_2))$,
$\ul{L}_{23}\in\Obj(\Don^\#(M_2,M_3))$ we have
$(\ul{L}_{01}\#\ul{L}_{12})\#\ul{L}_{23}=\ul{L}_{01}\#(\ul{L}_{12}\#\ul{L}_{23})$.
On the level of morphisms we apply the quilted gluing theorem \cite[Theorem 3.13]{quilts} to the gluings
indicated by dashed lines in Figure~\ref{ass2} to prove that
$(f\#g)\#h=f\#(g\#h)$ for all $f\in\Hom(\ul{L}_{01},\ul{L}_{01}')$,
$g\in\Hom(\ul{L}_{12},\ul{L}_{12}')$,
$h\in\Hom(\ul{L}_{23},\ul{L}_{23}')$.  The weak identity axiom follows
from Proposition \ref{diagonal}.  Hence $\Floer^\#$ is a $2$-category
with weak identities.

\begin{figure}
\begin{picture}(0,0)%
\includegraphics{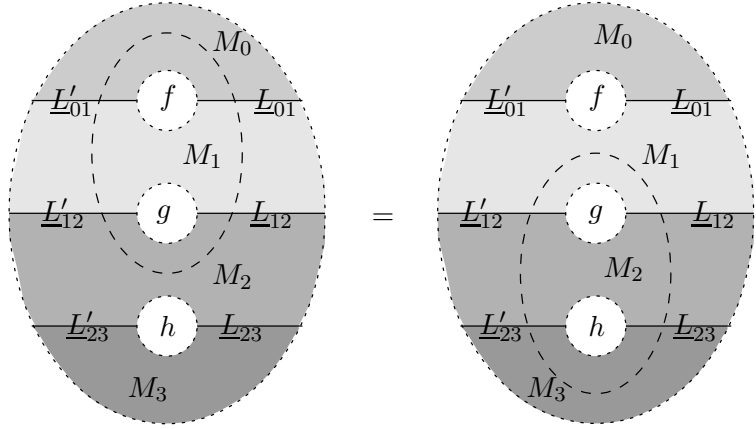}%
\end{picture}%
\setlength{\unitlength}{4144sp}%
\begingroup\makeatletter\ifx\SetFigFont\undefined%
\gdef\SetFigFont#1#2#3#4#5{%
  \reset@font\fontsize{#1}{#2pt}%
  \fontfamily{#3}\fontseries{#4}\fontshape{#5}%
  \selectfont}%
\fi\endgroup%
\begin{picture}(4479,2534)(-56,-2003)
\put(3421,-1471){\makebox(0,0)[lb]{$h$}}
\put(3931,-1501){\makebox(0,0)[lb]{$\ul{L}_{23}$}}
\put(4041,-826){\makebox(0,0)[lb]{$\ul{L}_{12}$}}
\put(3916,-151){\makebox(0,0)[lb]{$\ul{L}_{01}$}}
\put(2771,-1501){\makebox(0,0)[lb]{$\ul{L}_{23}'$}}
\put(3421,-121){\makebox(0,0)[lb]{$f$}}
\put(2801,-151){\makebox(0,0)[lb]{$\ul{L}_{01}'$}}
\put(3421,-800){\makebox(0,0)[lb]{$g$}}
\put(2656,-826){\makebox(0,0)[lb]{$\ul{L}_{12}'$}}
\put(856,-1471){\makebox(0,0)[lb]{$h$}}
\put(1216,-1501){\makebox(0,0)[lb]{$\ul{L}_{23}$}}
\put(1396,-826){\makebox(0,0)[lb]{$\ul{L}_{12}$}}
\put(1421,-151){\makebox(0,0)[lb]{$\ul{L}_{01}$}}
\put(1171,209){\makebox(0,0)[lb]{$M_0$}}
\put(666,-1861){\makebox(0,0)[lb]{$M_3$}}
\put(300,-1501){\makebox(0,0)[lb]{$\ul{L}_{23}'$}}
\put(1171,-1186){\makebox(0,0)[lb]{$M_2$}}
\put(991,-466){\makebox(0,0)[lb]{$M_1$}}
\put(856,-121){\makebox(0,0)[lb]{$f$}}
\put(200,-151){\makebox(0,0)[lb]{$\ul{L}_{01}'$}}
\put(836,-796){\makebox(0,0)[lb]{$g$}}
\put(150,-826){\makebox(0,0)[lb]{$\ul{L}_{12}'$}}
\put(2125,-800){\makebox(0,0)[lb]{$=$}}
\put(3451,254){\makebox(0,0)[lb]{$M_0$}}
\put(3736,-466){\makebox(0,0)[lb]{$M_1$}}
\put(3511,-1141){\makebox(0,0)[lb]{$M_2$}}
\put(3050,-1861){\makebox(0,0)[lb]{$M_3$}}
\end{picture}%
\caption{Associativity of the concatenation functor}
\label{ass2}
\end{figure}

\begin{remark} \label{Floer indep}
$\Floer^\#_\tau$ is independent up to $2$-isomorphism of
$2$-categories of the choices of perturbation data and strip widths,
as in Remarks \ref{Don indep}, \ref{Donsharp indep},
and the proofs of independence of
quilted Floer cohomology and relative quilt invariants in \cite{quiltfloer,quilts}.
\end{remark}

Theorem \ref{main22} implies that the definition of composition in the
Weinstein-Floer 2-category $\Floer^\#_\tau$ agrees with the geometric
definition, in the case that geometric composition is smooth,
embedded, and monotone.

\begin{theorem} \label{standard}
The map
$ M_0 \mapsto {\Don^\#}(M_0)$
and the functors
$${\Don^\#}(M_0,M_1) \to \Fun({\Don^\#}(M_0),{\Don^\#}(M_1)) $$ 
as in Proposition \ref{functor} define a 
{\em categorification $2$-functor} $\Floer_\tau^\# \to \Cat$
for every $\tau\ge 0$.
\end{theorem}

\begin{proof}
Compatibility with the composition follows from the identity \eqref{concateq}.
The weak identities $\Delta_M\in\Hom(M,M)$ are mapped to weak identities
$\Phi(\Delta)\simeq 1_{\Don^\#(M)}$ by Corollary \ref{phidelta}.
Here the shift functor $\Psi_M$ is the identity since $w_2(M)=0$.
\end{proof}

\begin{remark}
\ben
\item
For any genuinely monotone symplectic manifold (i.e.\ with $\tau>0$)
we can achieve $\tau=1$ by rescaling. It thus suffices to consider the
{\em exact Weinstein-Floer 2-category} $\Floer^\#_0$ and the {\em
monotone Weinstein-Floer 2-category} $\Floer^\#_1$.  Note however that
we cannot incorporate Lagrangian correspondences between monotone
symplectic manifolds with different monotonicity constants. This is
due to bubbling effects which in our present setup are true obstructions to the equivalence of algebraic composition $L_{01}\#L_{12}$ and embedded geometric composition $L_{01}\circ L_{12}$. 
We expect that the $A_\infty$-setup, incorporating all bubbling
effects, has better behavior.
\item
One can define an analogous {\em graded Weinstein-Floer 2-category}
$\Floer^\#_{N,\tau}$ for any $\tau\ge0$ and integer $N$, whose objects
are monotone symplectic manifolds with the additional structure of a
Maslov cover $\Lag^N(M)\to M$.  Its $1$-morphisms are graded
generalized Lagrangian correspondences, and its $2$-morphism spaces
are the graded Floer cohomology groups.
\een
\end{remark}

\begin{remark}
\ben
\item One can define a strong identity $1_M \in \Hom(M,M)$ by
allowing the empty sequence $1_M := \emptyset$ as a generalized
Lagrangian correspondence. The various constructions in this Section
extend to the case of empty sequences by allowing cylindrical ends.
\item
In the case $w_2(M) \ne 0$, the diagonal is not an automorphism but a
morphism $\Delta_M\in\Hom((M,b),(M,b-w_2(M)))$, see Remark \ref{rmk
diag}.  Hence
$$\ul{L}\#\Delta_M\in\Hom((M_1,b_1),(M,b-w_2(M))), \ \ \
\ul{L}\in\Hom((M_1,b_1),(M,b))$$
lie in different morphism spaces that are not related by a simple
shift in the background class.  However, the categorification functor in
Theorem \ref{standard} generalizes directly to this setup as follows. 
The functor maps the special $\Floer^\#_\tau$ $1$-morphisms
$\Delta_M\in\Don^\#((M,b),(M,b-w_2(M))$ to $\Cat$ $1$-morphisms that
are isomorphic to the shift functors
$\Psi_M\in\Fun(\Don^\#(M,b),\Don^\#(M,b-w_2(M)))$.  
\item One can make the diagonal a strong identity by modding out by
the equivalence relation discussed Section~\ref{symp cat}. 
Let $\Brane_\tau^\#$ denote the $2$-category whose objects and
$1$-morphisms are those of $\Floer_\tau^\#$, modulo the equivalence
relation $L_{01} \# L_{12} \sim L_{01} \circ L_{12}$ for embedded
compositions as in Section \ref{symp cat},
 and whose $2$-morphisms
are defined as follows.  Given a pair $[\ul{L}_{01}],[\ul{L}_{01}']$
of $1$-morphisms from $M_0$ to $M_1$, define the space of
$2$-morphisms $\Hom([\ul{L}_{01}],[\ul{L}_{01}'])$ by
$ \Hom([\ul{L}_{01}],[\ul{L}_{01}']) = HF(\ul{L}_{01}, \ul{L}_{01}') $
for some choice of representatives $\ul{L}_{01}, \ul{L}_{01}'$.
Define composition by concatenation $\#$, as in \eqref{concat}.  The
equivalence classes of the diagonal $[\Delta_M]$ define true
identities in case $w_2(M)=0$.  Our main result, Theorem \ref{main2},
implies that $\Brane_\tau^\#$ is independent of the choice of
representatives up to $2$-isomorphism of $2$-categories.  Theorem
\ref{main22} implies that the categorification $2$-functor of Theorem
\ref{standard} induces a $2$-functor $\Brane_\tau^\# \to \on{Cat}$ to
the $2$-category of categories $\on{Cat}$.
\een
\end{remark}

\def\cprime{$'$} \def\cprime{$'$} \def\cprime{$'$} \def\cprime{$'$}
  \def\cprime{$'$} \def\cprime{$'$}
  \def\polhk#1{\setbox0=\hbox{#1}{\ooalign{\hidewidth
  \lower1.5ex\hbox{`}\hidewidth\crcr\unhbox0}}} \def\cprime{$'$}
  \def\cprime{$'$}

\end{document}